\documentclass[12pt,UTF-8,reqno]{amsart}
\usepackage{geometry}
\usepackage{amsfonts}
\usepackage{mathrsfs}
\usepackage[greek,english]{babel}
\usepackage{amsmath,amsfonts,amsthm,amssymb,mathtools,amscd,color,xcolor,mathrsfs,verbatim,microtype}
\usepackage{graphicx,eurosym}
\usepackage{amssymb,amsmath,comment}
\usepackage{amsfonts,mathrsfs}
\usepackage{color}
\usepackage{bbm}
\usepackage{bm}
\usepackage{hyperref}
\usepackage{stmaryrd} 
\usepackage{cleveref}

\usepackage[cyr]{aeguill}

\usepackage{enumerate, bbm}

\colorlet{darkblue}{blue!50!black}

\hypersetup{
	colorlinks,%
	citecolor=darkblue,%
	filecolor=darkred,%
	linkcolor=blue,%
	urlcolor=blue,%
	pdfnewwindow=true,%
	pdfstartview={FitH}
}

\usepackage{graphicx,amscd,mathrsfs,wrapfig,mathrsfs,lipsum}
\usepackage{eufrak}
\usepackage{float}
\usepackage{tikz}
\usepackage{multicol}
\usepackage{caption}
\usetikzlibrary{arrows}
\usepackage{capt-of}

\colorlet{darkblue}{blue!50!black}

\newcommand{\vertiii}[1]{{\left\vert\kern-0.25ex\left\vert\kern-0.25ex\left\vert #1
		\right\vert\kern-0.25ex\right\vert\kern-0.25ex\right\vert}}

\usepackage{marginnote}

\newcommand{\bae}{\begin{equation}\begin{aligned}}
		\newcommand{\eae}{\end{aligned}\end{equation}}
\newcommand{\baee}{\begin{equation*}\begin{aligned}}
		\newcommand{\eaee}{\end{aligned}\end{equation*}}

\theoremstyle{plain}

\newtheorem*{lemma*}{Lemma}
\newtheorem{theorem}{Theorem}[section]
\newtheorem{lemma}[theorem]{Lemma}
\newtheorem{proposition}[theorem]{Proposition}
\newtheorem{corollary}[theorem]{Corollary}
\theoremstyle{definition}
\newtheorem{definition}[theorem]{Definition}
\newtheorem{condition}[theorem]{Condition}

\theoremstyle{remark}

\newtheorem{remark}{Remark}[section]

\numberwithin{equation}{section}

\newcommand{\ea}{\end{array}}

\newcommand{\worknote}[1]{}

\allowdisplaybreaks

\begin{document}
\title[Mean-Field Stochastic PDEs]
{Mean-Field Stochastic PDEs: Well-posedness and Quantitative Dimension-Free Propagation of Chaos$^\dagger$}
\thanks{$\dagger$
	This work is supported by National Key R\&D program of China (No.~2023YFA1010101).  W. Hong is supported by  NSFC (No.~12401177) and Basic
	Research Program of Jiangsu (No.~BK20241048).  W. Liu is supported by NSFC (No.~12571155). }

\maketitle
\centerline{ Wei Hong$^a$,  Wei Liu$^{a,}$\footnote{Corresponding author: weiliu@jsnu.edu.cn}, Luhan Yang$^{b}$ }

\vspace{2mm}
\medskip
{\footnotesize
	\centerline{ $a.$   School of Mathematics and Statistics, Jiangsu Normal University, Xuzhou 221116, China}
	
	\vspace{1mm}
	\centerline{ $b.$   School of Statistics and Data Science, Nankai University, Tianjin 300071, China}}

\begin{abstract}
This work investigates the mean-field stochastic PDEs involving a class of pseudo-monotone kernels. We first study the well-posedness--in both the strong and weak sense--within the variational framework by introducing a notion of measure-dependent pseudo-monotone operators, which generalizes  the classical framework due to  Br\'{e}zis. Furthermore, we establish   the quantitative dimension-free propagation of chaos within a $p$-uniformly convex Banach space for general infinite-dimensional weakly interacting systems,  obtaining convergence rates that are near-optimal in a suitable sense.

Our results reveal a new insight: the convergence rate of the mean-field limit is intrinsically governed by the geometry of the underlying solution space, specifically its modulus of convexity.
As applications, we study several finite- and infinite-dimensional interacting particle systems arising in machine learning and fluid mechanics, including stochastic Stein variational gradient descent, mean-field Allen-Cahn equations, and Lagrangian-averaged Burgers equations.
	
	\bigskip
	\noindent
	\textbf{Keywords}:  Mean-field system; Stochastic PDEs; Variational framework; Propagation of chaos; $p$-uniform convexity; Convergence rates.
	\\
	\textbf{Mathematics Subject Classification (2020)}: {60H10,~60H15,~60K35}
\end{abstract}

\tableofcontents

\section{Introduction}
Interacting particle systems provide a fundamental framework for modeling collective dynamics in statistical physics, machine learning, and game theory (cf.~e.g.~\cite{BS20,CD18,LW1,RV18}).
To circumvent the curse of dimensionality inherent in high-dimensional joint dynamics, the mean-field approximation replaces intricate  pairwise interactions with an averaged coupling to the empirical measure.
In the infinite-system limit, this approximation gives rise to the propagation of chaos (PoC), originally formulated by Kac in  kinetic theory \cite{Kac}, whereby particles become asymptotically independent. The limiting dynamics are governed by a  class of mean-field  (or McKean-Vlasov) stochastic differential equations (SDEs),   first introduced by McKean \cite{Mckean}, whose coefficients depend on the time-marginal law of the solution process itself. Nowadays, this class of SDEs has become a powerful theoretical tool for understanding the large-scale stochastic systems.

More formally, an $N$-particle interacting system  is described by the coupled SDEs
\begin{equation*}\label{SDE1}
    dX_t^{i,N} = b(X_t^{i,N}, \mu_t^N)dt + \sigma( X_t^{i,N},\mu_t^N)dW_t^i, \quad i=1, \dots, N,
\end{equation*}
where $\{W^i_t\}_{i=1}^N$ are independent standard Brownian motions and the empirical measure is defined as
$$\mu_t^N:= N^{-1}\sum_{j=1}^N \delta_{X_t^{j,N}}.$$
The PoC  asserts that, as $N\to\infty$, the empirical measures $\mu_t^N$ converge in law to the deterministic flow $\mathscr{L}_{X_t}$, where $X$ solves the following mean-field SDE
\begin{equation*}\label{SDE2}
    dX_t = b(X_t, \mu_t)dt + \sigma( X_t,\mu_t)dW_t,\quad \mu_t=\mathscr{L}_{X_t}.
\end{equation*}
This provides a rigorous foundation for the mean-field approximation: in the limit, particles become asymptotically independent and identically distributed.
When $b$ and $\sigma$ are globally Lipschitz in both the state and the measure (e.g.~in the $2$-Wasserstein metric), McKean \cite{Mc67} and Sznitman \cite{S} firstly established the following convergence result
$$\mathbb{E}\Big[\sup_{t\in[0,T]}|X_t^{1,N}-X_t|^2\Big]\lesssim N^{-1},$$
which implies an $L^2$-strong error of order $N^{-1/2}$.     This rate is sharp in the strong sense.
Subsequent work has significantly extended this theory to non-Lipschitz settings, which are crucial for practical applications in mathematical physics, machine learning, and optimization algorithms (cf.~e.g.~\cite{DT18,DNS,JW,S2020,V}).
For comprehensive surveys on  mean-field SDEs, we refer to \cite{BR24,BLPR,CD22,HSS,WR25} and references therein.

One application of the present work is motivated by interacting particle systems arising in Bayesian inference. For high-dimensional or large-scale posterior approximation, classical Markov Chain Monte Carlo (MCMC) methods \cite{H-MCMC} are asymptotically exact but scale poorly due to slow mixing. Conversely, Variational Inference (VI) \cite{BKM-VI} offers computational spectral gains by framing inference as optimization, but sacrifices asymptotic consistency. This trade-off led to interacting particle methods as a scalable and consistent alternative.
Particle-based methods address this problem by evolving   interacting particles  toward the target posterior. One important  example is Stein Variational Gradient Descent (SVGD), introduced in \cite{LW1}, which moves a set of particles $\{X^{i,N}\}_{i=1}^N$ to approximate a target distribution $\pi$.
SVGD updates the particles along the steepest descent direction of the Kullback--Leibler (KL) divergence. Given a step size $\varepsilon>0$, each particle is transported by
\[
x \mapsto x+\varepsilon \psi(x),
\]
where $\psi$ is a vector field determined by a user-chosen kernel $\kappa(x,y)$. When the target distribution is of the form
\[
\pi(x) \propto \exp(-\mathcal{V}(x)),
\]
the discrete SVGD update is given by
\begin{equation*}
	X_{n+1}^{i,N}
	=
	X_{n}^{i,N}
	-\frac{\varepsilon}{N}\sum_{j=1}^N \nabla_y \kappa(X_{n}^{i,N}, X_{n}^{j,N})
	-\frac{\varepsilon}{N}\sum_{j=1}^N \kappa(X_{n}^{i,N}, X_{n}^{j,N}) \nabla \mathcal{V}(X_{n}^{j,N}).
\end{equation*}
This interacting particle structure allows SVGD to approximate posterior distributions efficiently, while avoiding some limitations of traditional MCMC and variational inference methods.

Recently, Zhang et al.~\cite{ZZCC} pointed out a theoretical limitation of SVGD: under certain conditions, the particles may collapse to a local mode.  Motivated by the unified particle-optimization framework in \cite{CZWLC}, they proposed a stochastic variant of SVGD, where the deterministic transport field is augmented by a diffusion term.  Intuitively, the injected noise enhances the ability of the particles to escape local modes, thereby improving exploration compared with standard SVGD.  With Gaussian noise added, the particle update rule becomes
\begin{eqnarray*}
X_{n+1}^{i,N} =\!\!\!\!\!\!\!\!&&X_{n}^{i,N}-\varepsilon\lambda^{-1}\nabla \mathcal{V}(X_{n}^{i,N})-\frac{\varepsilon}{N}\sum_{j=1}^N \nabla_y \kappa(X_{n}^{i,N}, X_{n}^{j,N})\\
\!\!\!\!\!\!\!\!&&- \frac{\varepsilon}{N}\sum_{j=1}^N \kappa(X_{n}^{i,N}, X_{n}^{j,N}) \nabla \mathcal{V}(X_{n}^{j,N})+\sqrt{2 \lambda^{-1}\varepsilon}\xi_n^{i}.
\end{eqnarray*}
where $\lambda>0$ is the temperature parameter and $\xi_n^{i} \sim \mathcal{N}(0, I_d)$. In the continuous-time limit, this leads to SDEs
\begin{eqnarray}\label{SSVGD}
dX_{t}^{i,N} =\!\!\!\!\!\!\!\!&&-\lambda^{-1}\nabla \mathcal{V}(X_t^{i,N})dt-\frac{1}{N}\sum_{j=1}^N \nabla_y \kappa(X_t^{i,N}, X_t^{j,N}) dt\nonumber\\
\!\!\!\!\!\!\!\!&&- \frac{1}{N}\sum_{j=1}^N \kappa(X_t^{i,N}, X_t^{j,N}) \nabla \mathcal{V}(X_t^{j,N})dt+\sqrt{2 \lambda^{-1}}dW_t^{i},
\end{eqnarray}
where $W_t^i$, $i = 1, \dots, N$, are independent $\mathbb{R}$-valued Brownian motions.
Inspired by \cite[Remark 46]{DNS}, we consider the kernel
$$
\kappa(x,y) := x y,
\qquad
\mathcal{V}(y) := \frac{1}{2}y^{2}.
$$
In fact, Liu and Wang \cite{LW2} observed that super-linear kernels of the form $\kappa(x,y)=xy$ can handle cases in which the target distribution is asymptotically Gaussian, provided that the number of particles exceeds the dimension.
We also refer interested readers to \cite{DNS,GI} and the references therein for further studies on stochastic SVGD.

Notably, the interaction kernels in (\ref{SSVGD}) exhibit super-linear growth and fail to satisfy standard global Lipschitz continuity or monotonicity conditions widely adopted in the literature (e.g.~\cite{BCC,S,W18} and references therein). Dealing with such super-linear interaction kernels is a longstanding challenge (cf.~\cite{BCC,E,NTT}). As remarked in Section 3.1.2 in the recent review \cite{CD22}:

\noindent``{\it There is no real hope for better results at this level of generality, many directions have been explored to weaken the hypotheses in specific cases}.''

\vspace{1mm}
\noindent Motivated by the (stochastic) SVGD in Bayesian inference, this paper aims to establish a unified mathematical framework for systems with super-linear interaction kernels, encompassing both finite- and infinite-dimensional settings. Moreover, beyond proving the well-posedness of the  mean-field limit equations and  the convergence of $N$-particle interacting systems, we establish the quantitative dimension-free   PoC and obtain near-optimal rates, which is applicable to examples including the stochastic SVGD (\ref{SSVGD}).

\subsection{Well-posedness}\label{subCon1}
To establish  a general framework for studying the following mean-field stochastic dynamics in infinite dimensions
\begin{equation}\label{eqSPDE1}
	 dX_t=\mathcal{A}(t,X_t,\mathscr{L}_{X_t})dt+\mathcal{B}(t,X_t,\mathscr{L}_{X_t})dW_t,
\end{equation}
where $\{W_t\}_{t\geq 0}$ is a cylindrical Wiener process defined on a complete filtered probability space $\left(\Omega,\mathscr{F},\{\mathscr{F}_t\}_{t\geq 0},\mathbb{P}\right)$,
 we adopt the generalized variational approach in this work.

More precisely, let $(U,\langle\cdot,\cdot\rangle_U)$ and $({\mathbb{H}}, \langle\cdot,\cdot\rangle_{\mathbb{H}}) $ be  separable Hilbert spaces, and ${\mathbb{H}}^*$ be the dual space of ${\mathbb{H}}$. Let $({\mathbb{V}},\|\cdot\|_{\mathbb{V}})$ denote a reflexive Banach space  such that
$$\text{the embedding}~~\mathbb{V}\subset \mathbb{H}~~\text{is compact and dense}.$$
Identifying $\mathbb{H}$ with its dual space in terms of the Riesz isomorphism, we obtain a classical Gelfand triple
\begin{equation}\label{gelfand}
 {\mathbb{V}}\subset {\mathbb{H}}(\simeq {\mathbb{H}}^*)\subset {\mathbb{V}}^*.
 \end{equation}
The dualization between ${\mathbb{V}}$ and ${\mathbb{V}}^*$ is denoted by $_{{\mathbb{V}}^*}\langle\cdot,\cdot\rangle_{\mathbb{V}}$. Then we have
$$_{{\mathbb{V}}^*}\langle\cdot,\cdot\rangle_{\mathbb{V}}|_{{\mathbb{H}}\times {\mathbb{V}}}=\langle\cdot,\cdot\rangle_{\mathbb{H}}.$$
Let $L_2(U;\mathbb{H})$ be  the space consisting of all Hilbert-Schmidt operators from $U$ to $\mathbb{H}$, which is equipped with the standard Hilbert-Schmidt norm $\|\cdot\|_{L_2(U;\mathbb{H})}$.

The classical variational approach to stochastic partial differential equations (PDEs) was pioneered by Pardoux, Krylov, and Rozovskii (cf.~\cite{KR,P,RL}), who utilized the well-known  monotonicity trick to establish the existence and uniqueness of solutions under standard monotonicity and coercivity conditions. To handle more complex nonlinear systems arising in practical applications, Liu and R\"{o}ckner \cite{liu2015stochastic} subsequently introduced the framework of local monotonicity for stochastic PDEs. More recently, R\"{o}ckner et al. \cite{RSZ} established the well-posedness of  stochastic PDEs driven by multiplicative noise with fully locally monotone operators. By combining pseudo-monotonicity techniques with stochastic compactness arguments, their work unifies and extends existing frameworks, making them applicable to a wider class of nonlinear stochastic PDEs.
We also refer the reader to \cite{AV24,HLL21,HLL24,HLL25,NS,NTT,W15} for recent progress on the  variational approach to stochastic PDEs.

In this work, we introduce the notion of  measure-dependent pseudo-monotone operators adapted for mean-field stochastic dynamics,  which generalized  the classical framework due to  Br\'{e}zis \cite{B68} and  are applicable to a broad class of finite- and infinite-dimensional mean-field models with super-linear kernels. Unlike classical stochastic PDEs, establishing the well-posedness of mean-field stochastic PDEs is inherently more challenging due to the intricate interplay between highly nonlinear operators and the nonlocal effects introduced by measure-dependent coefficients. For instance, if we consider  the stopped process $Y_t:= X_{t\wedge\tau}$ associated with (\ref{eqSPDE1}) for some stopping time $\tau$, it satisfies
$$Y_t=Y_0+\int_0^{t\wedge\tau}\mathcal{A}(s,Y_s,\mathscr{L}_{X_s})ds
+\int_0^{t\wedge\tau}\mathcal{B}(s,Y_s,\mathscr{L}_{X_s})dW_s,$$
where, crucially, the law of the original process (i.e.,~$\mathscr{L}_{X_s}$)--rather than that of the stopped process $Y_s$--remains in the coefficients, thereby causing standard localization arguments to fail. This limitation is highlighted  by Ertel and Stannat  \cite{ES24}:

\noindent``{\it So far there does not seem to exist a localization argument that is suited to prove the well-posedness of a general class of McKean-Vlasov equations under local Lipschitz assumptions}.''

\noindent To resolve this issue, we will propose a refined formulation of the local monotonicity condition.

We now briefly summarize the well-posedness results established in this work. First, by developing  a measure cut-off argument, we establish the existence of weak solutions to  mean-field SDEs whose drift and diffusion coefficients  exhibit
super-linear growth (see Theorem \ref{finite-th1}). Notably, this result is of independent interest within the finite-dimensional setting.
Subsequently,  by leveraging the framework of  measure-dependent pseudo-monotone operators, we establish the existence of weak solutions to Eq.~(\ref{eqSPDE1})  via the Faedo-Galerkin approximation and stochastic compactness method (as presented in Theorem \ref{th1}).
To address the existence of strong solutions  to Eq.~(\ref{eqSPDE1}), we introduce a fully local monotonicity condition (see condition $(\mathbf{A}_5)$ in Subsection \ref{Strongresult}) with respect to both the state and measure variables, i.e.,
\begin{eqnarray}\label{unip1}
		\!\!\!\!\!\!\!\!&&2_{{\mathbb{V}}^*}\langle \mathcal{A}(t,u,\mu)-\mathcal{A}(t,v,\nu),u-v\rangle_{\mathbb{V}}+\|\mathcal{B}(t,u,\mu)-\mathcal{B}(t,v,\nu)\|_{L_2(U;\mathbb{H})}^2
		\nonumber\\
		\leq\!\!\!\!\!\!\!\!&&
		(\rho(u,\mu)+\eta(v,\nu))(\|u-v\|_{\mathbb{H}}^2+\mathcal{W}_{2,{\mathbb{H}}}(\mu,\nu)^2).
	\end{eqnarray}
Under suitable additional assumptions, we verify that the operators satisfying this condition (\ref{unip1}) are indeed measure-dependent pseudo-monotone. We then establish the existence of strong solutions by proving the uniqueness of the corresponding decoupled stochastic PDEs (see Theorem \ref{th2}).

However,  in contrast to the classical variational theory for stochastic PDEs,  the fully local monotonicity condition  in the mean-field setting fails to guarantee the uniqueness of either strong or weak solutions; we refer to the paper \cite{S87} by Scheutzow for relevant counterexamples. To overcome this difficulty, we propose a {\it decoupled} local monotonicity condition (see $(\mathbf{A}^*_5)$ in Section \ref{sec4}):
\begin{eqnarray*}
		\!\!\!\!\!\!\!\!&&2{}_{{\mathbb{V}}^*}\langle \mathcal{A}(t,u,\mu)-\mathcal{A}(t,v,\nu),u-v\rangle_{\mathbb{V}}+\|\mathcal{B}(t,u,\mu)-\mathcal{B}(t,v,\nu)\|_{L_2(U;\mathbb{H})}^2
		\nonumber\\
		\leq\!\!\!\!\!\!\!\!&&
		(\rho(0,\mu)+\eta(0,\nu))\|u-v\|_{\mathbb{H}}^2
		+(\rho(u,\mu)+\eta(v,\nu))\mathcal{W}_{2,\mathbb{H}}(\mu,\nu)^2.
	\end{eqnarray*}
Under this  condition, we first establish the uniqueness of strong solutions to the equation, and then derive the existence and uniqueness of both weak and strong solutions, as well as their continuous dependence on the initial data (as illustrated in Theorem \ref{th3}).

We now compare the well-posedness  results  with related literature  in both finite and infinite-dimensional settings. For the mean-field SDEs, while Bolley et al.~\cite{BCC} extended McKean's classical theory to locally Lipschitz interactions, their framework relies on several restrictive assumptions, such as bounded exponential moments of the solutions and linear growth of the interaction kernels. Similarly, a recent work by Erny \cite{E} established the well-posedness and mean-field limit for mean-field SDEs under local Lipschitz conditions. However, their approach still requires linear growth of the drift, bounded diffusion coefficients, and exponential integrability of the initial values.

In contrast, our theory drops exponential moment requirements and linear growth constraints on the interaction kernels. Instead, we generalize the local Lipschitz condition by employing pseudo-monotonicity and local monotonicity frameworks. Furthermore, the interaction kernels in our diffusion coefficients are also  of super-linear growth. Consequently, our results  cover a wider range of mean-field dynamics (e.g., stochastic SVGD (\ref{SSVGD})) that cannot be analyzed using the methods in \cite{BCC,E}.
For infinite-dimensional systems, the presence of nonlinear kernels introduces highly non-trivial challenges. Therefore,  there are only very few works on general mean-field stochastic PDEs, which were  limited to settings with linear growth and Lipschitz-type conditions (cf.~\cite{BKKX,CKS,C23,ES,KX}). Nevertheless, several key infinite-dimensional mean-field models (cf.~e.g.~\cite{Drivas2020,GRZ,SSZZ}) arising in stochastic quantization and fluid mechanics inherently exhibit superlinear kernels.  To address this limitation, we develop a measure-dependent variational framework for mean-field stochastic PDEs with coefficients satisfying pseudo-monotonicity or local monotonicity conditions, thereby presenting a unified result for such infinite-dimensional mean-field dynamics.

\subsection{Propagation of chaos}
PoC theory characterizes the behavior of high-dimensional stochastic particle system by
approximating the collective influence of the ensemble on a representative particle as a single averaged effect. Recently, PoC  has been  investigated within broader frameworks and applied to areas such as  data science \cite{D}, mean-field games \cite{CD18}, and neural network training \cite{RV}. A comprehensive overview of the mathematical models, methods, and applications of PoC can be found in the survey by Chaintron and Diez \cite{CD22}.

The convergence of empirical measures is central to the analysis of PoC. In high-dimensional settings, however, this convergence typically suffers from the curse of dimensionality (CoD), whereby the convergence rate decreases as the dimension of the state space increases. For instance, under classical metrics such as Wasserstein distances, the empirical measure of $N$ particles in $\mathbb{R}^d$ converges at a rate of order $N^{-\frac{c_0}{d}}$ for some constant $c_0>0$ independent of $d$ (cf.~\cite{D69,FG,WB}). This CoD phenomenon--fundamentally tied to the choice of the Wasserstein metric and the nature of the underlying interaction kernels--remains a major challenge in high-dimensional analysis and algorithm design.

Accordingly,  the CoD phenomenon becomes a fundamental challenge in studying the PoC for mean-field stochastic PDEs, where the  state space is inherently infinite-dimensional.  Establishing quantitative PoC for such systems requires substantial improvements that go beyond adaptation of existing finite-dimensional results. In particular, it is crucial to leverage the specific structures of the underlying nonlinear operators and solution spaces to overcome the infinite-dimensional nature of the system. Therefore, this work aims to develop quantitative estimates  that remain stable under finite-dimensional approximations, thereby yielding the quantitative dimension-free PoC.

Building on the well-posedness theory for mean-field stochastic PDEs developed in Subsection \ref{subCon1}, we  consider the associated non-interacting particle system
\begin{equation}\label{nIPS1}
	X^i_t=\xi^i+\int_0^t\mathcal{A}(s,X^i_s,\mathscr{L}_{X^i_s})ds
	+\int_0^t\mathcal{B}(s,X^i_s,\mathscr{L}_{X^i_s})dW^i_s,
\end{equation}
with $i \in \{1, \dots, N\}$, where $\{\xi^i\}_{i=1}^{N}$ and $\{W^i\}_{i=1}^{N}$ are $N$ independent copies of the initial value and cylindrical Wiener process associated with Eq.~(\ref{eqSPDE1}), respectively.
We then turn to the corresponding interacting particle system
\begin{equation}\label{IPS1}
	X_t^{i,N} = \xi^i + \int_0^t \mathcal{A}(s,X^{i,N}_s, \mu_s^{N}) ds + \int_0^t \mathcal{B}(s,X^{i,N}_s, \mu_s^{N}) dW_s^i,
\end{equation}
where each particle evolves under the empirical measure $\mu^{N}$ and is driven by the same family of cylindrical Wiener processes and initial conditions as in the non-interacting system (\ref{nIPS1}).

In this work, we combine Sznitman's  synchronous coupling method with a delicate analysis of martingale difference sequences taking values in the dual of an $\alpha$-uniformly convex Banach space, in order to address the infinite-dimensional interacting particle dynamics. Moreover,  under the following strengthened decoupled local monotonicity condition (see $(\mathbf{A}'_5)$ in Subsection \ref{poc}),
\begin{equation}\label{ass01}
		_{{\mathbb{V}}^*}\langle \mathcal{A}(t,u,\mu)-\mathcal{A}(t,v,\nu),u-v\rangle_{\mathbb{V}}
		\leq
		\rho(0,\mu)\|u-v\|_{\mathbb{H}}^2
		+\rho(u,\mu)\mathcal{W}_{2,\mathbb{H}}(\mu,\nu)^2,
	\end{equation}
we employ a refined stopping time argument to study the strong convergence rate from Eq.~(\ref{IPS1}) to Eq.~(\ref{nIPS1}).

We first state an informal version of our quantitative dimension-free PoC result; the precise statement is given in Theorem \ref{rate1}.
\begin{theorem}\label{pocth}
Let $\mathbb{V}$ be a separable $\alpha$-uniformly convex Banach space with $\alpha\geq2$. Under assumption $(\ref{ass01})$, together with other suitable assumptions on the coefficients and initial data,  the following pathwise chaos  holds
	\begin{equation}\label{pocth11}
		\sup_{i \in \{1,\dots,N\}}  \mathbb{E}\Big[\sup_{t \in [0,T]}\|X_t^{i,N} - X_t^i\|_{\mathbb{H}}^2\Big] \lesssim N^{-\frac{1}{\alpha}}.
	\end{equation}
\end{theorem}

\noindent This theorem  shows  that the convergence rate of PoC for infinite-dimensional interacting particle systems is  intrinsically governed by the geometry of the underlying solution space, and in particular by its modulus of convexity. Indeed, estimate (\ref{pocth11}) yields a pathwise chaos rate of order $N^{-\frac{1}{2\alpha}}$, where $\alpha$ quantifies the degree of convexity of the space.
 In particular, this result gives, at best, a convergence rate of order $N^{-\frac{1}{4}}$
in $2$-uniformly convex Banach spaces, with Euclidean spaces being the prototypical example. In addition, we also provide a motivating example involving i.i.d. random variables to confirm such dependence, see Remark \ref{remark001} below in details.

To further improve the PoC estimates,  we take into account the dissipation induced by the underlying second-order differential operators and impose the following decoupled local monotonicity condition (see  $(\mathbf{A}''_5)$ in Subsection \ref{poc}),
\begin{eqnarray}\label{ass02}
		\!\!\!\!\!\!\!\!&&_{{\mathbb{V}}^*}\langle \mathcal{A}(t,u,\mu)-\mathcal{A}(t,v,\nu),u-v\rangle_{\mathbb{V}}+\delta_0\|u-v\|_{\mathbb{V}}^{\alpha}
		\nonumber\\
		\leq\!\!\!\!\!\!\!\!&&
		(\rho(0,\mu)+\eta(0,\nu))\|u-v\|_{\mathbb{H}}^2
		+(\rho(u,\mu)+\eta(v,\nu))\mathcal{W}_{2,\mathbb{H}}(\mu,\nu)^2
	\end{eqnarray}

We now state our second quantitative dimension-free PoC result; its rigorous formulation is given in  Theorem \ref{rate2}.
\begin{theorem}\label{pocth2}
Let $\mathbb{V}$ be a separable $\alpha$-uniformly convex Banach space with $\alpha\geq2$. Under  assumption $(\ref{ass02})$, together with other suitable assumptions on the coefficients and initial data, the following pathwise and pointwise chaos hold, respectively,	
	\begin{equation}\label{pocth31}
		\sup_{i \in \{1,\dots,N\}}\bigg\{\mathbb{E}\Big[\sup_{t \in [0,T]}\|X_t^{i,N} - X_t^i\|_{\mathbb{H}}^2\Big]+ \mathbb{E}\int_0^{T}\|X_s^i - X_s^{i,N}\|_{\mathbb{V}}^{\alpha}ds\bigg\} \lesssim N^{-\frac{1}{\alpha-1}\cdot\frac{q-\beta}{q-2}\cdot {\frac{p'-2}{p'}}}
	\end{equation}
and
	\begin{equation}\label{pocth41}
		\sup_{i \in \{1,\dots,N\}}  \sup_{t \in [0,T]}\mathbb{E}\|X_t^{i,N} - X_t^i\|_{\mathbb{H}}^2 \lesssim N^{-\frac{1}{\alpha-1}\cdot {\frac{p-\beta}{p}}}.
	\end{equation}
\end{theorem}
\noindent The pathwise estimate (\ref{pocth31}) and the pointwise estimate (\ref{pocth41}) improve upon (\ref{pocth11}). In particular, when  $\mathbb{V}$ is a $2$-uniformly convex space and constants $q,p'$ are sufficiently large,      these estimates yield near-optimal convergence rates of order $N^{-\frac{1}{2}+}$.

Concerning the mean-field limit in infinite dimensions, the  pioneering works of Chiang et al.~\cite{CKS} and Kallianpur et al.~\cite{KX} established the  qualitative PoC for interacting particle systems  in the dual of a countably Hilbertian nuclear space, under standard one-side Lipschitz assumptions on the coefficients. This framework was subsequently generalized to non-nuclear spaces  in \cite{BKKX}. Notably, the results in \cite{CKS,KX} have found important applications in neurophysiology, capturing the essential dynamical features of neural networks.
A dynamical model for polymer systems governed by $N$-coupled stochastic PDEs was  analyzed in \cite{ES}.
More recently, Criens \cite{C23} provided a systematic treatment for a class of weakly interacting semilinear stochastic PDEs, establishing their qualitative PoC under linear growth and Lipschitz-type conditions. Furthermore, significant progress has been made by Shen et al. \cite{SSZZ} on  $O(N)$ linear sigma
 models, which arise from stochastic quantization and generalize the classical $\Phi^4_d$ model to  $N$-component; in this work, the large $N$ limit  was characterized by mean field type stochastic Allen-Cahn equations.

To the best of our knowledge, quantitative convergence rates for PoC have not yet been established for general infinite-dimensional interacting particle dynamics.  In contrast to existing literatures (cf.~\cite{BKKX,CKS,C23,ES,HLL25,HLL26,KX}), which only prove qualitative convergence, we address the problem by combining the theory of martingale differences in uniformly convex Banach spaces with delicately constructed stopping times. Consequently, we establish quantitative estimates within the variational framework under refined local monotonicity assumptions (i.e., Conditions (\ref{ass01}) and (\ref{ass02})), leading to explicit convergence rates in two distinct regimes (i.e., Theorems \ref{pocth} and \ref{pocth2}). Furthermore, we apply our general framework to several important models, including stochastic SVGD (\ref{SSVGD}),  mean-field Allen-Cahn equations, and Lagrangian averaged Burgers equations.
Remarkably, even in the finite-dimensional setting, these results are novel and of independent interest, yielding near-optimal, dimension-free convergence rates for mean-field SDEs.

\subsection*{Structure of the paper}
The remainder of the paper is structured as follows. Section \ref{sec4} presents our main results regarding well-posedness and the quantitative PoC. In Section \ref{secex}, we provide several illustrative examples to demonstrate the applicability of our general framework. Section \ref{well-posed} is devoted to proving the existence and uniqueness of solutions to the mean-field SPDEs. Finally, the proof of the quantitative mean-field limit is detailed in Section \ref{ProofPoC}.

\section{Main results}\label{sec4}
\subsection{Preliminaries}

This part establishes the necessary mathematical framework by introducing the function spaces and operators employed throughout this work.

For any separable Banach space $({\mathbb{X}}, \|\cdot\|_{\mathbb{X}}) $, we denote by $\mathbb{C}_T(\mathbb{X}):=C([0,T];\mathbb{X})$ the space of all continuous functions from $[0,T]$ to $\mathbb{X}$, which is a Banach space equipped with  the uniform norm as follows
$$\|u\|_{T,\mathbb{X}}:=\sup_{t\in[0,T]}\|u_t\|_{\mathbb{X}},~u\in \mathbb{C}_T(\mathbb{X}).$$
Let $\mathscr{P}({\mathbb{X}})$ be the space consisting of all probability measures on $\mathbb{X}$ endowed with the weak convergence topology. Moreover, for any $p>0$ we denote
$$\mathscr{P}_p({\mathbb{X}}):=\Big\{\mu\in\mathscr{P}({\mathbb{X}}):\mu(\|\cdot\|_{{\mathbb{X}}}^p):=\int\|\xi\|_{\mathbb{X}}^p\mu(d\xi)<\infty\Big\}.$$
It follows that $\mathscr{P}_p({\mathbb{X}})$ is a Polish space under the following $L^p$-Wasserstein metric
$$\mathcal{W}_{p,{\mathbb{X}}}(\mu,\nu):=\inf_{\pi\in\mathscr{C}(\mu,\nu)}\Bigg(\int\|\xi-\eta\|_{\mathbb{X}}^p\pi(d\xi,d\eta)\Bigg)^{\frac{1}{p\vee1}},~~\mu,\nu\in\mathscr{P}_p({\mathbb{X}}),$$
where $\mathscr{C}(\mu,\nu)$ stands for the set of all couplings for  $\mu$ and $\nu$,  i.e., the set of all Borel probability measures $\pi$ on $\mathbb{X}\times \mathbb{X}$ such that
$\pi(\cdot\times \mathbb{X})=\mu(\cdot)~\text{and}~\pi(\mathbb{X}\times \cdot)=\nu(\cdot).$

\vspace{1mm}
In what follows, we summarize some  definitions for $p$-uniform convexity and smoothness on Banach spaces,   which are essential in deriving the quantitative (dimension-free) Propagation of Chaos for mean-field dynamics.
\begin{definition}$($cf.~\cite{DSDC}$)$
	A Banach space $\mathbb{X}$ is called strictly convex if for all $x, y \in \mathbb{X}$ with $x \neq y$ and $\|x\|_\mathbb{X} = \|y\|_\mathbb{X} = 1$, we have
	$$
	\Big\|\frac{x+y}{2}\Big\|_\mathbb{X}<1.
	$$
	Moreover, the modulus of convexity of space $\mathbb{X}$ is the function $\delta_{\mathbb{X}} : (0, 2] \to [0, 1]$ defined by
	\begin{equation*}
		\delta_\mathbb{X}(\varepsilon) := \inf\Big\{ 1 - \| \frac{x + y}{2} \|_\mathbb{X} : \|x\|_\mathbb{X} = \|y\|_\mathbb{X} = 1,  \|x - y\|_\mathbb{X}=\varepsilon \Big\}.
	\end{equation*}
\end{definition}

\begin{definition}$($cf.~\cite{X}$)$
	A Banach space $\mathbb{X}$ is said to be
	\begin{enumerate}
		\item[(i)] uniformly convex if $ \delta_\mathbb{X}(\varepsilon) > 0$ for all $ \varepsilon \in (0, 2]$;
		
		\vspace{1mm}
		\item[(ii)] $p$-uniformly convex for $p > 1$, if
		\begin{equation*}
			\delta_\mathbb{X}(\varepsilon) \gtrsim \varepsilon^p,\quad \varepsilon \in (0, 2].
		\end{equation*}
	\end{enumerate}
\end{definition}

\begin{remark}
	Strict convexity implies that the midpoint of any non-trivial chord of the unit sphere lies strictly within the unit sphere. As mentioned in \cite{B},  strict convexity is a local property in the sense that it is not necessarily uniform across the unit sphere.
	Uniform convexity strengthens this notion by providing a precise measurement of the inward displacement. Specifically, the modulus of convexity  $\delta_\mathbb{X}(\varepsilon)$ defines the minimum depth at which the midpoint of any two unit vectors, separated by at least  $\varepsilon$, must lie inside the closed unit ball.
	
	Furthermore, regarding  the $p$-uniform convexity, this inward deviation is bounded below by a constant multiple of $\varepsilon^p$; thus the midpoint cannot approach the sphere arbitrarily closely when the endpoints remain at least $\varepsilon$ apart. This offers a quantitative refinement of uniform convexity, which only requires $\delta_\mathbb{X}(\varepsilon)$ to be strictly positive.
\end{remark}
\begin{remark}\label{uniformspace}
	It is well-known (see e.g. \cite{LT,XR}) that every Hilbert space $\mathbb{H}$ is $2$-uniformly convex with specific modulus of convexity given by
	$$
	\delta_\mathbb{H}(\varepsilon) := 1 - \sqrt{1 - \frac{1}{4}\varepsilon^2}.
	$$
	Moreover, the classical Banach spaces $L^p$ and $W^{k,p}$ are $2$-uniformly convex if $1<p< 2$
	and $p$-uniformly convex if $p \geq 2$ with
	$$
	\delta_{L^p}(\varepsilon) = \delta_{W^{k,p}}(\varepsilon) :=
	\begin{cases}
		\dfrac{p-1}{8}\varepsilon^2 + o(\varepsilon^2) > \dfrac{p-1}{8}\varepsilon^2, & 1 < p < 2, \\
		1 - \Big[1 - \Big(\dfrac{\varepsilon}{2}\Big)^p\Big]^{1/p} > \dfrac{1}{p}\Big(\dfrac{\varepsilon}{2}\Big)^p, & p \geq 2.
	\end{cases}
	$$
\end{remark}

\begin{definition}$($cf.~\cite{C}$)$
	A Banach space $\mathbb{X}$ is said to be smooth, if for all $x \in \mathbb{X}$ with $\|x\|_\mathbb{X} = 1$, there exists a unique $f \in \mathbb{X}^*$, which is the dual space of $\mathbb{X}$, such that $\|f\|_{\mathbb{X}^*} = f(x) = 1$.
	Moreover, the modulus of smoothness of space $\mathbb{X}$ is the function $$\rho_\mathbb{X} : [0, \infty) \to [0, \infty)$$ defined by
	\begin{equation*}
		\rho_\mathbb{X}(\sigma) := \sup\bigg\{ \frac{\|x + y\|_\mathbb{X} + \|x - y\|_\mathbb{X}}{2} - 1 : x, y \in \mathbb{X}, \|x\|_\mathbb{X} = 1, \|y\|_\mathbb{X} = \sigma \bigg\}.
	\end{equation*}
\end{definition}

\begin{definition}$($cf.~\cite{LT}$)$
	A Banach space $\mathbb{X}$ is said to be $q$-uniformly smooth for $q > 1$, if
	\begin{equation*}
		\rho_\mathbb{X}(\sigma) \lesssim \sigma^q, \quad \sigma \in (0, \infty).
	\end{equation*}
\end{definition}

We point out that there is a duality relationship between uniform smoothness and uniform convexity by the following result.
\begin{proposition}\label{dualization}$($cf.~\cite{C}$)$
	Let $\mathbb{X}$ be a Banach space.
	Then $\mathbb{X}$ is $p$-uniformly convex iff $\mathbb{X}^*$ is $q$-uniformly smooth, where $q$ is the conjugate exponent of $p$.
\end{proposition}



We now introduce the notion of martingale difference sequences, which plays a pivotal role in the proof presented in Section  \ref{ProofPoC} below.
\begin{definition}\label{mddef}
	Let $(\Omega, \mathscr{F}, \mathbb{P})$ be a probability space and let $\mathbb{X}$ be a separable Banach space. For any $r \geq 1$, we denote by $\mathbb{L}_\mathbb{X}^r$ the space of $\mathbb{X}$-valued random variables such that
	$$\|\cdot\|_{\mathbb{L}_\mathbb{X}^r}^r := \mathbb{E}\|\cdot\|_\mathbb{X}^r<\infty.$$ Let $(\mathscr{F}_i)_{i \geq 1}$ be a non-decreasing sequence of sub-$\sigma$-algebras of $\mathscr{F}$. We say that a sequence of $\mathbb{X}$-valued random variables $\{Y_i\}_{i \geq 1}$ is a martingale difference sequence with respect to the filtration $(\mathscr{F}_i)_{i \geq 1}$ if
	\begin{enumerate}
		\item[(i)] for any $i \geq 1$, $Y_i$ is $\mathscr{F}_i$-measurable and belongs to $\mathbb{L}_\mathbb{X}^1$;
		
		\vspace{1mm}
		\item[(ii)] for any $i \geq 2$, $\mathbb{E}[Y_i \mid \mathscr{F}_{i-1}] = 0$~~ $\mathbb{P}$-a.s..
	\end{enumerate}
\end{definition}

\subsection{Weak solutions}\label{secin}

Recall the Gelfand triple (\ref{gelfand}). For some  constants $\alpha>1$ and $\beta> 0$, let
$$\mathfrak{M}:=\mathscr{P}_{\beta}(\mathbb{H})\cap\mathscr{P}_{2}(\mathbb{H})\cap \mathscr{P}_{\alpha}(\mathbb{V}).$$
Then for the measurable maps
$$
\mathcal{A}:[0,T]\times {\mathbb{V}}\times\mathfrak{M}\rightarrow {\mathbb{V}}^*,~~\mathcal{B}:[0,T]\times {\mathbb{V}}\times\mathfrak{M}\rightarrow L_2(U;\mathbb{H}),
$$
we consider the following general type of mean-field stochastic PDEs
\begin{equation}\label{eqSPDE}
	dX_t=\mathcal{A}(t,X_t,\mathscr{L}_{X_t})dt+\mathcal{B}(t,X_t,\mathscr{L}_{X_t})dW_t,
\end{equation}
where $\{W_t\}_{t\geq 0}$ is an $U$-valued cylindrical Wiener process defined on a complete filtered probability space $\left(\Omega,\mathscr{F},\{\mathscr{F}_t\}_{t\geq 0},\mathbb{P}\right)$.

We  first recall the definition of weak solutions to mean-field dynamics (\ref{eqSPDE}).
\begin{definition}\label{dew} $($Weak solution$)$ A pair $(X,W)$ is called a (probabilistically) weak solution to mean-field stochastic PDE $(\ref{eqSPDE})$, if there exists a stochastic basis $(\Omega,\mathscr{F},\{\mathscr{F}_t\}_{t\geq0},\mathbb{P})$ such that $X$ is an $\{\mathscr{F}_t\}$-adapted process and  $W$ is an $U$-valued cylindrical Wiener process on $(\Omega,\mathscr{F},\{\mathscr{F}_t\}_{t\geq 0},\mathbb{P})$ and the following holds:
	
	\vspace{2mm}
	(i) $X\in \mathbb{C}_T(\mathbb{H})\cap L^\alpha([0,T];\mathbb{V})$ $\mathbb{P}$-a.s.;
	
	\vspace{2mm}
	(ii) $\int_0^T\|\mathcal{A}(s,X_s,\mathscr{L}_{X_s})\|_{\mathbb{V}^*}ds+\int_0^T\|\mathcal{B}(s,X_s,\mathscr{L}_{X_s})\|_{L_2(U;\mathbb{H})}^2ds<\infty$~~ $\mathbb{P}$-a.s.;
	
	\vspace{2mm}
	(iii) The following identity holds in ${\mathbb{V}}^*$
	$$X_t=X_0+\int_0^t \mathcal{A}(s,X_s,\mathscr{L}_{X_s})ds+\int_0^t \mathcal{B}(s,X_s,\mathscr{L}_{X_s})dW_s,~t\in[0,T],~\mathbb{P}\text{-a.s.}.$$
\end{definition}

We introduce, for the first time, the notion of measure-dependent pseudo-monotone operators. To this end, we denote $$\mathfrak{M}_b:=\big\{\mu\in\mathfrak{M}:\mu(\|\cdot\|_{\mathbb{V}}^{\alpha})\leq K\big\}$$ for some $K>0$.
\begin{definition}\label{deps}$($Measure-dependent pseudo-monotone operators$)$ For some constants $\beta> 0$ and $\alpha>1$,  an operator
	$$\mathcal{A}:\mathbb{V}\times\mathfrak{M}\to\mathbb{V}^*$$
	is said to be pseudo-monotone if, for any sequences $\{u_n\}_{n=1}^{\infty},u$ in $\mathbb{V}$ and  $\{\mu_n\}_{n=1}^{\infty},\mu$ in $\mathfrak{M}_b$ such that   $u_n\to u$ weakly in $\mathbb{V}$ and  $\mu_n\to\mu$ in $\mathscr{P}_{\beta}(\mathbb{H})\cap\mathscr{P}_{2}(\mathbb{H})$, the following holds
	$$\liminf _{n \rightarrow \infty}\,_{\mathbb{V}^*}\langle \mathcal{A}(u_{n},\mu_n), u_{n}-u\rangle_{\mathbb{V}} \geq 0,$$
	then for any $v \in \mathbb{V}$, we have
	$$\limsup _{n \rightarrow \infty}\,_{\mathbb{V}^*}\langle \mathcal{A}(u_{n},\mu_n), u_{n}-v\rangle_{\mathbb{V}} \leq \,_{\mathbb{V}^*}\langle \mathcal{A}(u,\mu), u-v\rangle_{\mathbb{V}}.$$
\end{definition}
\begin{remark}
	Note that Definition \ref{deps} generalizes the classical pseudo-monotone operators, first introduced by Br\'{e}zis in \cite{B68}, to the setting of measure-dependent operators. This might shed new light beyond the present work,  especially for nonlinear mean-field evolution equations in Banach spaces.
\end{remark}

To ensure the existence of weak solutions, we suppose that there are some constants $\delta>0$, $\alpha>1$,  $\beta> 0$, and $p \in [\vartheta_1, \infty) \cap (\vartheta_2, \infty)$, where $\vartheta_1:=\max\{\beta+2,2\beta\}$ and $\vartheta_2:=\max\{2\beta(\alpha-1),4\}$, such that the following conditions hold for a.e.~$t\in[0,T]$.

\vspace{1mm}
\begin{enumerate}
	\item [$(\mathbf{A}_1)$] (Continuity of $\mathcal{B}$)
	For any sequences $\{u_n\}_{n\in\mathbb{N}},u$ in $\mathbb{V}$ and  $\{\mu_n\}_{n\in\mathbb{N}},\mu$ in $\mathfrak{M}$, such that $u_n\to u$ in $\mathbb{H}$ and  $\mu_n\to\mu$ in $\mathscr{P}_{\beta}(\mathbb{H})$, then
	\begin{equation*}
		\lim_{n\to\infty}\| \mathcal{B}(t,u_n,\mu_n)-\mathcal{B}(t,u,\mu)\|_{L_2(U;\mathbb{H})}=0.
	\end{equation*}
	
	\item [$(\mathbf{A}_2)$] (Pseudo-monotonicity)
	$\mathcal{A}(t,\cdot,\cdot)$ is pseudo-monotone in the sense of Definition \ref{deps}.
	
	\vspace{1mm}
	\item [$(\mathbf{A}_3)$] (Coercivity)
	For any $u\in \mathbb{V}$ and $\mu\in\mathfrak{M}$,
	\begin{eqnarray*}
		\!\!\!\!\!\!\!\!&&2_{\mathbb{V}^*}\langle \mathcal{A}(t,u,\mu),u\rangle_{\mathbb{V}}+(p-1)\|\mathcal{B}(t,u,\mu)\|_{L_2(U;\mathbb{H})}^2+\delta\|u\|_{\mathbb{V}}^\alpha
		\nonumber \\
		\!\!\!\!\!\!\!\!&&\lesssim   1+\|u\|_{\mathbb{H}}^2+\mu(\|\cdot\|_{\mathbb{H}}^{2}).
	\end{eqnarray*}

	\vspace{1mm}
	\item [$(\mathbf{A}_4)$] (Polynomial growth)
	For any $u\in \mathbb{V}$ and $\mu\in\mathfrak{M}$,
	\begin{equation}\label{conA3}
		\|\mathcal{A}(t,u,\mu)\|_{{\mathbb{V}}^*}^{\frac{\alpha}{\alpha-1}}\lesssim (1+\|u\|_{\mathbb{V}}^{\alpha}+\mu(\|\cdot\|_{\mathbb{V}}^{\alpha} ))(1+\|u\|_{{\mathbb{H}}}^{\beta}+\mu(\|\cdot\|_{\mathbb{H}}^{\beta})),
	\end{equation}
	\begin{equation}\label{conb}
		\|\mathcal{B}(t,u,\mu)\|_{L_2({U},{\mathbb{H}})}^2\lesssim 1+\|u\|_{\mathbb{H}}^{\beta}+\mu(\|\cdot\|_{\mathbb{H}}^{\beta}).
	\end{equation}
\end{enumerate}

The following is the first main result concerning the  existence of weak solutions to  mean-field stochastic PDEs (\ref{eqSPDE}).
\begin{theorem}\label{th1}
	Suppose that $(\mathbf{A}_1)$-$(\mathbf{A}_4)$ hold.
	For any initial data $\xi\sim\mu_0\in \mathscr{P}_p(\mathbb{H})$, where $p$ is the same as in $(\mathbf{A}_3)$,
	then Eq.~$(\ref{eqSPDE})$ has a weak solution in the sense of Definition \ref{dew}. Moreover,  for any $\gamma \in (0,1)$ we have
	\begin{equation}\label{esq370}
		\sup_{t\in[0,T]}\mathbb{E}\|X_t\|_{\mathbb{H}}^{p}+\mathbb{E}\Big[\sup_{t\in[0,T]}\|X_t\|_{\mathbb{H}}^{\gamma p}\Big]+\mathbb{E}\int_0^T(1+\|X_t\|_{\mathbb{H}}^{p-2})\|X_t\|_{\mathbb{V}}^{\alpha}dt<\infty.
	\end{equation}
\end{theorem}
\begin{remark}
	(i) Compared to previous works \cite{BKKX,CKS,C23,HLL24,HLL26}, Theorem \ref{th1} establishes a more general framework for the existence of weak solutions to infinite-dimensional mean-field dynamics governed by pseudo-monotone operators. We will derive an equivalent characterization of Definition \ref{deps} and employ it to establish the convergence of Galerkin sequences associated with  Eq.~(\ref{eqSPDE}), see Subsections \ref{sec3.1} and \ref{sec3.4} below for more details.
	
	(ii) In the finite-dimensional  setting (see Subsection \ref{finite_dim} for more details), this result seems also new in the literature (cf.~\cite{E,GHM,HSS} and references therein), as it  allows diffusion coefficients with, possibly, super-linear kernels, which is of particular importance in filtering problems (cf.~\cite{DT18}).
\end{remark}

%
%

\subsection{Strong solutions}\label{Strongresult}
This subsection addresses the existence of strong solutions and establishes the well-posedness of the mean-field stochastic PDEs (\ref{eqSPDE}). To this end, we recall the following definition.
\begin{definition}\label{de1}
	$($Strong solution$)$ We say that there exists a (probabilistically) strong solution to (\ref{eqSPDE}) if for every probability space $(\Omega,\mathscr{F},\{\mathscr{F}_t\}_{t\geq 0},\mathbb{P})$ with an $U$-valued cylindrical Wiener process $W$, there exists
	an $\{\mathscr{F}_t\}$-adapted process $X$ such that the properties (i)-(iii) in Definition \ref{dew} hold.

\end{definition}

To ensure the existence of strong solutions, we propose the following conditions.
\vspace{1mm}
\begin{enumerate}
	\item [$(\mathbf{A}_2^*)$] (Demi-continuity) For any $\phi\in {\mathbb{V}}$, the map
	\begin{equation*}
		{\mathbb{V}}\times\mathfrak{M}\ni(u,\mu)\mapsto~_{{\mathbb{V}}^*}\langle \mathcal{A}(t,u,\mu),\phi\rangle_{\mathbb{V}}
	\end{equation*}
	is continuous;
	
	\item [$(\mathbf{A}_5)$] (Local monotonicity) For any $t\in[0,T]$, $u,v\in {\mathbb{V}}$, and $\mu,\nu\in\mathfrak{M}$,
	\begin{eqnarray}\label{unip}
		\!\!\!\!\!\!\!\!&&2_{{\mathbb{V}}^*}\langle \mathcal{A}(t,u,\mu)-\mathcal{A}(t,v,\nu),u-v\rangle_{\mathbb{V}}+\|\mathcal{B}(t,u,\mu)-\mathcal{B}(t,v,\nu)\|_{L_2(U;\mathbb{H})}^2
		\nonumber\\
		\!\!\!\!\!\!\!\!&&\leq
		(\rho(u,\mu)+\eta(v,\nu))(\|u-v\|_{\mathbb{H}}^2+\mathcal{W}_{2,{\mathbb{H}}}(\mu,\nu)^2),
	\end{eqnarray}
	where $\rho,\eta:{\mathbb{V}}\times \mathfrak{M}\to [0,\infty)$ are  measurable functions satisfying
	\begin{equation}\label{esq22}
		\rho(u,\mu)+\eta(u,\mu)\lesssim (1+\|u\|_{\mathbb{V}}^{\alpha}+\mu(\|\cdot\|_{\mathbb{V}}^{\alpha}))(1+\|u\|_{\mathbb{H}}^{\beta}+\mu(\|\cdot\|_{\mathbb{H}}^{\beta})).
	\end{equation}
\end{enumerate}

\vspace{1mm}
Our second main result concerns the  existence of strong solutions to  mean-field stochastic PDEs (\ref{eqSPDE}).
\begin{theorem}\label{th2}
	Suppose that $(\mathbf{A}_1)$, $(\mathbf{A}_2^*)$, $(\mathbf{A}_3)$, $(\mathbf{A}_4)$, and $(\mathbf{A}_5)$ hold.
	For any initial data $\xi\in L^p(\Omega,\mathscr{F}_0,\mathbb{P};{\mathbb{H}})$, where $p$ is the same as in $(\mathbf{A}_3)$,
	then Eq.~$(\ref{eqSPDE})$ has a strong solution in the sense of Definition \ref{de1}. Moreover, the estimate $(\ref{esq370})$ holds.
\end{theorem}
\begin{remark}\label{re0}
	Differing from the classical variational setting (cf.~e.g.~\cite{liu2015stochastic,RSZ}), some highly non-trivial challenges arise  primarily due to the inherent non-local nature of measure-dependent operators. Indeed, as pointed out in a recent review (cf.~\cite{CD22}), there is no real hope of obtaining stronger results at this level of generality. For instance, in previous finite-dimensional studies  (cf.~\cite{E,GHM}),  controlling exponential moments was essential for proving the existence of solutions, whereas, in the present work, we  do not require any exponential moment assumptions.
\end{remark}

To establish the well-posedness of the  mean-field stochastic PDEs (\ref{eqSPDE}),  we require the following  {\it decoupled} local monotonicity condition.

\vspace{2mm}
\begin{enumerate}
	\item [$(\mathbf{A}^*_5)$] (Decoupled local monotonicity)
	For any $t\in[0,T]$, $u,v\in {\mathbb{V}}$, and $\mu,\nu\in\mathfrak{M}$,
	\begin{eqnarray*}
		&&2{}_{{\mathbb{V}}^*}\langle \mathcal{A}(t,u,\mu)-\mathcal{A}(t,v,\nu),u-v\rangle_{\mathbb{V}}+\|\mathcal{B}(t,u,\mu)-\mathcal{B}(t,v,\nu)\|_{L_2(U;\mathbb{H})}^2
		\nonumber\\
		&&\leq
		(\rho(0,\mu)+\eta(0,\nu))\|u-v\|_{\mathbb{H}}^2
		+(\rho(u,\mu)+\eta(v,\nu))\mathcal{W}_{2,\mathbb{H}}(\mu,\nu)^2,
	\end{eqnarray*}
	where $\rho,\eta$ are the same as in  $(\mathbf{A_5})$.

\end{enumerate}

\vspace{1mm}
The following presents the well-posedness result of the  mean-field stochastic PDEs (\ref{eqSPDE}).
\begin{theorem}\label{th3}
	Suppose that $(\mathbf{A}_1)$, $(\mathbf{A}_2^*)$, $(\mathbf{A}_3)$, $(\mathbf{A}_4)$,   and $(\mathbf{A}^*_5)$ hold.
	For any initial data $\xi\in L^p(\Omega,\mathscr{F}_0,\mathbb{P};{\mathbb{H}})$, where $p$ is the same as in $(\mathbf{A}_3)$,
	then Eq.~$(\ref{eqSPDE})$ has a unique (strong and weak) solution and  the estimate $(\ref{esq370})$ holds.
	
	Furthermore, Eq.~$(\ref{eqSPDE})$ is continuous on the initial data in the sense that for any $\xi_n,\xi\in L^p(\Omega,\mathscr{F}_0,\mathbb{P};{\mathbb{H}})$ with
	$\|\xi_n-\xi\|_{L^{p'}(\Omega;\mathbb{H})}\to0$, $p'<p$, then
	\begin{equation}\label{initialconti}
		\lim_{n\to \infty}\mathbb{E}\Big[\sup_{t\in[0,T]}\|X_t(\xi_n)-X_t(\xi)\|_\mathbb{H}^{p'}\Big]=0.
	\end{equation}
	
\end{theorem}

\begin{remark}\label{re1}
	(i) In contrast to the classical setting of stochastic PDEs  (cf.~\cite{liu2015stochastic,RSZ}), if the operator $\mathcal{A}(t,\cdot,\cdot)$ is merely locally monotone in the sense of $(\mathbf{A}_5)$, the pathwise uniqueness,
	or even uniqueness in distribution, of solutions to Eq.~(\ref{eqSPDE}) generally fails to hold. For relevant counterexamples, we refer to  the classical paper \cite{S87}.
	
	\vspace{1mm}
	(ii) We remark that $(\mathbf{A}^*_5)$ is a mild condition to ensure the uniqueness of solutions, which finds a wide range of  applications in finite and infinite-dimensional mean-field dynamics with nonlinear kernels.
	In particular, our main results encompass stochastic Stein variational gradient descent and Lagrangian averaged Burgers equations. A more rigorous treatment of these examples can be found  in Section \ref{secex}.
\end{remark}

\subsection{Quantitative propagation of chaos}\label{poc}
The mean field limit theory provides  a simplification of an $N$-body system to a one-body dynamic which
interacts with itself. In this subsection, we aim to establish quantitative dimension-free PoC for mean-field stochastic interacting particle systems in the infinite-dimensional framework.

Let $\{\xi^i\}_{i=1}^{N}$ and  $\{W^i\}_{i=1}^{N}$ be $N$-independent and identically distributed copies of the initial data $\xi$ and the $U$-valued cylindrical Wiener process $W$ associated with the mean-field stochastic PDEs (\ref{eqSPDE}), respectively. We formulate the non-interacting particles system (nIPS) as follows
\begin{equation}\label{nIPS}
	X^i_t=\xi^i+\int_0^t\mathcal{A}(s,X^i_s,\mathscr{L}_{X^i_s})ds
	+\int_0^t\mathcal{B}(s,X^i_s,\mathscr{L}_{X^i_s})dW^i_s,
\end{equation}
with $i \in \{1, \dots, N\}$. We mention that the well-posedness  of the nIPS (\ref{nIPS}) has been established in Theorem \ref{th3}.

We subsequently proceed to consider the following interacting particle system  (IPS)
\begin{equation}\label{IPS}
	X_t^{i,N} = \xi^i + \int_0^t \mathcal{A}(s,X^{i,N}_s, \mu_s^{N}) ds + \int_0^t \mathcal{B}(s,X^{i,N}_s, \mu_s^{N}) dW_s^i,
\end{equation}
where
$
\mu_t^{N} := \frac{1}{N} \sum_{j=1}^N \delta_{X_t^{j,N}}
$
is the empirical law of $(X_t^{1,N},X_t^{2,N},\dots,X_t^{N,N})$.

The existence and uniqueness of (probabilistically) strong solutions to the IPS (\ref{IPS}) are given in the following.
\begin{theorem}\label{wp-nIPS}
	Suppose that the assumptions in Theorem \ref{th2} hold.
	For any initial data $\xi^i\in L^p(\Omega,\mathscr{F}_0,\mathbb{P};{\mathbb{H}})$, where $p$ is the same as in $(\mathbf{A}_3)$,
	then Eq.~$(\ref{IPS})$ has a unique strong solution $X^N:=(X_t^{1,N},\ldots,X_t^{N,N})$.
\end{theorem}
\begin{proof}
	This result follows directly from verifying  the coefficients of IPS (\ref{IPS}) satisfy the assumptions in Theorem 2.6 in \cite{RSZ}
	on the product spaces. We omit details.
\end{proof}

In this work, we require the following structure  to derive the dimension-free PoC.

\vspace{1mm}
\begin{enumerate}
	\item [$(\mathbf{A}_6)$]
	There exist  constants $\alpha>1$,  $\beta> 0$, and measurable maps $$\tilde{\mathcal{A}}:[0,T]\times\mathbb{V}\times\mathbb{V}\to\mathbb{V}^*,~~\tilde{\mathcal{B}}:[0,T]\times {\mathbb{V}}\times\mathbb{V}\rightarrow L_2(U;\mathbb{H})$$
	such that
	\begin{equation*}
		\|\tilde{\mathcal{A}}(t,u,y)\|_{{\mathbb{V}}^*}^{\frac{\alpha}{\alpha-1}}\lesssim (1+\|u\|_{\mathbb{V}}^{\alpha}+\|y\|_{\mathbb{V}}^{\alpha})(1+\|u\|_{{\mathbb{H}}}^{\beta}+\|y\|_{\mathbb{H}}^{\beta}),
	\end{equation*}
	and
	\begin{equation*}
		\|\tilde{\mathcal{B}}(t,u,y)\|_{L_2({U},{\mathbb{H}})}^2\lesssim 1+\|u\|_{\mathbb{H}}^{\beta}+\|y\|_{\mathbb{H}}^{\beta}.
	\end{equation*}
	Then for any $u\in \mathbb{V}$ and $\mu,\nu\in\mathfrak{M}$,
	\begin{equation}\label{es12}
		\|\mathcal{A}(t,u,\mu)-\mathcal{A}(t,u,\nu)\|_{{\mathbb{V}}^*}
		\lesssim\|\int\tilde{\mathcal{A}}(t,u,y)\mu(dy)
		-\int\tilde{\mathcal{A}}(t,u,y)\nu(dy)\|_{{\mathbb{V}}^*},
	\end{equation}
	and
	\begin{equation*}
		\|\mathcal{B}(t,u,\mu)-\mathcal{B}(t,u,y)\|_{L_2(U;\mathbb{H})}
		\lesssim\|\int\tilde{\mathcal{B}}(t,u,y)\mu(dy)
		-\int\tilde{\mathcal{B}}(t,u,y)\nu(dy)\|_{L_2(U;\mathbb{H})}.
	\end{equation*}
\end{enumerate}

\begin{remark}\label{remark1}
	Note that when $\mathcal{A}(t,u,\mu)$ (resp.~$\mathcal{B}(t,u,\mu)$) is of integral form, i.e.,
	$$\mathcal{A}(t,u,\mu):=\int\tilde{\mathcal{A}}(t,u,y)\mu(dy),$$
	then assumption (\ref{es12})  is automatically satisfied.
\end{remark}

We now state  the first quantitative PoC result as follows.
\begin{theorem}\label{rate1}
 Let $\mathbb{V}$ be a separable $\alpha$-uniformly convex Banach space with $\alpha\geq2$.	Suppose that the assumptions in Theorem \ref{th3}  and $(\mathbf{A}_6)$ hold with  $p \in [4\beta, \infty) \cap (\lambda, \infty)$, where $\lambda:=\max\{2\beta(\alpha-1),2\beta+4\}$, and $(\mathbf{A}^*_5)$ replaced by the following condition
	\begin{enumerate}
		\vspace{1mm}
		\item [$(\mathbf{A}'_5)$]
		There exist constants  $\alpha>1$, and $\beta> 0$ such that for any $t\in[0,T]$, $u,v\in {\mathbb{V}}$, and $\mu,\nu\in\mathfrak{M}$,
		\begin{equation*}
			_{{\mathbb{V}}^*}\langle \mathcal{A}(t,u,\mu)-\mathcal{A}(t,v,\nu),u-v\rangle_{\mathbb{V}}
			\leq
			\rho(0,\mu)\|u-v\|_{\mathbb{H}}^2
			+\rho(u,\mu)\mathcal{W}_{2,\mathbb{H}}(\mu,\nu)^2,
		\end{equation*}
		\begin{equation*}
			\|\mathcal{B}(t,u,\mu)-\mathcal{B}(t,v,\nu)\|_{L_2(U;\mathbb{H})}^2
			\leq
			\rho(0,\mu)\|u-v\|_{\mathbb{H}}^2
			+\rho(u,\mu)\mathcal{W}_{2,\mathbb{H}}(\mu,\nu)^2,
		\end{equation*}
		where $\rho$ is the same as in $(\mathbf{A}_5)$.
\end{enumerate}		
Then we have the following pathwise chaos estimate
	\begin{equation}\label{pocth1}
		\sup_{i \in \{1,\dots,N\}}  \mathbb{E}\Big[\sup_{t \in [0,T]}\|X_t^{i,N} - X_t^i\|_{\mathbb{H}}^2\Big] \lesssim N^{-\frac{1}{\alpha}}.
	\end{equation}
\end{theorem}

\begin{remark}
	The appearance of the uniform convexity parameter in the propagation of chaos rate  reveals, for the first time,  a nontrivial interaction between the geometry of the underlying state space and the  mean-field approximation.
	Specifically, the estimate (\ref{pocth1}) establishes a pathwise chaos rate of order $N^{-\frac{1}{2\alpha}}$. In particular, when  $\mathbb{V}$ is an Euclidean/Hilbert space (where $\alpha=2$), the convergence rate is of order $N^{-\frac{1}{4}}$, which is not optimal in the finite-dimensional setting.
\end{remark}

\begin{remark}\label{remark001}
	We provide a motivating example involving i.i.d.~random variables that confirms the dependence of the law of large numbers rate on the geometry of the space, as characterized by its convexity parameter.
	
	Let
	$\alpha\geq 2$ and $q:=\frac{\alpha}{\alpha-1}$.
	Let $\{e_k\}_{k= 2}^{\infty}$ be the canonical unit-vector basis of \(\ell^q\). Define
	$$
	a_k:=\frac{1}{k(\log k)^2},
	\quad k\geq 2,
	$$
	and set
	$$
	C_0:=\left(\sum_{k=2}^{\infty}a_k\right)^{-1},
	\quad
	p_k:=C_0a_k=\frac{C_0}{k(\log k)^2},
	\quad k\geq 2.
	$$
	Then it is clear that \((p_k)_{k\geq 2}\) is a probability distribution.
	
	Let $K$ be a random variable with values in $\{2,3,\dots\}$ such that
	$$
	\mathbb P(K=k)=p_k,
	\quad k\geq 2,
	$$
	and let \(\varepsilon\) be an independent Rademacher random variable, i.e.,
	$$
	\mathbb P(\varepsilon=1)
	=
	\mathbb P(\varepsilon=-1)
	=
	\frac12.
	$$
	Then we can	define the $\ell^q$-valued random variable
	$$
	Z:=\varepsilon e_K.
	$$
	Since \(\varepsilon\) and \(K\) are independent and \(\mathbb E\varepsilon=0\), we infer that
	$$
	\mathbb E Z
	=
	\mathbb E(\varepsilon e_K)
	=
	\mathbb E\bigl[\mathbb E(\varepsilon e_K\mid K)\bigr]
	=
	\mathbb E\bigl[e_K\mathbb E(\varepsilon\mid K)\bigr]
	=0.
	$$
	Moreover,
	$$
	\|Z\|_{\ell^q}
	=
	\|\varepsilon e_K\|_{\ell^q}
	=
	|\varepsilon|\,\|e_K\|_{\ell^q}
	=1
	\qquad\mathbb{P}\text{-a.s.}
	$$

	Let $\{Z_j\}_{j=1}^{\infty}$ be i.i.d.\ copies of \(Z\), and write
	$$
	\bar{Z}_N:=\frac1N\sum_{j=1}^N Z_j.
	$$
	We can obtain the following result on the law of large numbers.		
	\begin{proposition}\label{propos1}
		There exists $N_0\in\mathbb N$ such that, for every
		$N\geq N_0$, we have
		\[
		N^{-\frac{1}{\alpha-1}}(\log N)^{-1}
		\lesssim
		\mathbb E\|\bar{Z}_N\|_{\ell^q}^q
		\lesssim
		N^{-\frac{1}{\alpha-1}}.
		\]
	\end{proposition}
	The detailed proof of Proposition \ref{propos1} is left in Appendix.		
\end{remark}

In order to derive a sharper rate for infinite-dimensional PoC, we introduce the following theorem.
\begin{theorem}\label{rate2}
 Let $\mathbb{V}$ be a separable $\alpha$-uniformly convex Banach space with $\alpha\geq2$.	Suppose that the assumptions in Theorem \ref{th3}  and $(\mathbf{A}_6)$ hold with  $p \in [4\beta, \infty) \cap (\lambda, \infty)$, where $\lambda:=\max\{2\beta(\alpha-1),2\beta+4\}$, and $(\mathbf{A}^*_5)$ replaced by the following condition
 	\begin{enumerate}
 	\vspace{1mm}	
	\item [$(\mathbf{A}''_5)$]
	There  exist constants $\delta_0>0$, $\alpha>1$, and $\beta\geq2$ such that for any $t\in[0,T]$, $u,v\in {\mathbb{V}}$, and $\mu,\nu\in\mathfrak{M}$,
	\begin{eqnarray*}
		\!\!\!\!\!\!\!\!&&_{{\mathbb{V}}^*}\langle \mathcal{A}(t,u,\mu)-\mathcal{A}(t,v,\nu),u-v\rangle_{\mathbb{V}}+\delta_0\|u-v\|_{\mathbb{V}}^{\alpha}
		\nonumber\\
		\leq\!\!\!\!\!\!\!\!&&
		(\rho(0,\mu)+\eta(0,\nu))\|u-v\|_{\mathbb{H}}^2
		+(\rho(u,\mu)+\eta(v,\nu))\mathcal{W}_{2,\mathbb{H}}(\mu,\nu)^2,
	\end{eqnarray*}
	\begin{eqnarray*}
		\!\!\!\!\!\!\!\!&&\|\mathcal{B}(t,u,\mu)-\mathcal{B}(t,v,\nu)\|_{L_2(U;\mathbb{H})}^2
		\nonumber\\
		\leq\!\!\!\!\!\!\!\!&&
		(\rho(0,\mu)+\eta(0,\nu))\|u-v\|_{\mathbb{H}}^2
		+(\rho(u,\mu)+\eta(v,\nu))\mathcal{W}_{2,\mathbb{H}}(\mu,\nu)^2,
	\end{eqnarray*}
	where $\rho$ is the same as in $(\mathbf{A}_5)$ and $\eta:{\mathbb{V}}\times \mathfrak{M}\to [0,\infty)$ is a measurable function satisfying
	\begin{equation*}
		\eta(u,\mu)\lesssim (1+\|u\|_{\mathbb{V}}^{\alpha}+\mu(\|\cdot\|_{\mathbb{V}}^{\alpha}))(1+\mu(\|\cdot\|_{\mathbb{H}}^{\beta}))+\|u\|_{\mathbb{H}}^{\beta}(1+\mu(\|\cdot\|_{\mathbb{V}}^{\alpha})).
	\end{equation*}
\end{enumerate}
Then we have the following pathwise and pointwise chaos estimates, respectively,
\begin{equation}\label{pocth3}
	\sup_{i \in \{1,\dots,N\}}\bigg\{\mathbb{E}\Big[\sup_{t \in [0,T]}\|X_t^{i,N} - X_t^i\|_{\mathbb{H}}^2\Big]+ \mathbb{E}\int_0^{T}\|X_s^i - X_s^{i,N}\|_{\mathbb{V}}^{\alpha}ds\bigg\} \lesssim N^{-\frac{1}{\alpha-1}\cdot\frac{q-\beta}{q-2}\cdot {\frac{p'-2}{p'}}},
\end{equation}
for any $\beta<q<p$ and $2<p'<\frac{2p}{\beta}$, and
\begin{equation}\label{pocth4}
	\sup_{i \in \{1,\dots,N\}}  \sup_{t \in [0,T]}\mathbb{E}\|X_t^{i,N} - X_t^i\|_{\mathbb{H}}^2 \lesssim N^{-\frac{1}{\alpha-1}\cdot {\frac{p-\beta}{p}}}.
\end{equation}
\end{theorem}

\begin{remark}
Compared with Theorem \ref{rate1}, when also accounting for the dissipation induced by the underlying second-order differential operators (i.e., Condition $(\mathbf{A}''_5)$), the pathwise estimate (\ref{pocth3}) and  pointwise estimate (\ref{pocth4}) improve upon (\ref{pocth1}). Specifically, when  $\mathbb{V}$ is a Hilbert space (corresponding to $\alpha=2$) and  $p$ is sufficiently large, these yield near-optimal rates of order $N^{-\frac{1}{2}+}$. An  open problem is whether the extra factors $\frac{q-\beta}{q-2}\cdot {\frac{p'-2}{p'}}$ and $\frac{p-\beta}{p}$ in (\ref{pocth3}) and (\ref{pocth4}), respectively, can be removed, which deserves further investigation in the future work.
\end{remark}

\begin{remark}
	To the best of our knowledge, this is the first work  to derive  dimension-free convergence rates of PoC for general infinite-dimensional mean-field interacting particle systems. Specifically,  compared to existing studies on mean-field stochastic PDEs (cf.~\cite{BKKX,CKS,C23,ES,HLL25,HLL26}), which only established qualitative convergence without rates via the stochastic compactness method, our approach provides (near-optimal) quantitative estimates.
	
The main proof relies on the Sznitman's synchronous coupling method with a delicate analysis of martingale difference sequences taking values in the dual of an $\alpha$-uniformly convex Banach space. Moreover, we utilize a refined stopping time argument  to treat the infinite-dimensional mean-field dynamics.
	Notably, these results are also novel and of independent interest even in the  finite-dimensional setting.
\end{remark}

\section{Examples/Applications}\label{secex}
In this section, we explore the applications of our theoretical framework to finite and infinite-dimensional interacting particle systems arising from fields such as machine learning, fluid mechanics, and quantum field theory.    These include Stein variational gradient descent, mean-field Allen-Cahn equations, and modified  Lagrangian averaged Burgers equations.

Throughout the paper, we adopt the convention that $|\cdot|$ and $\langle \cdot, \cdot \rangle$
are the Euclidean norm and inner product  on $\mathbb{R}^d$, whereas $\|\cdot\|$ stands for the Hilbert-Schmidt norm for matrices within $\mathbb{R}^d\otimes\mathbb{R}^d$.

We also summarize some fundamental notations and functional spaces used in the infinite-dimensional analysis.
Let $\mathcal{O} \subset \mathbb{R}^d$ be a bounded domain with a smooth boundary $\partial \mathcal{O}$. We denote by $(L^p(\mathcal{O}; \mathbb{R}^d), \|\cdot\|_{L^p})$ the space of $p$-integrable functions on $\mathcal{O}$.
For any integer $m \ge 0$, let $W_0^{m,p}(\mathcal{O}; \mathbb{R}^d)$ be the classical Sobolev space with Dirichlet boundary, which is endowed with equivalent  norm
\begin{equation*}
	\|f\|_{m,p} := \Big( \sum_{|\alpha|= m} \int_{\mathcal{O}} |D^\alpha f(x)|^p \, dx \Big)^{1/p}.
\end{equation*}
For the case $p=2$, we simplify the notation $\|\cdot\|_{m,2}$     by writing $\|\cdot\|_m$.

\subsection{Stochastic Stein variational gradient descent}\label{sec3.1}
Stein variational gradient descent (SVGD), introduced in \cite{LW1}, is a deterministic particle-based inference method that transports an ensemble of particles toward a target distribution $\pi(x)\propto e^{-\mathcal{V}(x)}$ by  the steepest descent direction of the
Kullback--Leibler divergence (see also \cite{DNS,LW2}). To avoid particles tending to collapse to a local mode under some particular conditions, the stochastic SVGD has been
proposed in the following form \cite{ZZCC}:
\begin{eqnarray}\label{SVGD-IPS}
	dX_{t}^{i,N} =\!\!\!\!\!\!\!\!&&-\lambda^{-1}\nabla \mathcal{V}(X_t^{i,N})dt-\frac{1}{N}\sum_{j=1}^N \nabla_y \kappa(X_t^{i,N}, X_t^{j,N}) dt\nonumber\\
	\!\!\!\!\!\!\!\!&&- \frac{1}{N}\sum_{j=1}^N \kappa(X_t^{i,N}, X_t^{j,N}) \nabla \mathcal{V}(X_t^{j,N})dt+\sqrt{2 \lambda^{-1}}dW_t^{i},\quad X_0^{i,N}=\xi^i,
\end{eqnarray}
where $\lambda>0$ is the temperature parameter and $W_t^i$ and $\xi^i$, $i = 1, \dots, N$, are $N$-i.i.d. $\mathbb{R}$-valued Brownian motions and initial values, respectively.

In the work \cite{LW2}, the authors demonstrate that by using polynomial kernel $\kappa(x,y)=xy$, SVGD can exactly estimate the mean and variance of Gaussian distributions, provided that the number of particles exceeds the dimension. Since Gaussian-like distributions are ubiquitous in practical applications and the estimations of mean and variance are often of particular importance, polynomial kernels offer a notable advantage over the kernels appearing in \cite{DNS,LW1}.
Inspired by \cite[Remark~46]{DNS}, we consider a more general class of kernels in the present work, i.e.,
$$
\kappa(x,y) := x^{2k-1} y^{2m-1},
\qquad
\mathcal{V}(y) := \frac{1}{2n}y^{2n},
$$
where $k,m,n \geq 1$.

The formal large $N$ limit of particle system (\ref{SVGD-IPS}) is given by the following  mean-field stochastic equations
\begin{eqnarray}\label{SVGD-nIPS}
	dX_t^i =\!\!\!\!\!\!\!\!&&-\lambda^{-1}\nabla \mathcal{V}(X_t^i)dt
	+ \sqrt{2 \lambda^{-1}}dW_t^{i} \nonumber\\
	\!\!\!\!\!\!\!\!&&-\int_{\mathbb{R}} ( \nabla_y \kappa(X_t, y) + \kappa(X_t, y) \nabla \mathcal{V}(y) ) \mathscr{L}_{X^i_t}(dy) dt,\quad X_0^i=\xi^i.
\end{eqnarray}

We now establish the well-posedness and the quantitative PoC  for Eqs.~(\ref{SVGD-nIPS}) and (\ref{SVGD-IPS}), respectively.
\begin{theorem}\label{SVGDrate}
	Suppose that the initial data $\xi^i\in L^p(\Omega,\mathscr{F}_0,\mathbb{P};\mathbb{R})$ with $p\geq\max\{ 32k-24,32(m+n)-40\}$. Then Eq.~$(\ref{SVGD-nIPS})$ has a
	unique strong and weak solution.
	Moreover, the following estimates hold
	\begin{equation*}
		\sup_{i \in \{1,\dots,N\}}\mathbb{E}\Big[\sup_{t \in [0,T]}|X_t^{i,N} - X_t^i|^2\Big] \lesssim N^ {-\frac{q-\beta}{q-2}\cdot{\frac{p'-2}{p'}}},
	\end{equation*}
	for any $2<p'<\frac{2p}{\beta}$ and $\beta<q<p$, and
	\begin{equation*}
		\sup_{i \in \{1,\dots,N\}}  \sup_{t \in [0,T]}\mathbb{E}|X_t^{i,N} - X_t^i|^2 \lesssim N^{-\frac{p-\beta}{p}},
	\end{equation*}
where $\beta:=\max\{ 8k-6,8(m+n)-10\}$.
\end{theorem}
\begin{proof}Set
	\begin{eqnarray*}
		\mathcal{A}(u,\mu)\!\!\!\!\!\!\!\!&&:=-\nabla\mathcal{V}(u)-\int_{\mathbb{R}} ( \nabla_y \kappa(u, y) + \kappa(u, y) \nabla \mathcal{V}(y) ) \mu(dy)\\
		\!\!\!\!\!\!\!\!&&= -u^{2n-1}-\int_{\mathbb{R}}\big((2m-1)u^{2k-1}
		y^{2m-2}+u^{2k-1}y^{2(m+n)-2}\big)\mu(dy).
	\end{eqnarray*}
	Obviously, $\mathcal{A}(u,\mu)$ is continuous.
	According to Theorems \ref{th3} and \ref{rate2}, we need to show that $\mathcal{A}(u,\mu)$
	satisfies conditions $(\mathbf{A}_3)$, $(\mathbf{A}_4)$, $(\mathbf{A}''_5)$, and $(\mathbf{A}_6)$.
	It is easy to get
	\begin{equation*}
		\langle\mathcal{A}(u,\mu),u\rangle=-u^{2n}-(2m-1)u^{2k}\mu( |\cdot|^{2m-2})-u^{2k}\mu(|\cdot|^{2(m+n)-2}) \leq 0
	\end{equation*}
	and
	\begin{eqnarray*}
		|\mathcal{A}(u,\mu)|^2\lesssim\!\!\!\!\!\!\!\!&&|u|^{4n-2}+ |u|^{4k-2}\mu( |\cdot|^{4m-4})+|u|^{4k-2}\mu(|\cdot|^{4(m+n)-4})
		\\
		\lesssim\!\!\!\!\!\!\!\!&&|u|^{4n-2}+|u|^{8k-4}+\mu(|\cdot|^{8(m+n)-8}),
	\end{eqnarray*}
	which imply $(\mathbf{A}_3)$-$(\mathbf{A}_4)$ hold with $\alpha=2$ and $\beta = \max\{ 8k-6,8(m+n)-10\}$.
	
	Next we verify that condition $(\mathbf{A}''_5)$ is satisfied. For any coupling $\pi \in \mathscr{C}(\mu,\nu)$, it follows that
	\begin{eqnarray*}
		\!\!\!\!\!\!\!\!&&\langle \mathcal{A}(u,\mu)-\mathcal{A}(v,\nu),u-v\rangle+
		|\mathcal{B}(u,\mu)-\mathcal{B}(v,\nu)|^2\\
		=\!\!\!\!\!\!\!\!&&-\langle u^{2n-1}-v^{2n-1},u-v\rangle\\
		\!\!\!\!\!\!\!\!&&-\langle\int_{\mathbb{R}}(2m-1)u^{2k-1}y_1^{2m-2}\mu(dy_1)-\int_{\mathbb{R}}(2m-1)v^{2k-1}y_2^{2m-2}\nu(dy_2),u-v\rangle\\
		\!\!\!\!\!\!\!\!&&-\langle \int_{\mathbb{R}}u^{2k-1}y_1^{2(m+n)-2}\mu(dy_1)-\int_{\mathbb{R}}v^{2k-1}y_2^{2(m+n)-2}\nu(dy_2),u-v\rangle
		\\
		\lesssim\!\!\!\!\!\!\!\!&&\big(\mu(|\cdot|^{4(m+n)-6})+\nu(|\cdot|^{4(m+n)-6})\big)|u-v|^2
		+|v|^{4k-2}\Big(\int_{\mathbb R\times\mathbb R} |y_1-y_2|^2 \pi(dy_1,dy_2)\Big),
	\end{eqnarray*}
	where we used the following estimates in the last step
	\begin{eqnarray*}
		\!\!\!\!\!\!\!\!&&-\langle \int_{\mathbb{R}}u^{2k-1}y_1^{2(m+n)-2}\mu(dy_1)-\int_{\mathbb{R}}v^{2k-1}y_2^{2(m+n)-2}\nu(dy_2),u-v\rangle\\
		\leq\!\!\!\!\!\!\!\!&& - \langle \int_{\mathbb R} (u^{2k-1}-v^{2k-1})y_1^{2(m+n)-2}\mu(dy_1),u-v \rangle\\
		\!\!\!\!\!\!\!\!&&- \langle \int_{\mathbb R\times\mathbb R}
		v^{2k-1}\big(y_1^{2(m+n)-2}-y_2^{2(m+n)-2}\big)\pi(dy_1,dy_2),u-v \rangle\\
		\lesssim\!\!\!\!\!\!\!\!&&|v|^{2k-1}|u-v|\int_{\mathbb R\times\mathbb R}
		\big(|y_1|^{2(m+n)-3}+|y_2|^{2(m+n)-3}\big)|y_1-y_2|\pi(dy_1,dy_2)\\
		\lesssim\!\!\!\!\!\!\!\!&&\big(\mu(|\cdot|^{4(m+n)-6})+\nu(|\cdot|^{4(m+n)-6})\big)|u-v|^2
		+|v|^{4k-2}\Big(\int_{\mathbb R\times\mathbb R} |y_1-y_2|^2 \pi(dy_1,dy_2)\big),
	\end{eqnarray*}
	and similarly,
	\begin{eqnarray*}
		\!\!\!\!\!\!\!\!&& -\langle\int_{\mathbb{R}}(2m-1)u^{2k-1}
		y_1^{2m-2}\mu(dy_1)-\int_{\mathbb{R}}(2m-1)v^{2k-1}
		y_2^{2m-2}\nu(dy_2),u-v\rangle\\
		\lesssim\!\!\!\!\!\!\!\!&&\big(\mu(|\cdot|^{4m-6})+\nu(|\cdot|^{4m-6})\big)|u-v|^2
		+|v|^{4k-2}\Big(\int_{\mathbb R\times\mathbb R} |y_1-y_2|^2 \pi(dy_1,dy_2)\Big).
	\end{eqnarray*}
	Taking infimum for couplings $\pi \in \mathscr{C}(\mu, \nu)$ leads to
	\begin{eqnarray*}
		\!\!\!\!\!\!\!\!&&\langle \mathcal{A}(u, \mu) - \mathcal{A}(v, \nu), u - v \rangle\\
		\lesssim\!\!\!\!\!\!\!\!&&\big(\mu(|\cdot|^{4(m+n)-6}) + \nu(|\cdot|^{4(m+n)-6}) \big) |u - v|^2 +  |v|^{4k-2} \mathcal{W}_{2,\mathbb{R}}(\mu, \nu)^2.
	\end{eqnarray*}
	Thus, $(\mathbf{A}_5'')$ holds.
	
	Finally, we can define the operator
	$$
	\tilde{\mathcal{A}}(u,y):=-u^{2n-1}-(2m-1)u^{2k-1}
	y^{2m-2}-u^{2k-1}y^{2(m+n)-2},
	$$
	which satisfies
	\begin{equation*}
		|\tilde{\mathcal{A}}(u,\mu)|^2
		\lesssim|u|^{4n-2}+|u|^{8k-4}+|y|^{8(m+n)-8}.
	\end{equation*}
	Thus, according to Remark \ref{remark1}, $(\mathbf{A}_6)$ is satisfied.	
	
	In summary, due to the $2$-uniform smoothness of space $\mathbb{R}$ and $\alpha=2$, we complete the proof.
\end{proof}

\begin{remark}
	(i) As far as we know, this seems to be  the first result deriving  the  quantitative propagation of chaos  for stochastic SVGD with  super-linear kernels, which should be of independent interest in the literature.

	(ii) Beyond the example in Subsection \ref{sec3.1}, our general framework provides a unified treatment of other important models, arising in machine learning and optimization theory, such as the ensemble Kalman sampler (cf.~e.g.~\cite{V}) and consensus-based optimization (cf.~e.g.~\cite{PTTM}).
	We omit the details to keep down the
	length of this paper.
\end{remark}

\subsection{Mean-field Allen-Cahn equations}
The Allen-Cahn equation, originally introduced in \cite{AC}, is  to describe phase separation and interface motion in binary alloys.
Recently,   mean-field Allen-Cahn equations have been derived in the context of stochastic quantization theory
(cf.~\cite{SSZZ})  and as a scaling limit of the classical Allen-Cahn equation with rough random initial data (cf.~\cite{GRZ}).

Inspired by these developments, we study the following weakly interacting Allen-Cahn equations defined on domain $\mathcal{O}: = [0, 1]$,
\begin{equation}\label{AC-IPS}
	\begin{cases}
		&\!\!\!\!\!\!dX^{i,N}_t = \big[ \Delta X^{i,N}_t + X^{i,N}_t - \sum_{j=1}^N (X^{j,N}_t)^2 X^{i,N}_t  \big] dt+dW_t^i,
		\\
		&\!\!\!\!\!\!X^{i,N}_t(0) = X^{i,N}_t(1) = 0, \\
		&\!\!\!\!\!\!X^{i,N}_0=\xi^i.
	\end{cases}
\end{equation}
Here, $\{\xi^i\}_{i\in\{1,\dots,N\}}$ and  $\{W^i\}_{i\in\{1,\dots,N\}}$ are i.i.d.-initial values and $Q$-Wiener processes, respectively, where $Q$ is a  non-negative, symmetric and  finite trace operator.

Considering the large $N$ limit, the microscopic dynamic is characterized by the  mean-field Allen-Cahn equations
\begin{equation}\label{AC-nIPS}
	\begin{cases}
		&\!\!\!\!\!\!dX^i_t = \big[ \Delta X^i_t + X^i_t -  X^i_t\mathbb{E}(X^i_t)^2 \big] dt+dW_t^i,
		\\
		&\!\!\!\!\!\!X^{i}_t(0) = X^{i}_t(1) = 0, \\
		&\!\!\!\!\!\!X^{i}_0=\xi^i.
	\end{cases}
\end{equation}

In this subsection, the problem is formulated based on the Gelfand triple
	\begin{equation}\label{gel}
\mathbb{V} \subset \mathbb{H} \subset \mathbb{V}^*,
\end{equation}
where $\mathbb{V} := W_0^{1,2}(\mathcal{O}; \mathbb{R})$ and $\mathbb{H} := L^2(\mathcal{O}; \mathbb{R})$.
We now establish the well-posedness of the mean-field Allen-Cahn equations and the quantitative PoC.
\begin{theorem}\label{ACrate}
	Suppose that the initial data $\xi^i\in L^p(\Omega,\mathscr{F}_0,\mathbb{P};\mathbb{H})$ with $p\geq24$. Then Eq.~$(\ref{AC-nIPS})$ has a
	unique strong and weak solution.
	Moreover, the following estimates hold
	\begin{equation*}
		\sup_{i \in \{1,\dots,N\}}\bigg\{\mathbb{E}\Big[\sup_{t \in [0,T]}\|X_t^{i,N} - X_t^i\|_{\mathbb{H}}^2\Big]+ \mathbb{E}\int_0^{T}\|X_s^i - X_s^{i,N}\|_{\mathbb{V}}^{2}ds\bigg\} \lesssim N^{-{\frac{p'-2}{p'}}},
	\end{equation*}
	for any $2<p'<\frac{p}{3}$, and
	\begin{equation*}
		\sup_{i \in \{1,\dots,N\}}  \sup_{t \in [0,T]}\mathbb{E}\|X_t^{i,N} - X_t^i\|_{\mathbb{H}}^2 \lesssim N^{-{\frac{p-2}{p}}}.
	\end{equation*}
\end{theorem}
\begin{proof}
	We show that the operator
	\begin{equation*}
		\mathcal{A}(u, \mu) := \Delta u + u - u\int y^2 \mu(dy)
	\end{equation*}
	satisfies conditions $(\mathbf{A}_2^*)$, $(\mathbf{A}_3)$, $(\mathbf{A}_4)$, $(\mathbf{A}''_5)$ and $(\mathbf{A}_6)$.
	
	Note that
	\begin{eqnarray*} _{\mathbb{V}^*}\langle\mathcal{A}(u,\mu),u\rangle_{\mathbb{V}}
		=\!\!\!\!\!\!\!\!&&-\|u\|_1^2 + \|u\|_{L^2}^2 - \int_{\mathcal{O}} u^2 \Big( \int y^2 \mu(dy) \Big) dx\\
		\leq\!\!\!\!\!\!\!\!&&-\|u\|_1^2 + \|u\|_{L^2}^2,
	\end{eqnarray*}
	which leads to $(\mathbf{A}_3)$ is satisfied.
	By Sobolev embedding  inequality and Gagliardo-Nirenberg interpolation inequality (cf. \cite{T-interpolation}), we obtain
	\begin{eqnarray*}
		\!\!\!\!\!\!\!\!&&| _{\mathbb{V}^*}\langle \mathcal{A}(u,\mu), v \rangle_{\mathbb{V}} |\\
		\leq\!\!\!\!\!\!\!\!&& |\int_{\mathcal{O}} \Delta u\cdot v dx|
		+ |\int_{\mathcal{O}} uv dx|
		+ \int_{\mathcal{O}} |uv| \Big( \int y^2\mu(dy) \Big) dx \\
		\lesssim\!\!\!\!\!\!\!\!&& \|u\|_1 \|v\|_1 + \|u\|_{L^2} \|v\|_{L^2} + \|v\|_{L^\infty} \|u\|_{L^2} \int \|y\|_{L^4}^2 \mu(dy)  \\
		\lesssim\!\!\!\!\!\!\!\!&& \|v\|_1 \big( \|u\|_1 + \|u\|_{L^2} \int ( \|y\|_1 + \|y\|_{L^2}^3 ) \mu(dy) \big)\\
		\lesssim\!\!\!\!\!\!\!\!&&\|v\|_1 \big( \|u\|_1 + \|u\|_{L^2}(\mu(\|\cdot\|_1)+\mu(\|\cdot\|_{L^2}^3))\big).
	\end{eqnarray*}
	Then
	\begin{equation*}
		\|\mathcal{A}(u,\mu)\|_{\mathbb{V}^*}^2 \lesssim  \|u\|_1^2 + \|u\|_{L^2}^2 ( \mu(\| \cdot \|_1^2) + \mu(\|\cdot\|_{L^2}^6) )
	\end{equation*}
	implies that $(\mathbf{A}_4)$ holds with $\alpha=2$ and $\beta=6$.
	
	Now we verify the demicontinuity and local monotonicity of $\mathcal{A}$.
	Firstly,
	\begin{eqnarray*}
		\!\!\!\!\!\!\!\!&& _{\mathbb{V}^*} \langle \mathcal{A}(u, \mu) - \mathcal{A}(v, \nu), w \rangle_{\mathbb{V}} \\
		\leq\!\!\!\!\!\!\!\!&& \|u - v\|_1 \|w\|_1 + |\int_{\mathcal{O}} (u - v)\cdot w \Big( \int y_1^2 \mu(dy_1) \Big) dx| \\
		\!\!\!\!\!\!\!\!&&  + \int_{\mathcal{O}} |v\cdot w| \Big( \int |y_1 - y_2|(|y_1|+|y_2|) \pi(dy_1, dy_2) \Big) dx \\
		\lesssim\!\!\!\!\!\!\!\!&& \|u - v\|_1 \|w\|_1 + \|u - v\|_{L^{\infty}} \|w\|_{L^\infty} \mu(\|\cdot\|_{L^2}^2) \\
		\!\!\!\!\!\!\!\!&&  +  \|v\|_{L^{\infty}} \|w\|_{L^\infty}\Big((\mu(\|\cdot\|_{L^2}^2))^{\frac{1}{2}}+(\nu(\|\cdot\|_{L^2}^2))^{\frac{1}{2}} \Big)\Big(\int \|y_1-y_2\|_{L^2}^2  \pi(dy_1, dy_2) \Big)^{\frac{1}{2}} .
	\end{eqnarray*}
	Taking infimum for couplings $\pi \in \mathscr{C}(\mu, \nu)$, we obtain
	\begin{eqnarray*}
		\!\!\!\!\!\!\!\!&&_{\mathbb{V}^*} \langle \mathcal{A}(u, \mu) - \mathcal{A}(v, \nu), w \rangle_{\mathbb{V}} \\
		\lesssim\!\!\!\!\!\!\!\!&& \|u - v\|_1 \|w\|_1 (1+\mu(\|\cdot\|_{L^2}^2))
		+\|v\|_{1} \|w\|_{1}\Big((\mu(\|\cdot\|_{L^2}^2))^{\frac{1}{2}}+(\nu(\|\cdot\|_{L^2}^2))^{\frac{1}{2}} \Big)\mathcal{W}_{2, L^2}(\mu, \nu),
	\end{eqnarray*}
	which implies  $(\mathbf{A}_2^*)$ holds.
	Next, we have
	\begin{eqnarray*}
		\!\!\!\!\!\!\!\!&&_{\mathbb{V}^*} \langle \mathcal{A}(u, \mu) - \mathcal{A}(v, \nu), u - v \rangle_{\mathbb{V}} \\
		=\!\!\!\!\!\!\!\!&& -\|u - v\|_1^2 + \|u - v\|_{L^2}^2 - \int_{\mathcal{O}} (u - v)^2 \Big( \int y_1^2 \mu(dy_1) \Big) dx \\
		\!\!\!\!\!\!\!\!&&  - \int_{\mathcal{O}} v\cdot(u - v) \Big( \int (y_1^2 - y_2^2) \pi(dy_1, dy_2) \Big) dx\\
		\leq\!\!\!\!\!\!\!\!&& -\|u - v\|_1^2+\|u - v\|_{L^2}^2
		+ C\int_{\mathcal{O}} \Big( |v|\cdot|(u - v)| \big( \int (y_1 - y_2)^2 \pi(dy_1, dy_2) \big)^{1/2} \\
		\!\!\!\!\!\!\!\!&& \cdot \big( \int y_1^2 \mu(dy_1) + \int y_2^2 \nu(dy_2) \big)^{1/2} \Big) dx \\
		\leq\!\!\!\!\!\!\!\!&& -\|u - v\|_1^2+\|u - v\|_{L^2}^2
		+C \int_{\mathcal{O}} (u - v)^2 \Big( \int y_1^2 \mu(dy_1) + \int y_2^2 \nu(dy_2) \Big) dx \\
		\!\!\!\!\!\!\!\!&& +C \int_{\mathcal{O}} v^2 \Big( \int (y_1 - y_2)^2 \pi(dy_1, dy_2) \Big) dx \\
		\leq\!\!\!\!\!\!\!\!&& -\|u - v\|_1^2+C\|u - v\|_{L^2}^2 \Big( 1 + \int \|y_1\|_{L^\infty}^2 \mu(dy_1) + \int \|y_2\|_{L^\infty}^2 \nu(dy_2) \Big) \\
		\!\!\!\!\!\!\!\!&& +C\|v\|_{L^\infty}^2 \int \|y_1 - y_2\|_{L^2}^2 \pi(dy_1, dy_2) \\
		\leq\!\!\!\!\!\!\!\!&& -\|u - v\|_1^2 +C( 1 + \mu(\|\cdot\|_1^2) + \nu(\|\cdot\|_1^2) )\|u - v\|_{L^2}^2
		\\
		\!\!\!\!\!\!\!\!&&
		+ C\|v\|_1^2 \int \|y_1 - y_2\|_{L^2}^2 \pi(dy_1, dy_2).
	\end{eqnarray*}
	Taking infimum for couplings $\pi \in \mathscr{C}(\mu, \nu)$, we obtain
	\begin{eqnarray*}
		\!\!\!\!\!\!\!\!&& _{\mathbb{V}^*} \langle \mathcal{A}(u, \mu) - \mathcal{A}(v, \nu), u - v \rangle_{\mathbb{V}}+\|u - v\|_1^2 \\
		\!\!\!\!\!\!\!\!&& \leq C( 1 + \mu(\|\cdot\|_1^2) + \nu(\|\cdot\|_1^2) )\|u - v\|_{L^2}^2 + C\|v\|_1^2 \mathcal{W}_{2, L^2}(\mu, \nu)^2,
	\end{eqnarray*}
	which leads to $(\mathbf{A}''_5)$.
	
	Finally, we define the operator
	$$
	\tilde{\mathcal{A}}(u,y):=\Delta u + u - u y^2.
	$$
	It is easy to get that
	\begin{eqnarray*}
		\|\tilde{\mathcal{A}}(t,u,y)\|_{{\mathbb{V}}^*}^2
		\lesssim\!\!\!\!\!\!\!\!&&  \|u\|_1^2 + \|u\|_{L^2}^2 ( \|y\|_1^2 + \|y\|_{L^2}^6),
	\end{eqnarray*}
	and according to Remark \ref{remark1}, $(\mathbf{A}_6)$ is satisfied.	
	
	In summary, due to the $2$-uniform smoothness of space $W_0^{1,2}(\mathcal{O}; \mathbb{R})$ and $\alpha=2$, we complete the proof.
\end{proof}

\subsection{Lagrangian averaged Burgers equations}
Within the Lagrangian averaged  stochastic  advection by Lie transport (LA-SALT) framework proposed by Drivas, Holm, and Leahy \cite{Drivas2020}, the Lagrangian averaged Burgers equation has been derived as a foundational one-dimensional model. The key  feature compared to the classical Burgers equations is a modification of the nonlinear advection term, wherein the stochastic velocity field $u$ is transported not by itself, but by its own expectation $\mathbb{E}u$. This introduces a non-locality in probability space, analogous to mean-field dynamics.

Motivated by the work \cite{Drivas2020}, we consider the following weakly interacting particle systems defined on domain $\mathcal{O}: = [0, 1]$,
\begin{eqnarray}\label{Burgers-IPS}
	\left\{ \begin{aligned}
		&dX_t^{i,N}= \Big[\Delta X_t^{i,N}-\frac{1}{N} \sum_{j=1}^N \psi (X_t^{j,N})\cdot \partial_x X_t^{i,N} \Big] dt + dW_t^i,\\
		&X_{t}^{N, i}(0)=X_{t}^{N, i}(1)=0,\\
		&X_0^{i,N}=\xi^i,
	\end{aligned} \right.
\end{eqnarray}
where $\psi\in C_{b}^1(\mathbb{R}; \mathbb{R})$, $\{\xi^i\}_{i\in\{1,\dots,N\}}$ and  $\{W^i\}_{i\in\{1,\dots,N\}}$ are i.i.d.-initial values and $Q$-Wiener processes, respectively.

Regarding the large $N$ limit, the microscopic dynamic is characterized by the variant of Lagrangian averaged Burgers equations
\begin{eqnarray}\label{Burgers-nIPS}
	\left\{ \begin{aligned}
		&dX_t^{i} = \big[\Delta X_t^{i}-(\mathbb{E}\psi(X_t^{i}))\cdot \partial_x X_t^{i} \big] dt + dW_t^i,
\\
		&X_{t}^{i}(0)=X_{t}^{i}(1)=0,\\
		&X_0^{i}=\xi^i.
	\end{aligned} \right.
\end{eqnarray}

We formulate this problem also within the Gelfand triple (\ref{gel}).
We now establish the well-posedness of  Lagrangian averaged Burgers equations (\ref{Burgers-nIPS}) and the quantitative PoC.
\begin{theorem}\label{Burgersrate}
Suppose that the initial data $\xi^i\in L^p(\Omega,\mathscr{F}_0,\mathbb{P};\mathbb{H})$ with $p>8$. Then Eq.~$(\ref{Burgers-nIPS})$ has a
	unique strong and weak solution.
	Moreover, the following estimates hold
	\begin{equation*}
		\sup_{i \in \{1,\dots,N\}}\bigg\{\mathbb{E}\Big[\sup_{t \in [0,T]}\|X_t^{i,N} - X_t^i\|_{\mathbb{H}}^2\Big]+ \mathbb{E}\int_0^{T}\|X_s^i - X_s^{i,N}\|_{\mathbb{V}}^{2}ds\bigg\} \lesssim N^{-{\frac{p'-2}{p'}}},
	\end{equation*}
	for any $2<p'<p$, and
	\begin{equation*}
		\sup_{i \in \{1,\dots,N\}}  \sup_{t \in [0,T]}\mathbb{E}\|X_t^{i,N} - X_t^i\|_{\mathbb{H}}^2 \lesssim N^{-{\frac{p-2}{p}}}.
	\end{equation*}
\end{theorem}
\begin{proof}
	Set
	\begin{equation*}
		\mathcal{A}(u,\mu):=\Delta u-\int \psi(y) \mu(dy)\cdot \partial_x u.
	\end{equation*}
	We aim to prove that $\mathcal{A}$ satisfies conditions $(\mathbf{A}_2^*)$, $(\mathbf{A}_3)$, $(\mathbf{A}_4)$, $(\mathbf{A}''_5)$, and $(\mathbf{A}_6)$.
	
	It is easy to get that
	\begin{eqnarray*}
		_{\mathbb{V}^*}\langle\mathcal{A}(u,\mu),u\rangle_{\mathbb{V}}
		\leq\!\!\!\!\!\!\!\!&&-\|u\|_1^2+\|\psi\|_{\infty}\int_{\mathcal{O}} |u\cdot\partial_x u|  dx\\
		\leq\!\!\!\!\!\!\!\!&&-\frac{1}{2}\|u\|_1^2 + C\|u\|_{L^2}^2,
	\end{eqnarray*}
	and
	\begin{eqnarray*}
		|_{\mathbb{V}^*}\langle \mathcal{A}(u,\mu), v \rangle_{\mathbb{V}} |
		\leq\!\!\!\!\!\!\!\!&& |\int_{\mathcal{O}} \Delta u\cdot v dx|
		+ \int_{\mathcal{O}} \partial_xu\cdot v\Big( \int \psi(y)\mu(dy) \Big) dx |
		\\
		\lesssim\!\!\!\!\!\!\!\!&& \|u\|_1 \|v\|_1 + \int_{\mathcal{O}} |\partial_xv\cdot u|dx \\
		\lesssim\!\!\!\!\!\!\!\!&& \|v\|_1(\|u\|_1+\|u\|_{L^2}),
	\end{eqnarray*}
	then
	\begin{equation*}
		\|\mathcal{A}(u,\mu)\|_{\mathbb{V}^*}^2 \lesssim  \|u\|_1^2+\|u\|^2_{L^2},
	\end{equation*}
	which imply $(\mathbf{A}_3)$ and $(\mathbf{A}_4)$ are satisfied with $\alpha=2$ and $\beta=2$.
	
	Now we show the demicontinuity and local monotonicity of $\mathcal{A}$.
	Firstly, by Sobolev embedding inequality we have
	\begin{eqnarray*}
		\!\!\!\!\!\!\!\!&& _{\mathbb{V}^*} \langle \mathcal{A}(u, \mu) - \mathcal{A}(v, \nu), w \rangle_{\mathbb{V}} \\
		\leq\!\!\!\!\!\!\!\!&& \|u - v\|_1 \|w\|_1 + |\int_{\mathcal{O}} \partial_x(u - v)\cdot w \Big( \int \psi(y_1) \mu(dy_1) \Big) dx| \\
		\!\!\!\!\!\!\!\!&&  + \int_{\mathcal{O}} |\partial_xv\cdot w| \Big( \int |\psi(y_1) - \psi(y_2)| \pi(dy_1, dy_2) \Big) dx \\
		\lesssim\!\!\!\!\!\!\!\!&& \|u - v\|_1 \|w\|_1 + \|u - v\|_{1} \|w\|_{L^2} \\
		\!\!\!\!\!\!\!\!&&  +  \|v\|_{1} \|w\|_{L^\infty}\Big(\int \|y_1-y_2\|_{L^2}^2 \pi(dy_1, dy_2) \Big)^{\frac{1}{2}} \\
		\lesssim\!\!\!\!\!\!\!\!&& \|u - v\|_1 \|w\|_1 +\|v\|_{1} \|w\|_{1}\Big(\int \|y_1-y_2\|_{L^2}^2 \pi(dy_1, dy_2) \Big)^{\frac{1}{2}}.
	\end{eqnarray*}
	Taking infimum for couplings $\pi \in \mathscr{C}(\mu, \nu)$, we obtain
	\begin{eqnarray*}
		\!\!\!\!\!\!\!\!&&_{\mathbb{V}^*} \langle \mathcal{A}(u, \mu) - \mathcal{A}(v, \nu), w \rangle_{\mathbb{V}} \\
		\lesssim\!\!\!\!\!\!\!\!&& \|u - v\|_1 \|w\|_1+ \|v\|_1 \|w\|_1\mathcal{W}_{2, L^2}(\mu, \nu).
	\end{eqnarray*}
	Thus, $(\mathbf{A}_2^*)$ holds.
Similarly,
	\begin{eqnarray*}
		\!\!\!\!\!\!\!\!&&_{\mathbb{V}^*} \langle \mathcal{A}(u, \mu) - \mathcal{A}(v, \nu), u - v \rangle_{\mathbb{V}} \\
		=\!\!\!\!\!\!\!\!&& -\|u - v\|_1^2+_{\mathbb{V}^*}\langle \big( \int \psi(y_{1})\mu(dy_{1})- \int \psi(y_{2})\nu(dy_{2}) \big)\partial_{x} v , u - v \rangle_{\mathbb{V}} \\
		\!\!\!\!\!\!\!\!&&+ _{\mathbb{V}^*}\langle \int \psi(y_{1})\mu(dy_{1})\partial_{x}(u - v),u - v\rangle_{\mathbb{V}}\\
		\leq \!\!\!\!\!\!\!\!&&-\|u - v\|_1^2+\|u - v\|_{L^\infty}\int_{\mathcal{O}} \int |y_{1} - y_{2}|\cdot|\partial_{x} v|\pi(dy_{1}, dy_{2}) dx\\
		\!\!\!\!\!\!\!\!&&+C \|u - v\|_{1}\|u - v\|_{L^{2}}\\
		\leq \!\!\!\!\!\!\!\!&&-\frac{1}{2}\|u - v\|_1^2+C\|v\|_{1}^{2}\int \|y_{1} - y_{2}\|_{L^{2}}^{2}\pi(dy_{1}, dy_{2})+C\|u - v\|_{L^{2}}^{2} .
	\end{eqnarray*}
	Taking infimum for couplings $\pi \in \mathscr{C}(\mu, \nu)$, we obtain
	\begin{eqnarray*}
		\!\!\!\!\!\!\!\!&& _{\mathbb{V}^*} \langle \mathcal{A}(u, \mu) - \mathcal{A}(v, \nu), u - v \rangle_{\mathbb{V}}+\|u - v\|_1^2 \\
		\leq\!\!\!\!\!\!\!\!&& -\frac{1}{2}\|u - v\|_1^2+C\|u - v\|_{L^2}^2 + C\|v\|_1^2 \mathcal{W}_{2, L^2}(\mu, \nu)^2,
	\end{eqnarray*}
	which implies $(\mathbf{A}''_5)$ holds.
	
	Finally, we show that the condition $(\mathbf{A}_6)$ is satisfied.
	We define the operator
	$$
	\tilde{\mathcal{A}}(u,y):=\Delta u - \psi(y) \cdot \partial_x u.
	$$
	Note that
	\begin{eqnarray*}
		|_{\mathbb{V}^*}\langle \tilde{\mathcal{A}}(u,y), v \rangle_{\mathbb{V}} |
		\leq |\int_{\mathcal{O}} \Delta u\cdot v dx|
		+ |\int_{\mathcal{O}} \partial_xu\cdot v\psi(y) dx |
		\lesssim \|v\|_1\|u\|_1(1+\|y\|_{1}),
	\end{eqnarray*}
	then
	\begin{equation*}
		\|\tilde{\mathcal{A}}(u,y)\|_{\mathbb{V}^*}^2 \lesssim  \|u\|_1^2,
	\end{equation*}
and according to Remark \ref{remark1}, $(\mathbf{A}_6)$ is satisfied.	
\end{proof}

\section{Proof of well-posedness}\label{well-posed}
In Subsection \ref{sec3.1}, we present an equivalent definition of measure-dependent pseudo monotonicity in the sense of Definition \ref{deps}. In Subsection \ref{finite_dim}, we establish a general result for the existence of weak solutions to mean-field SDEs by introducing a measure cut-off technique.   In Subsection \ref{sec4.3}, we establish some uniform energy
bounds and study the tightness of approximating sequences.  In Subsections \ref{sec3.4} and \ref{sec4.5}, we aim to prove the main theorems.

\subsection{An equivalent characterization of pseudo monotonicity}\label{sec3.1}

We present an equivalent lemma about the pseudo monotonicity in the sense of Definition \ref{deps}. Throughout the work, we use the notation ``$\rightharpoonup$'' to denote the weak convergence in Banach spaces.
\begin{lemma}\label{pseudoequivalent}
	For some  constants $\alpha>1$ and $\beta> 0$,  the operator $\mathcal{A}:\mathbb{V}\times\mathfrak{M}\to\mathbb{V}^*$ is  pseudo-monotone (in the sense of Definition \ref{deps}) iff for any sequences $\{u_n\}_{n=1}^{\infty},u$ in $\mathbb{V}$ and  $\{\mu_n\}_{n=1}^{\infty},\mu$ in $\mathfrak{M}_b$ such that   $u_n\to u$ weakly in $\mathbb{V}$ and  $\mu_n\to\mu$ in $\mathscr{P}_{\beta}(\mathbb{H})\cap\mathscr{P}_{2}(\mathbb{H})$,  and
	$$\liminf _{n \rightarrow \infty}\,_{\mathbb{V}^*}\langle \mathcal{A}(u_{n},\mu_n), u_{n}-u\rangle_{\mathbb{V}} \geq 0,$$
	then $\mathcal{A}(u_n,\mu_n)\to \mathcal{A}(u,\mu)$ weakly in $\mathbb{V}^*$ and
	$$\lim _{n \rightarrow \infty}\,_{\mathbb{V}^*}\langle \mathcal{A}(u_{n},\mu_n), u_{n}\rangle_{\mathbb{V}}=\,_{\mathbb{V}^*}\langle \mathcal{A}(u,\mu), u\rangle_{\mathbb{V}}.$$
\end{lemma}

\begin{proof}
	Clearly, the above definition directly implies the pseudo monotonicity of the operator $\mathcal{A}$. Conversely, if $\mathcal{A}$ is  pseudo-monotone and $u_n\rightharpoonup u$ in $\mathbb{V}$, $\mu_n\to\mu$ in $\mathscr{P}_{\beta}(\mathbb{H})\cap\mathscr{P}_{2}(\mathbb{H})$ such that
	$$\liminf _{n \rightarrow \infty}\,_{\mathbb{V}^*}\langle \mathcal{A}(u_{n},\mu_n), u_{n}-u\rangle_{\mathbb{V}} \geq 0,$$
	then we claim that
	\begin{equation}\label{eq1}
		\limsup_{n \rightarrow \infty}\,_{\mathbb{V}^*}\langle \mathcal{A}(u_{n},\mu_n), u_{n}-u\rangle_{\mathbb{V}} \leq 0,
	\end{equation}
	whose proof will be given later. Therefore, it follows that
	\begin{equation}\label{eq3}
		\lim_{n \rightarrow \infty}\,_{\mathbb{V}^*}\langle \mathcal{A}(u_{n},\mu_n), u_{n}-u\rangle_{\mathbb{V}}= 0.
	\end{equation}
	For each $v\in\mathbb{V}$,
	\begin{eqnarray*}
		\!\!\!\!\!\!\!\!&&\limsup_{n \rightarrow \infty}\,_{\mathbb{V}^*}\langle \mathcal{A}(u_{n},\mu_n), u-v\rangle_{\mathbb{V}}
		\nonumber \\
		=\!\!\!\!\!\!\!\!&&\limsup_{n \rightarrow \infty}\,_{\mathbb{V}^*}\langle \mathcal{A}(u_{n},\mu_n), u-v\rangle_{\mathbb{V}}+
		\lim_{n \rightarrow \infty}\,_{\mathbb{V}^*}\langle \mathcal{A}(u_{n},\mu_n), u_{n}-u\rangle_{\mathbb{V}}
		\nonumber \\
		=\!\!\!\!\!\!\!\!&&\limsup_{n \rightarrow \infty}\,_{\mathbb{V}^*}\langle \mathcal{A}(u_{n},\mu_n), u_n-v\rangle_{\mathbb{V}}
		\nonumber \\
		\leq\!\!\!\!\!\!\!\!&&_{\mathbb{V}^*}\langle \mathcal{A}(u,\mu), u-v\rangle_{\mathbb{V}}.
	\end{eqnarray*}
	Thus for each $w\in\mathbb{V}$,
	\begin{equation}\label{eq2}
		\limsup_{n \rightarrow \infty}\,_{\mathbb{V}^*}\langle \mathcal{A}(u_{n},\mu_n), w\rangle_{\mathbb{V}}\leq{}
		_{\mathbb{V}^*}\langle \mathcal{A}(u,\mu), w\rangle_{\mathbb{V}}.
	\end{equation}
	Replacing $w$ by $-w$ in (\ref{eq2}), it follows that $\mathcal{A}(u_n,\mu_n)\rightharpoonup\mathcal{A}(u,\mu)$ in $\mathbb{V}^*$.
	Furthermore, by (\ref{eq3}) we can obtain that
	\begin{eqnarray*}
		\!\!\!\!\!\!\!\!&&\lim _{n \rightarrow \infty}\,_{\mathbb{V}^*}\langle \mathcal{A}(u_{n},\mu_n), u_{n}\rangle_{\mathbb{V}}
		\nonumber \\
		=\!\!\!\!\!\!\!\!&&\lim _{n \rightarrow \infty}\,_{\mathbb{V}^*}\langle \mathcal{A}(u_{n},\mu_n), u_{n}-u\rangle_{\mathbb{V}}
		+\lim _{n \rightarrow \infty}\,_{\mathbb{V}^*}\langle \mathcal{A}(u_{n},\mu_n), u\rangle_{\mathbb{V}}
		\nonumber \\
		=\!\!\!\!\!\!\!\!&&\,_{\mathbb{V}^*}\langle \mathcal{A}(u,\mu), u\rangle_{\mathbb{V}}.
	\end{eqnarray*}
	Thus the assertion follows.
	
	\vspace{1mm}
	Now we turn to prove (\ref{eq1}). Assume that (\ref{eq1}) is not true, then we can choose a subsequence $\{n_k\}$ such that $u_{n_k}\rightharpoonup u$ in $\mathbb{V}$, $\mu_{n_k}\to\mu$ in $\mathscr{P}_{\beta}(\mathbb{H})\cap\mathscr{P}_{2}(\mathbb{H})$ and
	\begin{equation}\label{eq4}
		\liminf_{n \rightarrow \infty}\,_{\mathbb{V}^*}\langle \mathcal{A}(u_{n_k},\mu_{n_k}), u_{n_k}-u\rangle_{\mathbb{V}}>0.
	\end{equation}
	Since $\mathcal{A}$ is  pseudo-monotone, it follows that
	\begin{eqnarray*}
		0=\!\!\!\!\!\!\!\!&& \,_{\mathbb{V}^*}\langle \mathcal{A}(u,\mu), u-u\rangle_{\mathbb{V}}
		\nonumber \\
		\geq\!\!\!\!\!\!\!\!&&\limsup_{k \rightarrow \infty}\,_{\mathbb{V}^*}\langle \mathcal{A}(u_{n_k},\mu_{n_k}), u_{n_k}-u\rangle_{\mathbb{V}}
		\nonumber \\
		\geq\!\!\!\!\!\!\!\!&&\liminf_{k \rightarrow \infty}\,_{\mathbb{V}^*}\langle \mathcal{A}(u_{n_k},\mu_{n_k}), u_{n_k}-u\rangle_{\mathbb{V}},
	\end{eqnarray*}
	which contradicts with (\ref{eq4}). We complete the proof.
\end{proof}

\begin{remark}\label{re2}
	We mention that in finite-dimensional case (i.e.,~$\mathbb{V}=\mathbb{H}=\mathbb{R}^d$), if the map
	$$\mathcal{A}:\mathbb{R}^d\times \mathscr{P}_{p}(\mathbb{R}^d)\to\mathbb{R}^d$$ is locally bounded for some $p>0$,  then the pseudo monotonicity is equivalent to the continuity.
	
\end{remark}

\subsection{Existence of weak solutions in finite dimensions}
\label{finite_dim}
This subsection is devoted to establishing the existence of weak solutions to the mean-field equations in a finite-dimensional setting (i.e.,~$U = \mathbb{H} = \mathbb{V} = \mathbb{V}^* = \mathbb{R}^d$).

Consider the following mean-field equations on $\mathbb{R}^d$
\begin{equation}\label{finite-eq}
	dX_t = \mathcal{A}(t, X_t, \mathscr{L}_{X_t})dt + \mathcal{B}(t, X_t, \mathscr{L}_{X_t})dW_t,
\end{equation}
where $W_t$ is an $\mathbb{R}^d$-valued standard Wiener process defined on $(\Omega, \mathscr{F}, \{\mathscr{F}_t\}_{t \geq 0}, \mathbb{P})$.
Throughout this subsection, for the coefficients of Eq.~(\ref{finite-eq}) we suppose that there are some constants $\kappa \geq 2$ and $p> \kappa$ such that for a.e.~$t \in [0,T]$ the following assumptions hold.
\vspace{1mm}
\begin{enumerate}
	\item [$(\mathbf{H}_1)$]
	The maps $\mathcal{A}(t,\cdot,\cdot), \mathcal{B}(t,\cdot,\cdot)$ are continuous on $\mathbb{R}^d \times \mathscr{P}_2(\mathbb{R}^d)$.
	\vspace{1mm}
	\item [$(\mathbf{H}_2)$]
	For any $u \in \mathbb{R}^d$ and $\mu \in \mathscr{P}_2(\mathbb{R}^d)$,
	\begin{equation*}
		2\langle \mathcal{A}(t,u,\mu),u\rangle+(p-1)\|\mathcal{B}(t,u,\mu)\|^2
		\lesssim 1+|u|^2+\mu(|\cdot|^2).
	\end{equation*}
	
	\vspace{1mm}
	\item [$(\mathbf{H}_3)$]
	For any $u \in \mathbb{R}^d$ and $\mu \in \mathscr{P}_\kappa(\mathbb{R}^d)$,
	\begin{equation*}
		|\mathcal{A}(t, u, \mu)|^2+\|\mathcal{B}(t, u, \mu)\|^2 \lesssim1 + |u|^\kappa + \mu(|\cdot|^\kappa).
	\end{equation*}
\end{enumerate}

The main result of this subsection is the following existence theorem.
\begin{theorem}\label{finite-th1}
	Suppose that $(\mathbf{H}_1)$-$(\mathbf{H}_3)$ hold. Then for any $X_0\sim\mu_0\in \mathscr{P}_p(\mathbb{R}^d)$, where $p$ is the same as in $(\mathbf{H}_2)$,
	Eq.~$(\ref{finite-eq})$ has a weak solution $(\Omega, \mathscr{F}, \{\mathscr{F}_t\}_{t \geq 0}, \mathbb{P}; X, W)$. Moreover,  for any $\gamma \in (0,1)$ we have
	\begin{equation}\label{finite-esti}
		\mathbb{E} \Big[ \sup_{t\in[0,T]}|X_t|^{\gamma p} \Big]  \lesssim_{p,T} 1 + \mu_0(|\cdot|^p).
	\end{equation}
\end{theorem}

To construct a weak solution of Eq.~(\ref{finite-eq}), we  employ a localization approximation approach by utilizing an appropriately chosen cut-off function. Specifically, for each $n \in \mathbb{N}$, we define
\begin{equation*}
	\chi_n(u) := \frac{nu}{n \vee |u|}, \quad u \in \mathbb{R}^d,
\end{equation*}
and the truncated coefficients
\begin{equation*}
	\mathcal{A}^n(t,u,\mu): = \mathcal{A}(t,\chi_n(u),\mu \circ \chi_n^{-1}), \quad \mathcal{B}^n(t,u,\mu): = \mathcal{B}(t,\chi_n(u),\mu \circ \chi_n^{-1}), \quad t \in [0,T].
\end{equation*}

The following lemma presents the coercivity of the truncated coefficients.
\begin{lemma}
	There exists a constant $C > 0$ (independent of $n$) such that for all $t \in [0,T]$, $u \in \mathbb{R}^d$, and $\mu \in \mathscr{P}_2 (\mathbb{R}^d)$,
	\begin{equation}\label{finite-A2}
		2\langle \mathcal{A}^n(t,u,\mu),u\rangle+(p-1)\|\mathcal{B}^n(t,u,\mu)\|^2
		\leq C(1+|u|^2+\mu(|\cdot|^2)),
	\end{equation}
	where $p$ is the same as in $(\mathbf{H}_2)$.
\end{lemma}
\begin{proof}
	It is noteworthy that by $(\mathbf{H}_2)$ we have
	\begin{eqnarray*}
		\!\!\!\!\!\!\!\!&&2\langle \mathcal{A}^n(t,u,\mu),u\rangle+(p-1)\|\mathcal{B}^n(t,u,\mu)\|^2
		\\
		\leq\!\!\!\!\!\!\!\!&&\frac{n \vee |u|}{n}\big(2\langle\mathcal{A}(t,\chi_n(u), \mu \circ \chi_n^{-1}), \chi_n(u) \rangle+(p-1)\|\mathcal{B}(t,\chi_n(u), \mu \circ \chi_n^{-1})\|^2\big)
		\\
		\lesssim\!\!\!\!\!\!\!\!&&  \frac{n \vee |u|}{n} (1 + |\chi_n(u)|^2 + \mu \circ \chi_n^{-1}(|\cdot|^2)),
	\end{eqnarray*}
	where
	\begin{equation}\label{es0}
		\mu \circ \chi_n^{-1}(|\cdot|^2) = \int_{\mathbb{R}^d} |\chi_n(y)|^2 \mu(dy)= \int_{\mathbb{R}^d} \frac{n^2|y|^2}{(n \vee |y|)^2} \mu(dy).
	\end{equation}
	From (\ref{es0}), it is easy to see that
	\begin{equation}\label{es1}
		\mu \circ \chi_n^{-1}(|\cdot|^2) \leq \mu(|\cdot|^2)
	\end{equation}
	and
	\begin{equation}\label{es2}
		\mu \circ \chi_n^{-1}(|\cdot|^2) \leq n^2.
	\end{equation}

	We proceed by considering the following two cases.
	When $1\leq n < |u|$,  it follows from (\ref{es2}) that
	\begin{eqnarray*}
		\!\!\!\!\!\!\!\!&&2\langle \mathcal{A}^n(t,u,\mu),u\rangle+(p-1)\|\mathcal{B}^n(t,u,\mu)\|^2
		\\
		\!\!\!\!\!\!\!\!&&\lesssim  \frac{|u|}{n} (1 + 2n^2)
		\lesssim  \frac{|u|}{n}+2n|u|
		\lesssim1 + |u|^2.
	\end{eqnarray*}
	When $n \geq |u|$, it follows from (\ref{es1}) that
	\begin{equation*}
		2\langle \mathcal{A}^n(t,u,\mu),u\rangle+(p-1)\|\mathcal{B}^n(t,u,\mu)\|^2
		\lesssim  1 + |u|^2+\mu(|\cdot|^2).
	\end{equation*}
	In conclusion,  (\ref{finite-A2}) holds where $C > 0$ is independent of $n$. The proof  is  completed.
\end{proof}

Setting $X_t^0 \equiv X_0$. For each $n \in \mathbb{N}$, we consider the following approximating equation
\begin{equation}\label{finite-appro}
	X_t^n = X_0^n + \int_0^t \mathcal{A}^n(s,X_s^n,\mathscr{L}_{X_s^n}) ds + \int_0^t \mathcal{B}^n(s,X_s^n,\mathscr{L}_{X_s^n}) dW_s,~ X_0^n\sim\mu_0.
\end{equation}

Note that the assumptions $(\mathbf{H}_1)$-$(\mathbf{H}_3)$  imply that the truncated coefficients $\mathcal{A}^n$ and $\mathcal{B}^n$ are continuous, for each $t\in[0,T]$, and bounded (with the bounds depend on $n$).
Thus, it follows from \cite{HSS}   that Eq.~(\ref{finite-appro}) has a weak solution, denoted as
$$(\Omega^n, \mathscr{F}^n, \{\mathscr{F}_t^n\}_{t \geq 0}, \mathbb{P}^n; X^n, W^n)~~\text{with}~~~~ \mathbb{P}^n \circ (X_0^n)^{-1} = \mu_0.$$

The following lemma provides uniform moment estimates and time H\"older continuity estimates of the approximating solutions.
\begin{lemma}\label{finite-unitime}
	For any $\gamma \in (0,1)$, we have
	\begin{equation}\label{finite-approesti}
		\sup_{n \in \mathbb{N}}\mathbb{E}^n \Big[\sup_{t\in[0,T]}|X_t^n|^{\gamma p}\Big] \lesssim_{p,T} 1 + \mu_0(|\cdot|^p),
	\end{equation}
	where $p$ is the same as in $(\mathbf{H}_2)$.
	Moreover, for any $t, s \in [0,T]$ and $l > 2$ we have
	\begin{equation}\label{finite-time}
		\mathbb{E}^n |X_t^n - X_s^n|^l \lesssim_{p,T} |t - s|^{\frac{l}{2}}.
	\end{equation}
\end{lemma}
\begin{proof}
	Applying It\^{o}'s formula to $|X_t^n|^p$ and using (\ref{finite-A2}) and Young's inequality, we derive
	\begin{eqnarray}\label{finite-Xall}
		\!\!\!\!\!\!\!\!&&\quad |X_t^n|^p-|X_0^n|^p\nonumber\\
		\!\!\!\!\!\!\!\!&&=\frac{p}{2} \int_{0}^{t}|X_s^n|^{p-2}\big(2\langle \mathcal{A}^n(s, X^n_s, \mathscr{L}_{X^n_s}), X^n_s\rangle+\|\mathcal{B}^n(s, X^n_s, \mathscr{L}_{X^n_s})\|^{2}\big) d s\nonumber\\
		\!\!\!\!\!\!\!\!&&\quad+\frac{p(p-2)}{2}\int_{0}^{t}|X_s^n|^{p-4}|( \mathcal{B}^n(s, X^n_s, \mathscr{L}_{X^n_s}) )^*X^n_s|^{2} d s\nonumber\\
		\!\!\!\!\!\!\!\!&&\quad+p\int_{0}^{t}|X_s^n|^{p-2}\langle X^n_s, \mathcal{B}^n(s, X^n_s, \mathscr{L}_{X^n_s}) d W_{s}\rangle\nonumber\\
		\!\!\!\!\!\!\!\!&&\leq\frac{p}{2} \int_{0}^{t}|X_s^n|^{p-2}\big[2\langle \mathcal{A}^n(s, X^n_s, \mathscr{L}_{X^n_s}), X^n_s\rangle+(p-1)\|\mathcal{B}^n(s, X^n_s, \mathscr{L}_{X^n_s})\|^{2}\big]ds
		\nonumber\\
		\!\!\!\!\!\!\!\!&&\quad+p\int_{0}^{t}|X_s^n|^{p-2}\langle X^n_s, \mathcal{B}^n(s, X^n_s, \mathscr{L}_{X^n_s}) d W_{s}\rangle\nonumber\\
		\!\!\!\!\!\!\!\!&&\lesssim_p \int_{0}^{t} (1+|X_s^n|^{2}+\mathbb{E}^n|X_s^n|^{2})|X_s^n|^{p-2}ds+p\int_{0}^{t}|X_s^n|^{p-2}\langle X^n_s, \mathcal{B}^n(s, X^n_s, \mathscr{L}_{X^n_s}) d W_{s}\rangle\nonumber\\
		\!\!\!\!\!\!\!\!&&\lesssim_p \int_{0}^{t} (1+|X_s^n|^{p}+\mathbb{E}^n|X_s^n|^{p})ds+p\int_{0}^{t}|X_s^n|^{p-2}\langle X^n_s, \mathcal{B}^n(s, X^n_s, \mathscr{L}_{X^n_s}) d W_{s}\rangle.
	\end{eqnarray}
	Taking the expectation on both sides of (\ref{finite-Xall}), we have
	\begin{equation*}
		\mathbb{E}^n|X_t^n|^p\lesssim_{p,T}(1+\mathbb{E}^n|X_0^n|^p)+\mathbb{E}^n\int_{0}^{t}|X_s^n|^{p}ds,
	\end{equation*}
	which yields
	\begin{equation}\label{finite-supEX}
		\sup_{t\in [0,T]}\mathbb{E}^n|X_t^n|^p \lesssim_{p,T} 1 + \mathbb{E}^n|X_0^n|^p.
	\end{equation}
	
	Let $\tau_0$ be a bounded stopping time with $\tau_0\leq T.$ Recalling (\ref{finite-Xall}), we also derive that for any $t \in [0,T]$,
	\begin{eqnarray}
		\mathbb{E}^n|X_{t\wedge \tau_0}^n|^p
		\!\!\!\!\!\!\!\!&&\lesssim_p\mathbb{E}^n\int_{0}^{t\wedge \tau_0} (1+|X_s^n|^{p}+\mathbb{E}^n|X_s^n|^{p})ds+\mathbb{E}^n|X_0^n|^p\nonumber \\
		\!\!\!\!\!\!\!\!&&\lesssim_{p,T}(1+\mathbb{E}^n|X_0^n|^p)+\mathbb{E}^n\int_{0}^{t\wedge \tau_0}|X_s^n|^{p}ds+\int_{0}^{T}\mathbb{E}^n|X_s^n|^{p}ds\nonumber \\
		\!\!\!\!\!\!\!\!&&\lesssim_{p,T}(1+\mathbb{E}^n|X_0^n|^p)+\mathbb{E}^n\int_{0}^{t}|X_{s\wedge \tau_0}^n|^{p}ds,
	\end{eqnarray}
	where we used (\ref{finite-supEX}) in the last step.
	Then by applying Gronwall's lemma, we deduce for any $t \in [0,T]$,
	\begin{equation*}
		\mathbb{E}^n|X_{t\wedge \tau_0}^n|^p
		\lesssim_{p,T} 1 + \mathbb{E}^n|X_0^n|^p = 1 + \mu_0(|\cdot|^p) .
	\end{equation*}
	Using Lenglart's inequality (cf.~\cite{GK}), 
	one gets
	\begin{equation*}\label{finite-EsupXrp}
		\mathbb{E}^n\Big[\sup_{t\in [0,T]} |X_t^n|^{\gamma p}\Big]
		\lesssim_{p,T} 1 + \mu_0(|\cdot|^p),\text{ for any }\gamma \in (0,1).
	\end{equation*}
	
	Furthermore, from the condition $(\mathbf{H}_3)$ and using the estimate (\ref{finite-approesti}), we can obtain that for any $l>2$,
	\begin{eqnarray*}
		\!\!\!\!\!\!\!\!&&\quad\mathbb{E}^n |X_t^n - X_s^n|^l\\
		\!\!\!\!\!\!\!\!&&=\mathbb{E}^n\Big|\int_s^t \mathcal{A}^n(r,X_r^n,\mathscr{L}_{X_r^n}) dr + \int_s^t \mathcal{B}^n(r,X_r^n,\mathscr{L}_{X_r^n}) dW_r\Big|^l\\
		\!\!\!\!\!\!\!\!&&\lesssim_T |t - s|^{l-1} \mathbb{E}^n \int_s^t |\mathcal{A}^n(r,X_r^n,\mathscr{L}_{X_r^n})|^{l} dr  + \mathbb{E}^n \Big| \int_s^t \mathcal{B}^n(r,X_r^n,\mathscr{L}_{X_r^n}) dW_r\Big|^l\\
		\!\!\!\!\!\!\!\!&&\lesssim_T |t - s|^{l-1}\mathbb{E}^n \int_s^t (1+|X_r^n|^{\frac{l\kappa }{2}}+\mathbb{E}^n|X_r^n|^{\frac{l\kappa }{2}} )dr \\
		\!\!\!\!\!\!\!\!&&\quad + |t - s|^{\frac{l}{2}-1}\mathbb{E}^n\int_s^t(1+|X_r^n|^{\frac{l\kappa }{2}}+\mathbb{E}^n|X_r^n|^{\frac{l\kappa }{2}} )dr\\
		\!\!\!\!\!\!\!\!&&\lesssim_{p,T} |t - s|^l +  |t - s|^{\frac{l}{2}}\\
		\!\!\!\!\!\!\!\!&&\lesssim_{p,T}|t - s|^{\frac{l}{2}},
	\end{eqnarray*}
	where we used B-D-G's inequality in the third step. The proof is finished.
\end{proof}

\noindent\textbf{Proof of Theorem \ref{finite-th1}.}
For brevity, we set $\mu^n:=\mathscr{L}_{X^n}$.
By (\ref{finite-approesti})-(\ref{finite-time}), the sequence $\{\mu^n\}_{n\in\mathbb{N}}$ is tight in $\mathscr{P}(\mathbb{C}_T(\mathbb{R}^d))$. Consequently, $\{ \mathscr{L}_{(X^n,W^n)}\}_{n\in\mathbb{N}}$ is also tight. The Prokhorov theorem guarantees the existence of a subsequence (still denoted by $\{n\}$) such that $\mu^n$ converges weakly to $\mu$ in $\mathscr{P}(\mathbb{C}_T(\mathbb{R}^d))$ and $\mathscr{L}_{(X^n,W^n)}$ converges weakly in $\mathscr{P}(\mathbb{C}_T(\mathbb{R}^d) \times \mathbb{C}_T(\mathbb{R}^d))$.

Therefore, applying the Skorokhod representation theorem (cf.~e.g.~\cite{MV}), we can deduce that there exists a probability space $(\tilde{\Omega},\tilde{\mathscr{F}},\tilde{\mathbb{P}})$ with $\mathbb{C}_T(\mathbb{R}^d) \times \mathbb{C}_T(\mathbb{R}^d)$-valued random variables $(\tilde{X}^n,\tilde{W}^n)$, which coincide in law with $(X^n,W^n)$, such that
\begin{equation}\label{finite-Xcon}
	\tilde{X}^n \xrightarrow{n\to \infty} \tilde{X} \sim \mu \text{ in } \mathbb{C}_T(\mathbb{R}^d) \quad \tilde{\mathbb{P}}\text{-a.s.},
\end{equation}
\begin{equation*}
	\tilde{W}^n \xrightarrow{n\to \infty} \tilde{W} \text{ in } \mathbb{C}_T(\mathbb{R}^d) \quad \tilde{\mathbb{P}}\text{-a.s.}.
\end{equation*}

Let $\tilde{\mathscr{F}}_t^n$ be a usual filtration generated by $\{\tilde{X}_s^n,\tilde{W}_s^n : s \leq t\}$. The uniform estimate (\ref{finite-approesti}) implies
\begin{equation}\label{f-1}
	\sup_{n\in\mathbb{N}} \tilde{\mathbb{E}} \Big[ \sup_{t\in[0,T]} |\tilde{X}_t^n|^{\gamma p} \Big] \lesssim_{p,T} 1 + \mu_0(|\cdot|^p).
\end{equation}
By taking $\gamma>\frac{2}{p}$ and applying the convergence (\ref{finite-Xcon}) and  Vitali's convergence theorem, we have
\begin{equation}\label{f-2}
	\tilde{\mathbb{E}} \Big[ \sup_{t\in[0,T]} |\tilde{X}_t^n - \tilde{X}_t|^{2} \Big] \xrightarrow{n\to \infty} 0.
\end{equation}
Moreover, the lower semicontinuity yields
\begin{equation}\label{f-3}
	\tilde{\mathbb{E}} \Big[ \sup_{t\in[0,T]} |\tilde{X}_t|^{\gamma p} \Big] \lesssim_{p,T} 1 + \mu_0(|\cdot|^p).
\end{equation}

Recall $\mu^n=\mathscr{L}_{\tilde{X}^n}|_{\tilde{\mathbb{P}}}$ and $\mu=\mathscr{L}_{\tilde{X}}|_{\tilde{\mathbb{P}}}$. For any $l \in \mathbb{R}^d$, we define a process $ \tilde{\mathcal{M}}_l^n(t)$ taking values in $ \mathbb{R}$ as follows
\begin{equation}
	\tilde{\mathcal{M}}_l^n(t) := \langle \tilde{X}_t^n,l \rangle - \langle \tilde{X}_0^n,l \rangle - \int_0^t \langle \mathcal{A}(s,\chi_n(\tilde{X}_s^n), \mu_s^n \circ \chi_n^{-1}), l \rangle ds.
\end{equation}
Notice that this process $ \tilde{\mathcal{M}}_l^n(\cdot)$ is a square integrable martingale with respect to the filtration $ \tilde{\mathscr{F}}_t^n$ with quadratic variation
\begin{equation}
	\langle \tilde{\mathcal{M}}_l^n \rangle(t) = \int_0^t |\mathcal{B}(s,\chi_n(\tilde{X}_s^n), \mu_s^n \circ \chi_n^{-1})^* l|^2 ds.
\end{equation}
In fact, for all $0\leq s < t\leq T$ and all bounded continuous functions $\Pi(\cdot)$ on $\mathbb{C}_T(\mathbb{R}^d) \times \mathbb{C}_T(\mathbb{R}^d)$,  we can derive that
\begin{equation}\label{finite-M1}
	\tilde{\mathbb{E}}\Big[ ( \tilde{\mathcal{M}}_l^n(t)-\tilde{\mathcal{M}}_l^n(s) ) \Pi( ( \tilde{X}^n, \tilde{W}^n )|_{[0,s]} ) \Big]= 0
\end{equation}
and
\begin{equation}\label{finite-M2}
	\tilde{\mathbb{E}}\Big[ \Big( (\tilde{\mathcal{M}}_l^n(t))^2 - (\tilde{\mathcal{M}}_l^n(s))^2 - \int_s^t | \mathcal{B}(r, \chi_n(\tilde{X}_r^n), \mu_r^n \circ \chi_n^{-1})^* l |^2 dr\Big) \Pi( (\tilde{X}^n, \tilde{W}^n)|_{[0,s]}) \Big] = 0.
\end{equation}

The next aim is to passage the limits of (\ref{finite-M1}) and (\ref{finite-M2}). According to the definition of the cut-off function $\chi_n$, we can obtain from Chebyshev's inequality and H\"older's inequality that for any $t \in [0,T]$,
\begin{eqnarray}\label{es3}
	\!\!\!\!\!\!\!\!&&\mathcal{W}_2(\mu_t^n \circ \chi_n^{-1}, \mu_t)^2
	\nonumber\\
	\!\!\!\!\!\!\!\!&&\leq \tilde{\mathbb{E}}|\chi_n(\tilde{X}_t^n) - \tilde{X}_t|^2\nonumber\\
	\!\!\!\!\!\!\!\!&&= \tilde{\mathbb{E}}\Big[|\chi_n(\tilde{X}_t^n) - \tilde{X}_t|^2\mathbf{1}_{\{|\tilde{X}_t^n| \geq n\}}\Big] + \tilde{\mathbb{E}}\Big[|\chi_n(\tilde{X}_t^n) - \tilde{X}_t|^2\mathbf{1}_{\{|\tilde{X}_t^n| < n\}}\Big] \nonumber\\
	\!\!\!\!\!\!\!\!&&\leq \tilde{\mathbb{E}}\Big[|\frac{n\tilde{X}_t^n}{|\tilde{X}_t^n|} - \tilde{X}_t|^2\mathbf{1}_{\{|\tilde{X}_t^n| \geq n\}}\Big] + \tilde{\mathbb{E}}\Big[|\tilde{X}_t^n - \tilde{X}_t|^2 \mathbf{1}_{\{|\tilde{X}_t^n| < n\}}\Big]\nonumber\\
	\!\!\!\!\!\!\!\!&&\leq \tilde{\mathbb{E}}\Big[(|\tilde{X}_t^n|^2 + |\tilde{X}_t|^2)\mathbf{1}_{\{|\tilde{X}_t^n| \geq n\}}\Big] + \tilde{\mathbb{E}}|\tilde{X}_t^n - \tilde{X}_t|^2 \nonumber\\
	\!\!\!\!\!\!\!\!&&\leq \Big[(\tilde{\mathbb{E}}|\tilde{X}_t^n|^{\gamma p})^\frac{2}{\gamma p} + (\tilde{\mathbb{E}}|\tilde{X}_t|^{\gamma p})^\frac{2}{\gamma p}\Big]\cdot \frac{(\tilde{\mathbb{E}}|\tilde{X}_t^n|^{\gamma p})^\frac{{\gamma p}-2}{\gamma p}}{n^\frac{{\gamma p}-2}{\gamma p}} + \tilde{\mathbb{E}}|\tilde{X}_t^n - \tilde{X}_t|^2,
\end{eqnarray}
where we take $\gamma\in(\frac{2}{p},1)$. Then by (\ref{f-1})-(\ref{f-3}) and (\ref{es3}) we obtain
$$\mu_t^n \circ \chi_n^{-1} \xrightarrow{n\to \infty} \mu_t~~\text{in}~~ \mathcal{W}_2\text{-sense}.$$

The continuity assumption $(\mathbf{H}_1)$ leads to
\begin{equation*}
	\lim_{n \to\infty } |\mathcal{A}(t, \chi_n(\tilde{X}_t^n), \mu_t^n \circ \chi_n^{-1}) - \mathcal{A}(t, \tilde{X}_t, \mu_t)|=0 \quad \tilde{\mathbb{P}}\text{-a.s.},
\end{equation*}
\begin{equation*}
	\lim_{n \to\infty } \|\mathcal{B}(t, \chi_n(\tilde{X}_t^n), \mu_t^n \circ \chi_n^{-1}) - \mathcal{B}(t, \tilde{X}_t, \mu_t)\|=0 \quad \tilde{\mathbb{P}}\text{-a.s.}.
\end{equation*}
Then by the dominated convergence theorem, one can infer that for any $t \in [0,T]$,
\begin{equation*}
	\lim_{n \to\infty }\int_{0}^{t}  |\mathcal{A}(s, \chi_n(\tilde{X}_s^n), \mu_s^n \circ \chi_n^{-1}) - \mathcal{A}(s, \tilde{X}_s, \mu_s)|^2ds=0 \quad \tilde{\mathbb{P}}\text{-a.s.},
\end{equation*}
\begin{equation*}
	\lim_{n \to\infty } \int_{0}^{t} \|\mathcal{B}(s, \chi_n(\tilde{X}_s^n), \mu_s^n \circ \chi_n^{-1}) - \mathcal{B}(s, \tilde{X}_s, \mu_s)\|^2ds=0 \quad \tilde{\mathbb{P}}\text{-a.s.}.
\end{equation*}
Moreover, based on the estimates (\ref{f-1}), (\ref{f-3}), and the condition $(\mathbf{H}_3)$, we can deduce that
\begin{eqnarray*}
	\!\!\!\!\!\!\!\!&&\quad \tilde{\mathbb{E}}\bigg\{\int_{0}^{t}  |\mathcal{A}(s, \chi_n(\tilde{X}_s^n), \mu_s^n \circ \chi_n^{-1}) - \mathcal{A}(s, \tilde{X}_s, \mu_s)|^2ds\bigg\}^{p_0}\\
	\!\!\!\!\!\!\!\!&&\lesssim_T\tilde{\mathbb{E}}\int_0^t\big(1+|\chi_n(\tilde{X}_s^n)|^{\kappa p_0}+\tilde{\mathbb{E}}|\chi_n(\tilde{X}_s^n)|^{\kappa p_0}+|\tilde{X}_s|^{\kappa p_0}+\tilde{\mathbb{E}}|\tilde{X}_s|^{\kappa p_0}\big)ds\\
	\!\!\!\!\!\!\!\!&&\lesssim_{p,T} 1 + \mu_0(|\cdot|^p)
\end{eqnarray*}
and
\begin{eqnarray*}
	\!\!\!\!\!\!\!\!&&\quad \tilde{\mathbb{E}}\bigg\{\int_{0}^{t}  \|\mathcal{B}(s, \chi_n(\tilde{X}_s^n), \mu_s^n \circ \chi_n^{-1}) - \mathcal{B}(s, \tilde{X}_s, \mu_s)\|^2ds\bigg\}^{p_0}\\
	\!\!\!\!\!\!\!\!&&\lesssim_T\tilde{\mathbb{E}}\int_0^t\big(1+|\chi_n(\tilde{X}_s^n)|^{\kappa p_0}+\tilde{\mathbb{E}}|\chi_n(\tilde{X}_s^n)|^{\kappa p_0}+|\tilde{X}_s|^{\kappa p_0}+\tilde{\mathbb{E}}|\tilde{X}_s|^{\kappa p_0}\big)ds\\
	\!\!\!\!\!\!\!\!&&\lesssim_{p,T} 1 + \mu_0(|\cdot|^p),
\end{eqnarray*}
for some $1<p_0< \frac{p}{\kappa } $.
Consequently, from the Vitali's convergence theorem, it follows that
\begin{equation}\label{es4}
	\lim_{n \to\infty }\tilde{\mathbb{E}}\int_{0}^{t}  |\mathcal{A}(s, \chi_n(\tilde{X}_s^n), \mu_s^n \circ \chi_n^{-1}) - \mathcal{A}(s, \tilde{X}_s, \mu_s)|^2ds=0,
\end{equation}
\begin{equation}\label{es5}
	\lim_{n \to\infty }\tilde{\mathbb{E}}\int_{0}^{t} \|\mathcal{B}(s, \chi_n(\tilde{X}_s^n), \mu_s^n \circ \chi_n^{-1}) - \mathcal{B}(s, \tilde{X}_s, \mu_s)\|^2ds=0.
\end{equation}

Now passaging to the limits of (\ref{finite-M1}) and (\ref{finite-M2}) due to (\ref{es4}) and (\ref{es5}),  for all $s \leq t \in [0, T]$ we have
\begin{equation}\label{f-M1}
	\tilde{\mathbb{E}}\Big[ ( \tilde{\mathcal{M}}_l(t)-\tilde{\mathcal{M}}_l(s) ) \Pi( ( \tilde{X}, \tilde{W} )|_{[0,s]} ) \Big]= 0
\end{equation}
and
\begin{equation}\label{f-M2}
	\tilde{\mathbb{E}}\Big[ ( (\tilde{\mathcal{M}}_l(t))^2 - (\tilde{\mathcal{M}}_l(s))^2 - \int_s^t | \mathcal{B}(r, \tilde{X}_r, \mu_r)^* l |^2 dr) \Pi( (\tilde{X}, \tilde{W})|_{[0,s]}) \Big] = 0,
\end{equation}
where $\tilde{\mathcal{M}}_l(t)$ is defined by
\begin{equation*}
	\tilde{\mathcal{M}}_l(t) := \langle \tilde{X}_t, l \rangle - \langle \tilde{X}_0, l \rangle - \int_0^t \langle \mathcal{A}(s, \tilde{X}_s, \mu_s), l \rangle ds.
\end{equation*}
Therefore, in view of (\ref{f-M1}) and (\ref{f-M2}), the process $\tilde{\mathcal{M}}_l(t)$ forms a square-integrable martingale in $\mathbb{R}$ with respect to the filtration $\tilde{\mathscr{F}}_t$ which satisfies the usual conditions and is generated by $\{\tilde{X}_s, \tilde{W}_s : s \leq t\}$. Its quadratic variation is given by
\begin{equation*}
	\langle \tilde{\mathcal{M}}_l \rangle(t) = \int_0^t | \mathcal{B}(s, \tilde{X}_s, \mu_s)^* l|^2 ds.
\end{equation*}
Utilizing the martingale representation theorem (cf.~\cite{DZ}), Eq.~(\ref{eqSPDE}) possesses a weak solution. \hspace{\fill}$\Box$

\subsection{Energy bounds and tightness}\label{sec4.3}
In this part, we construct  the Galerkin approximation of Eq.~$(\ref{eqSPDE})$.
Let $\{e_1,e_2,\dots\}\subset \mathbb{V}$ be an orthonormal basis of $\mathbb{H}$.
Consider the maps
$$
P_n : \mathbb{V}^* \rightarrow \mathbb{H}_n , ~ n \in \mathbb{N}
$$
given by
$$
P_n x := \sum_{i=1}^{n}\,_{\mathbb{V}^*}\langle x,e_i\rangle_{\mathbb{V}}e_i,~ x \in \mathbb{V}^*,
$$
where $\mathbb{H}_n$ denotes the $n$-dimensional subspace of $\mathbb{H}$ spanned by $\{e_i\}_{i=1}^{n}$.
It is straightforward that if we restrict $P_n$ to $\mathbb{H}$, denoted by $P_n|_\mathbb{H}$, then it is an orthogonal projection onto $\mathbb{H}_n$ on $\mathbb{H}$. Moreover, we know that for any $u \in \mathbb{V}$ and $v \in \mathbb{H}_n$,
$$
\,_{\mathbb{V}^*}\langle P_n\mathcal{A}(t,u,\mu), v\rangle_{\mathbb{V}}= \langle P_n \mathcal{A}(t,u,\mu),v\rangle_{\mathbb{H}} = \,_{\mathbb{V}^*}\langle \mathcal{A}(t,u,\mu), v\rangle_{\mathbb{V}}.
$$
Let $\{g_1,g_2,\dots\}$ be an orthonormal basis of Hilbert space $U$. Set
$$
W^{(n)}_t:= Q_n W_t := \sum_{i=1}^n \langle  W_t,g_i\rangle_U g_i,
$$
where $Q_n$ is the orthogonal projection onto  $\text{span}\{g_i\}_{i=1}^{n}$ on $U$.

For any integer $n \geq 1$, we consider the following stochastic differential equation in the finite-dimensional space $\mathbb{H}_n$
\begin{equation}\label{approeq}
	dX^{(n)}_t =  P_n \mathcal{A}(t, X^{(n)}_t, \mathscr{L}_{X^{(n)}_t}) dt + P_n \mathcal{B}(t, X^{(n)}_t, \mathscr{L}_{X^{(n)}_t})  dW^{(n)}_t,~X^{(n)}_0\sim\mu_0\circ P_n^{-1}.
\end{equation}
Under the assumptions $(\mathbf{A}_1)$-$(\mathbf{A}_4)$, Eq.~(\ref{approeq}) has a weak solution in view of  Theorem \ref{finite-th1}, i.e.
$$(\Omega^{(n)}, \mathscr{F}^{(n)}, \{\mathscr{F}_t^{(n)}\}_{t \geq 0}, \mathbb{P}^{(n)}; X^{(n)}, W^{(n)})~~\text{with}~~~~ \mathbb{P}^{(n)} \circ (X_0^{(n)})^{-1} = \mu_0\circ P_n^{-1}.$$
We denote by $\mathbb{E}^{(n)}$ the expectation associated with $\mathbb{P}^{(n)}$.

As preparation, we  define a sequence of stopping times as follows
$$\tau_{N}^{n}:=\inf \big\{t \in[0, T]:\|X^{(n)}_t\|_{\mathbb{H}}>N\big\}\wedge T ,~~N>0,~~n\in \mathbb{N},$$
where we conventionally set $\inf \emptyset =+\infty$. According to the estimate (\ref{finite-esti}), it's straightforward that for any $n\in \mathbb{N},$
$$\tau_{N}^{n}\xrightarrow{N\to \infty} T\quad\mathbb{P}\text{-a.s.}.$$

The following two lemmas give the energy bounds of the approximating sequence (\ref{approeq}).
\begin{lemma}\label{Yesti1}
	Assume that the assumptions in Theorem \ref{th1} hold, we have
	\begin{equation}\label{Yesti1p}
		\sup_{n\in \mathbb{N}}\Big\{\sup_{t\in [0,T]}\mathbb{E}^{(n)} \|X^{(n)}_t\|_{\mathbb{H}}^p+\mathbb{E}^{(n)}\int_{0}^{T}\|X^{(n)}_t\|_{\mathbb{V}}^{\alpha }\|X^{(n)}_t\|_{\mathbb{H}}^{p-2}dt \Big\}
		\lesssim_{p,T}1+
		\mu_0(\|\cdot\|_{\mathbb{H}}^p).
	\end{equation}
\end{lemma}
\begin{proof}
	Applying It\^{o}'s formula to  $\|X^{(n)}_t\|_{\mathbb{H}}^p$ yields that
	\begin{eqnarray}\label{A13}
		\!\!\!\!\!\!\!\!&&\|X^{(n)}_t\|_{\mathbb{H}}^p-\|X^{(n)}_0 \|_{\mathbb{H}}^p\nonumber\\
		=\!\!\!\!\!\!\!\!&&\frac{p}{2} \int_{0}^{t}\|X^{(n)}_s\|_{\mathbb{H}}^{p-2}\Big(2\,_{\mathbb{V}^*}\langle P_n \mathcal{A}(s, X^{(n)}_s, \mathscr{L}_{X^{(n)}_s}), X^{(n)}_s\rangle_{\mathbb{V}}
		\nonumber\\
		\!\!\!\!\!\!\!\!&&
		+\|P_n \mathcal{B}(s, X^{(n)}_s, \mathscr{L}_{X^{(n)}_s}) \|_{L_{2}(U; \mathbb{H})}^{2}\Big) d s\nonumber\\
		\!\!\!\!\!\!\!\!&&+\frac{p(p-2)}{2}\int_{0}^{t}\|X^{(n)}_s\|_{\mathbb{H}}^{p-4}\|(P_n \mathcal{B}(s, X^{(n)}_s, \mathscr{L}_{X^{(n)}_s}) Q_n)^*X^{(n)}_s\|_{U}^{2} d s\nonumber\\
		\!\!\!\!\!\!\!\!&&+p\int_{0}^{t}\|X^{(n)}_s\|_{\mathbb{H}}^{p-2}\langle X^{(n)}_s, P_n \mathcal{B}(s, X^{(n)}_s, \mathscr{L}_{X^{(n)}_s}) d W^{(n)}_{s}\rangle_{\mathbb{H}}\nonumber\\
		\leq\!\!\!\!\!\!\!\!&&\frac{p}{2} \int_{0}^{t}\|X^{(n)}_s\|_{\mathbb{H}}^{p-2}\Big[2\,_{\mathbb{V}^*}\langle P_n \mathcal{A}(s, X^{(n)}_s, \mathscr{L}_{X^{(n)}_s}), X^{(n)}_s\rangle_{\mathbb{V}}
		\nonumber\\
		\!\!\!\!\!\!\!\!&&+(p-1)\|P_n \mathcal{B}(s, X^{(n)}_s, \mathscr{L}_{X^{(n)}_s})Q_n \|_{L_{2}(U; \mathbb{H})}^{2}\Big]\nonumber\\
		\!\!\!\!\!\!\!\!&&+p\int_{0}^{t}\|X^{(n)}_s\|_{\mathbb{H}}^{p-2}\langle X^{(n)}_s, P_n \mathcal{B}(s, X^{(n)}_s, \mathscr{L}_{X^{(n)}_s})  d W^{(n)}_{s}\rangle_{\mathbb{H}}\nonumber\\
		=:\!\!\!\!\!\!\!\!&&\mathscr{I}_1(t)+\mathscr{I}_2(t).
	\end{eqnarray}
	By the condition $(\mathbf{A}_3)$ and using Young's inequality gives
	\begin{eqnarray*}
		\!\!\!\!\!\!\!\!&&\mathbb{E}^{(n)}\mathscr{I}_1(t\wedge \tau_{N}^{n})\\
		\leq\!\!\!\!\!\!\!\!&&-\frac{p\delta }{2} \mathbb{E}^{(n)}\int_{0}^{t \wedge \tau_{N}^{n}} \|X^{(n)}_s\|_{\mathbb{V}}^{\alpha }\|X^{(n)}_s\|_{\mathbb{H}}^{p-2}ds
		\\
		\!\!\!\!\!\!\!\!&&+C_p\mathbb{E}^{(n)}\int_{0}^{t \wedge \tau_{N}^{n}} (1+\|X^{(n)}_s\|_{\mathbb{H}}^{2}+\mathbb{E}^{(n)}\|X^{(n)}_s\|_{\mathbb{H}}^{2})\|X^{(n)}_s\|_{\mathbb{H}}^{p-2}ds\\
		\leq\!\!\!\!\!\!\!\!&&-\frac{p\delta }{2} \mathbb{E}^{(n)}\int_{0}^{t \wedge \tau_{N}^{n}} \|X^{(n)}_s\|_{\mathbb{V}}^{\alpha }\|X^{(n)}_s\|_{\mathbb{H}}^{p-2}ds
		\\
		\!\!\!\!\!\!\!\!&&+C_p\mathbb{E}^{(n)}\int_{0}^{t \wedge \tau_{N}^{n}} (1+\|X^{(n)}_s\|_{\mathbb{H}}^{p}+\mathbb{E}^{(n)}\|X^{(n)}_s\|_{\mathbb{H}}^{p})ds.
	\end{eqnarray*}
	By the martingale property of the term $\mathscr{I}_2(t)$, we deduce
	\begin{equation*}
		\mathbb{E}^{(n)}\mathscr{I}_2(t\wedge \tau_{N}^{n})=0.
	\end{equation*}
	Therefore, we derive for any $t \in [0,T]$,
	\begin{eqnarray}\label{Ye1}
		\!\!\!\!\!\!\!\!&&\quad\mathbb{E}^{(n)}\|X^{(n)}_{t\wedge \tau_{N}^{n}}\|_{\mathbb{H}}^p+\mathbb{E}^{(n)}\int_{0}^{t\wedge \tau_{N}^{n}}\|X^{(n)}_s\|_{\mathbb{V}}^{\alpha }\|X^{(n)}_s\|_{\mathbb{H}}^{p-2}ds\nonumber\\
		\!\!\!\!\!\!\!\!&&\lesssim_p\mu_0(\|\cdot\|_{\mathbb{H}}^p)+\mathbb{E}^{(n)}\int_{0}^{t \wedge \tau_{N}^{n}} (1+\|X^{(n)}_s\|_{\mathbb{H}}^{p}+\mathbb{E}^{(n)}\|X^{(n)}_s\|_{\mathbb{H}}^{p})ds\nonumber\\
		\!\!\!\!\!\!\!\!&&\lesssim_{p,T} (1+\mu_0(\|\cdot\|_{\mathbb{H}}^p))+\mathbb{E}^{(n)}\int_{0}^{T}\|X^{(n)}_s\|_{\mathbb{H}}^{p}ds.
	\end{eqnarray}
	Taking $N \to \infty$ and using Fatou's lemma and  Gronwall's lemma, we have for any $t \in [0,T]$,
	\begin{equation*}
		\mathbb{E}^{(n)}\|X^{(n)}_t\|_{\mathbb{H}}^p+\mathbb{E}^{(n)}\int_{0}^{t}\|X^{(n)}_s\|_{\mathbb{V}}^{\alpha }\|X^{(n)}_s\|_{\mathbb{H}}^{p-2}ds
		\lesssim_{p,T} 1+\mu_0(\|\cdot\|_{\mathbb{H}}^p),
	\end{equation*}
	which leads to (\ref{Yesti1p}).
\end{proof}

\begin{lemma}\label{Yesti2}
	Assume that the assumptions in Theorem \ref{th1} hold, for any $\gamma \in (0,1)$ we have
	\begin{equation}\label{Yesti2p}
		\sup_{n\in \mathbb{N}}\mathbb{E}^{(n)}\Big[\sup_{t\in [0,T]} \|X^{(n)}_t\|_{\mathbb{H}}^{\gamma p}\Big]
		\lesssim_{p,T}(1+\mu_0(\|\cdot\|_{\mathbb{H}}^p))^\gamma .
	\end{equation}
\end{lemma}
\begin{proof}
	Let $\tau_0$ be a bounded stopping time with $\tau_0\leq T.$ By replacing $\tau_{N}^{n}$ with $\tau_{N}^{n}\wedge \tau_0$, we rewrite (\ref{Ye1}) as follows
	\begin{eqnarray*}
		\!\!\!\!\!\!\!\!&&\quad \mathbb{E}^{(n)}\|X^{(n)}_{t\wedge \tau_{N}^{n}\wedge \tau_0}\|_{\mathbb{H}}^p+\mathbb{E}^{(n)}\int_{0}^{t\wedge \tau_{N}^{n}\wedge \tau_0}\|X^{(n)}_s\|_{\mathbb{V}}^{\alpha }\|X^{(n)}_s\|_{\mathbb{H}}^{p-2}ds\\
		\!\!\!\!\!\!\!\!&&\lesssim_p \mu_0(\|\cdot\|_{\mathbb{H}}^p)+\mathbb{E}^{(n)}\int_{0}^{t \wedge \tau_{N}^{n}\wedge \tau_0} (1+\|X^{(n)}_s\|_{\mathbb{H}}^{p}+\mathbb{E}^{(n)}\|X^{(n)}_s\|_{\mathbb{H}}^{p})ds\\
		\!\!\!\!\!\!\!\!&&\lesssim_{p,T}(1+\mu_0(\|\cdot\|_{\mathbb{H}}^p))+\int_{0}^{T}\mathbb{E}^{(n)}\|X^{(n)}_s\|_{\mathbb{H}}^{p}ds\\
		\!\!\!\!\!\!\!\!&&\lesssim_{p,T}(1+\mu_0(\|\cdot\|_{\mathbb{H}}^p)),
	\end{eqnarray*}
	where the last step used Lemma \ref{Yesti1}. Taking $N \to \infty$ and  applying Fatou's lemma, we derive for any $t \in [0,T]$,
	\begin{equation*}
		\mathbb{E}^{(n)}\|X^{(n)}_{t\wedge \tau_0}\|_{\mathbb{H}}^p
		\lesssim_{p,T} 1+\mu_0(\|\cdot\|_{\mathbb{H}}^p).
	\end{equation*}
	Using Lenglart's inequality (cf.~\cite{GK}), 
	one gets
	\begin{equation*}
		\mathbb{E}^{(n)}\Big[\sup_{t\in [0,T]} \|X^{(n)}_t\|_{\mathbb{H}}^{\gamma p}\Big]
		\lesssim_{p,T}(1+\mu_0(\|\cdot\|_{\mathbb{H}}^p))^\gamma ,\text{ for any }\gamma \in (0,1).
	\end{equation*}
	We complete the proof.
\end{proof}

\begin{lemma}\label{nIPSesti2}
	Assume that the assumptions in Theorem \ref{th1} hold, for any $l \in (1,\frac{p}{\beta}] \cap (1,\frac{p}{2}) $ we have
	\begin{equation}\label{Yesti12}
		\sup_{n\in \mathbb{N}}\mathbb{E}^{(n)}\Big(\int_{0}^{T}\|X^{(n)}_t\|_{\mathbb{V}}^{\alpha }dt\Big)^l
		\lesssim_{T}1+\big(\mu_0(\|\cdot\|_{\mathbb{H}}^{p})\big)^\frac{2l}{p}+\mu_0(\|\cdot\|_{\mathbb{H}}^{\beta l}).
	\end{equation}
\end{lemma}
\begin{proof}
	Applying It\^{o}'s formula to  $\|X^{(n)}_t\|_{\mathbb{H}}^{2}$ and using the condition $(\mathbf{A}_3)$, it is straightforward that
	\begin{eqnarray*}
		\int_{0}^{t}\|X^{(n)}_s\|_{\mathbb{V}}^{\alpha}ds-\|X^{(n)}_0 \|_{\mathbb{H}}^2
		\lesssim \!\!\!\!\!\!\!\!&&\int_{0}^{t}(1+\|X^{(n)}_s\|_{\mathbb{H}}^{2}+\mathbb{E}^{(n)}\|X^{(n)}_s\|_{\mathbb{H}}^{2})ds\nonumber\\
		\!\!\!\!\!\!\!\!&&+\Big|\int_{0}^{t}\langle X^{(n)}_s, P_n\mathcal{B}(s, X^{(n)}_s, \mathscr{L}_{X^{(n)}_s})  d W^{(n)}_{s}\rangle_{\mathbb{H}}\Big|.
	\end{eqnarray*}
	Then by the condition $(\mathbf{A}_4)$, B-D-G's inequality and Young's inequality, we have
	\begin{eqnarray*}
		\!\!\!\!\!\!\!\!&&\quad\mathbb{E}^{(n)}\Big(\int_{0}^{T}\|X^{(n)}_t\|_{\mathbb{V}}^{\alpha}dt\Big)^l
		\\
		\!\!\!\!\!\!\!\!&&\lesssim\mu_0(\|\cdot\|_{\mathbb{H}}^{2l})+\mathbb{E}^{(n)}\int_{0}^{T}\big(1+\|X^{(n)}_t\|_{\mathbb{H}}^{2}+\mathbb{E}^{(n)}\|X^{(n)}_t\|_{\mathbb{H}}^{2}\big)^l dt\\
		\!\!\!\!\!\!\!\!&&\quad+\mathbb{E}^{(n)}\bigg\{\sup_{t\in [0,T]}\Big|\int_{0}^{t}\langle X^i_s, P_n\mathcal{B}(s, X^{(n)}_s, \mathscr{L}_{X^{(n)}_s})  d W^{(n)}_{s}\rangle_{\mathbb{H}}\Big|\bigg\}^l\\
		\!\!\!\!\!\!\!\!&&\lesssim_T 1+\mu_0(\|\cdot\|_{\mathbb{H}}^{2l})+\int_{0}^{T}\mathbb{E}^{(n)}\|X^{(n)}_t\|_{\mathbb{H}}^{2l}dt\\
		\!\!\!\!\!\!\!\!&&\quad+\mathbb{E}^{(n)}\bigg\{\sup_{t \in[0,T]}\|X_t^{(n)}\|_{\mathbb{H}}^{2}\cdot \int_{0}^{T}\big(1+\|X_t^{(n)}\|_{\mathbb{H}}^{\beta}+\mathbb{E}^{(n)}\|X_t^{(n)}\|_{\mathbb{H}}^{\beta}\big)dt\bigg\}^\frac{l}{2}\\
		\!\!\!\!\!\!\!\!&&\lesssim_T 1+\mu_0(\|\cdot\|_{\mathbb{H}}^{2l})+\mathbb{E}^{(n)}\Big[\sup_{t \in[0,T]}\|X_t^{(n)}\|_{\mathbb{H}}^{2l}\Big]+\sup_{t \in[0,T]}\mathbb{E}^{(n)}\|X_t^{(n)}\|_{\mathbb{H}}^{\beta l}\\
		\!\!\!\!\!\!\!\!&&\lesssim_T 1+\big(\mu_0(\|\cdot\|_{\mathbb{H}}^{p})\big)^\frac{2l}{p}+\mu_0(\|\cdot\|_{\mathbb{H}}^{\beta l}),
	\end{eqnarray*}
	where the last step used the estimates (\ref{Yesti1p}) and (\ref{Yesti2p}).
\end{proof}

Define a sequence of stopping times
$$\tau^{n}_R := \inf\Big\{t \in[0,T]: \|X^{(n)}_t\|_{\mathbb{H}} >R\Big\} \wedge \inf\Big\{t \in[0,T]: \int_0^t \|X^{(n)}_s\|_{\mathbb{V}}^{\alpha}ds > R\Big\} \wedge T.$$

We present the following lemmas to establish the tightness of the family of laws $\{\mathscr{L}_{X^{(n)}}\}.$
\begin{lemma}\label{tightness}
	$\{\mathscr{L}_{X^{(n)}}\}_{n=1}^{\infty}$ is tight in $\mathscr{P}(\mathbb{C}_T(\mathbb{V}^{*}))$.
\end{lemma}
\begin{proof}
	By Chebyshev's inequality and Lemma \ref{Yesti2}, it is easy to get
	\begin{equation}\label{tight1}
		\lim_{M \to \infty} \sup_{n \in \mathbb{N}} \mathbb{P}^{(n)}\Big( \sup_{t\in[0,T]} \|X^{(n)}_t\|_{\mathbb{H}} > M\Big)=0.
	\end{equation}
	Since the embedding $\mathbb{H} \subset \mathbb{V}^{*}$ is compact, Theorem 3.1 in \cite{J} implies that it is sufficient to verify the tightness of $\{\,_{\mathbb{V}^*}\langle X^{(n)}, e\rangle_{\mathbb{V}}\}_{n=1}^{\infty}$ in $\mathbb{C}_T(\mathbb{R})$, for every $e \in \mathbb{H}_m$ with $m \geq 1$. By Aldous's criterion \cite{A} and the estimate (\ref{tight1}), the proof reduces to demonstrating that for any stopping time $0 \leq \zeta_n \leq T$ and any $\varepsilon > 0$, the following limit holds:
	\begin{equation}\label{tight2}
		\lim_{\Lambda  \to 0} \sup_{n \in \mathbb{N}} \mathbb{P}^{(n)} \big( | \,_{\mathbb{V}^*}\langle X^{(n)}_{\zeta_n + \Lambda} - X^{(n)}_{\zeta_n}, e \rangle_{\mathbb{V}} | > \varepsilon \big) = 0,
	\end{equation}
	where $\zeta_n + \Lambda := T \wedge (\zeta_n + \Lambda) \lor 0$.
	
	Using Chebyshev's inequality, we derive
	\begin{eqnarray}\label{tight3}
		\!\!\!\!\!\!\!\!&&\mathbb{P}^{(n)} \Big( | \,_{\mathbb{V}^*}\langle X^{(n)}_{\zeta_n + \Lambda} - X^{(n)}_{\zeta_n}, e \rangle_{\mathbb{V}} | > \varepsilon \Big)\nonumber\\
		\leq\!\!\!\!\!\!\!\!&& \mathbb{P}^{(n)} \Big( | \,_{\mathbb{V}^*}\langle X^{(n)}_{\zeta_n + \Lambda} - X^{(n)}_{\zeta_n}, e \rangle_{\mathbb{V}} | > \varepsilon,\tau_{R}^{n} \geq T \Big)+\mathbb{P}^{(n)} \big(\tau_{R}^{n}<T \big)\nonumber\\
		\leq \!\!\!\!\!\!\!\!&&\frac{1}{\varepsilon^{\frac{\alpha}{\alpha-1}}} \mathbb{E}^{(n)}\Big(| \,_{\mathbb{V}^*}\langle X^{(n)}_{(\zeta_n + \Lambda)\wedge\tau_{R}^{n}} - X^{(n)}_{\zeta_n\wedge\tau_{R}^{n}}, e \rangle_{\mathbb{V}} |^{\frac{\alpha}{\alpha-1}}\Big) + \mathbb{P}^{(n)} \big(\tau_{R}^{n}<T \big).
	\end{eqnarray}
	For the second term on the right-hand side of (\ref{tight3}), by applying Chebyshev's inequality it follows that
	\begin{eqnarray}\label{tauR}
		\mathbb{P}^{(n)} \big(\tau_{R}^{n}<T \big)
		\leq\!\!\!\!\!\!\!\!&& \mathbb{P}^{(n)} \Big( \int_0^T \|X^{(n)}_t\|_{\mathbb{V}}^\alpha dt > R \Big)+\mathbb{P}^{(n)} \Big( \sup_{t \in [0,T]} \|X^{(n)}_t\|_{\mathbb{H}} > R\Big)\nonumber\\
		\leq\!\!\!\!\!\!\!\!&& \frac{1}{R} \mathbb{E}^{(n)} \int_0^T \|X^{(n)}_t\|_{\mathbb{V}}^\alpha dt + \frac{1}{R^2} \mathbb{E}^{(n)} \Big[ \sup_{t \in [0,T]} \|X^{(n)}_t\|_{\mathbb{H}}^2 \Big].
	\end{eqnarray}
	Recalling the estimate (\ref{Yesti12}) and Lemma \ref{Yesti2}, we can derive
	\begin{equation*}
		\lim_{R \to \infty}\sup_{n \in \mathbb{N}} \mathbb{P}^{(n)} \big(\tau_{R}^{n}<T \big)=0.
	\end{equation*}
	We claim that
	\begin{equation}\label{tight4}
		\sup_{n \in \mathbb{N}}\mathbb{E}^{(n)}\Big(| \,_{\mathbb{V}^*}\langle X^{(n)}_{(\zeta_n + \Lambda)\wedge\tau_{R}^{n}} - X^{(n)}_{\zeta_n\wedge\tau_{R}^{n}}, e \rangle_{\mathbb{V}} |^{\frac{\alpha}{\alpha-1}}\Big)\to0~~\text{as}~~\Lambda \to0,
	\end{equation}
	whose proof is given later. Then taking $\Lambda \to0$ and then $R \to \infty$ in (\ref{tight3}), we can conclude (\ref{tight2}) holds.
	
	Now we prove (\ref{tight4}), by Burkholder--Davis--Gundy inequality it follows that
	\begin{eqnarray}\label{tight5}
		\!\!\!\!\!\!\!\!&&\mathbb{E}^{(n)}\Big(| \,_{\mathbb{V}^*}\langle X^{(n)}_{(\zeta_n + \Lambda)\wedge\tau_{R}^{n}} - X^{(n)}_{\zeta_n\wedge\tau_{R}^{n}}, e \rangle_{\mathbb{V}} |^{\frac{\alpha}{\alpha-1}}\Big)\nonumber\\
		\lesssim\!\!\!\!\!\!\!\!&&  \mathbb{E}^{(n)}\bigg\{ \int_{\zeta_n\wedge\tau_{R}^{n}}^{(\zeta_n + \Lambda)\wedge\tau_{R}^{n}} | \,_{\mathbb{V}^*}\langle P_{n} \mathcal{A}(t, X^{(n)}_t, \mathscr{L}_{X^{(n)}_t}), e \rangle_{\mathbb{V}}| dt \bigg\}^{\frac{\alpha}{\alpha - 1}}\nonumber\\
		\!\!\!\!\!\!\!\!&&+ \mathbb{E}^{(n)}\bigg\{ \int_{\zeta_n\wedge\tau_{R}^{n}}^{(\zeta_n + \Lambda)\wedge\tau_{R}^{n}} \|P_n \mathcal{B}(t, X^{(n)}_t, \mathscr{L}_{X^{(n)}_t}) \|_{L_{2}(U, \mathbb{H})}^{2}\|e\|_{\mathbb{H}}^2 dt \bigg\}^{\frac{\alpha}{2(\alpha - 1)}}\nonumber\\
		=:\!\!\!\!\!\!\!\!&&\mathscr{J}_1+\mathscr{J}_2.
	\end{eqnarray}
	It is noteworthy that by Lemmas \ref{Yesti1} and \ref{Yesti2} we know
	\begin{equation*}
		\int_{0}^{T\wedge\tau_{R}^{n}} \mathscr{L}_{X^{(n)}_t}(\|\cdot\|_{\mathbb{V}}^{\alpha})dt\leq \mathbb{E}^{(n)}\int_{0}^{T}\|X^{(n)}_t\|_{\mathbb{V}}^{\alpha}dt\lesssim 1+\mu_0(\|\cdot\|_{\mathbb{H}}^2),
	\end{equation*}
	and
	\begin{equation*}
		\mathscr{L}_{X^{(n)}_t}(\|\cdot\|_{\mathbb{H}}^{\beta})\leq\sup_{t \in [0,T]}\mathbb{E}^{(n)} \|X^{(n)}_t\|_{\mathbb{H}}^\beta \lesssim 1+\mu_0(\|\cdot\|_{\mathbb{H}}^{\beta}).
	\end{equation*}
	Therefore, applying the condition $(\mathbf{A}_4)$ and H\"older's inequality, we derive
	\begin{eqnarray*}
		\mathscr{J}_1\!\!\!\!\!\!\!\!&&\lesssim |\Lambda |^{\frac{1}{\alpha-1}} \cdot \mathbb{E}^{(n)} \int_{\zeta_n\wedge\tau_{R}^{n}}^{(\zeta_n + \Lambda)\wedge\tau_{R}^{n}}|\,_{\mathbb{V}^*}\langle P_{n} \mathcal{A}(t, X^{(n)}_t, \mathscr{L}_{X^{(n)}_t}), e \rangle_{\mathbb{V}}|^{\frac{\alpha}{\alpha-1}} dt\\
		\!\!\!\!\!\!\!\!&&\lesssim_{\|e\|_{\mathbb{V}}} |\Lambda |^{\frac{1}{\alpha-1}} \cdot \mathbb{E}^{(n)} \int_0^{T \wedge \tau_{R}^{n}}\Big\{ (1 + \|X^{(n)}_t\|_{\mathbb{V}}^\alpha + \mathscr{L}_{X^{(n)}_t}(\| \cdot \|_{\mathbb{V}}^\alpha))
		\\
		\!\!\!\!\!\!\!\!&&
		\cdot(1 + \|X^{(n)}_t\|_{\mathbb{H}}^\beta + \mathscr{L}_{X^{(n)}_t}(\| \cdot \|_{\mathbb{H}}^\beta))\Big\} dt
		\\
		\!\!\!\!\!\!\!\!&&\lesssim_{R,\|e\|_{\mathbb{V}}} |\Lambda |^{\frac{1}{\alpha-1}}(1+\mu_0(\|\cdot\|_{\mathbb{H}}^{\beta}))\cdot\mathbb{E}^{(n)} \int_0^{T \wedge \tau_{R}^{n}} (1 + \|X^{(n)}_t\|_{\mathbb{V}}^\alpha + \mathscr{L}_{X^{(n)}_t}(\| \cdot \|_{\mathbb{V}}^\alpha))dt
		\\
		\!\!\!\!\!\!\!\!&&\lesssim_{R,\|e\|_{\mathbb{V}}} |\Lambda |^{\frac{1}{\alpha-1}}(1+\mu_0(\|\cdot\|_{\mathbb{H}}^{\beta}))(1+\mu_0(\|\cdot\|_{\mathbb{H}}^{2})),
	\end{eqnarray*}
	and
	\begin{eqnarray*}
		\mathscr{J}_2\!\!\!\!\!\!\!\!&&\lesssim_{\|e\|_{\mathbb{H}}}\mathbb{E}^{(n)} \bigg[\bigg( \int_{\zeta_n\wedge\tau_{R}^{n}}^{(\zeta_n + \Lambda)\wedge\tau_{R}^{n}}(1 + \|X^{(n)}_t\|_{\mathbb{H}}^\beta + \mathscr{L}_{X^{(n)}_t}(\| \cdot \|_{\mathbb{H}}^\beta)) dt \bigg)^{\frac{\alpha}{2(\alpha-1)}}\bigg]\\
		\!\!\!\!\!\!\!\!&&\lesssim_{R,\|e\|_{\mathbb{H}}} |\Lambda |^{\frac{\alpha}{2(\alpha-1)}}(1+\mu_0(\|\cdot\|_{\mathbb{H}}^{\beta}))^{\frac{\alpha}{2(\alpha-1)}}.
	\end{eqnarray*}
	By revisiting (\ref{tight5}) and taking the limit $\Lambda \to 0,$ we conclude that (\ref{tight4}) is satisfied. We complete the proof.
\end{proof}

\begin{lemma}\label{tightness2}
	$\{\mathscr{L}_{X^{(n)}}\}_{n=1}^{\infty}$ is tight in $\mathscr{P}( L^{\alpha}([0,T], \mathbb{H})).$
\end{lemma}
\begin{proof}
	Since $\mathbb{V}\subset \mathbb{H}$ is compact, according to \cite[Lemma 5.2]{RSZ}, it is sufficient to verify
	
	\vspace{1mm}
	(a)
	\begin{equation*}
		\lim_{M \to \infty} \sup_{n \in \mathbb{N}} \mathbb{P}^{(n)} \bigg(\int_{0}^{T} \|X^{(n)}_t\|_{\mathbb{V}}^\alpha dt>M\bigg)=0;
	\end{equation*}
	
	\vspace{1mm}
	(b) for any $\epsilon>0$
	\begin{equation*}
		\lim_{\delta\to 0^+}\sup_{n \in \mathbb{N}} \mathbb{P}^{(n)} \bigg(\sup _{0 \leq \Lambda \leq \delta}\int_{0}^{T-\Lambda } \|X^{(n)}_{t+\Lambda }-X^{(n)}_t\|_{\mathbb{V}^*}^\alpha dt>\epsilon \bigg)= 0.
	\end{equation*}
	
	Based on Lemma \ref{Yesti1}, it is clear that (a) is satisfied.
	Now we  verify (b). By Chebyshev's inequality, we derive that
	\begin{eqnarray}\label{es6}
		\!\!\!\!\!\!\!\!&&\quad\mathbb{P}^{(n)} \bigg(\sup _{0 \leq \Lambda \leq \delta}\int_{0}^{T-\Lambda } \|X^{(n)}_{t+\Lambda }-X^{(n)}_t\|_{\mathbb{V}^*}^\alpha dt>\epsilon \bigg)\nonumber\\
		\!\!\!\!\!\!\!\!&&\lesssim\mathbb{P}^{(n)} \bigg(\sup _{0 \leq \Lambda \leq \delta}\int_{0}^{T-\Lambda } \|X^{(n)}_{t+\Lambda }-X^{(n)}_t\|_{\mathbb{V}^*}^\alpha dt>\epsilon, \tau_{R}^{n} \geq T \bigg)+\mathbb{P}^{(n)} \big(\tau_{R}^{n} < T \big)\nonumber\\
		\!\!\!\!\!\!\!\!&&\lesssim \frac{1}{\epsilon} \mathbb{E}^{(n)} \bigg(\sup _{0 \leq \Lambda \leq \delta}\int_{0}^{T-\Lambda } \|X^{(n)}_{(t+\Lambda )\wedge \tau_{R}^{n} }-X^{(n)}_{t\wedge \tau_{R}^{n} }\|_{\mathbb{V}^*}^\alpha dt\bigg)+\mathbb{P}^{(n)} \big(\tau_{R}^{n} < T \big).
	\end{eqnarray}
	We claim that
	\begin{equation}\label{ball}
		\lim_{\delta  \to 0^+}\sup_{n\in\mathbb{N}} \mathbb{E}^{(n)} \bigg(\sup _{0 \leq \Lambda \leq \delta}\int_{0}^{T-\Lambda } \|X^{(n)}_{(t+\Lambda )\wedge \tau_{R}^{n} }-X^{(n)}_{t\wedge \tau_{R}^{n} }\|_{\mathbb{H}}^\alpha dt\bigg)=0.
	\end{equation}
	Combining (\ref{tauR}), (\ref{ball}), and the fact that $\mathbb{H}\subset \mathbb{V}^*$, (b) is satisfied by  letting $\delta  \to 0^+$ and then $R\to \infty$ in (\ref{es6}). Then the lemma follows.
	
	Now the  remaining step is to prove (\ref{ball}). Regarding the range of values of $\alpha$, the proof is divided into two cases:
	
	\noindent\textbf{Case 1:} ($1<\alpha \leq 2$). Note that H\"older's inequality implies
	\begin{eqnarray*}
		\!\!\!\!\!\!\!\!&&\lim_{\delta  \to 0^+}\sup_{n\in\mathbb{N}} \mathbb{E}^{(n)} \bigg(\sup _{0 \leq \Lambda \leq \delta}\int_{0}^{T-\Lambda } \|X^{(n)}_{(t+\Lambda )\wedge \tau_{R}^{n} }-X^{(n)}_{t\wedge \tau_{R}^{n} }\|_{\mathbb{H}}^\alpha dt\bigg)\\
		\lesssim\!\!\!\!\!\!\!\!&& \lim_{\delta  \to 0^+}\sup_{n\in\mathbb{N}} \bigg\{\mathbb{E}^{(n)} \bigg(\sup _{0 \leq \Lambda \leq \delta}\int_{0}^{T-\Lambda } \|X^{(n)}_{(t+\Lambda )\wedge \tau_{R}^{n} }-X^{(n)}_{t\wedge \tau_{R}^{n} }\|_{\mathbb{H}}^2 dt\bigg)\bigg\}^\frac{\alpha }{2}.
	\end{eqnarray*}
	Thus it remains to prove
	\begin{equation}\label{b1}
		\sup_{n\in\mathbb{N}} \mathbb{E}^{(n)} \bigg(\sup _{0 \leq \Lambda \leq \delta}\int_{0}^{T-\Lambda } \|X^{(n)}_{(t+\Lambda )\wedge \tau_{R}^{n} }-X^{(n)}_{t\wedge \tau_{R}^{n} }\|_{\mathbb{H}}^2 dt\bigg)\xrightarrow{\delta\to 0^+}0.
	\end{equation}
	
	Applying It\^{o}'s formula, we have
	\begin{eqnarray*}
		\!\!\!\!\!\!\!\!&&\|X^{(n)}_{(t+\Lambda )\wedge \tau_{R}^{n} }-X^{(n)}_{t\wedge \tau_{R}^{n} }\|_{\mathbb{H}}^2\\
		\leq\!\!\!\!\!\!\!\!&&2\int_{t\wedge \tau_{R}^{n}}^{(t+\Lambda )\wedge \tau_{R}^{n}} \,_{\mathbb{V}^*}\langle  \mathcal{A}(s, X^{(n)}_s, \mathscr{L}_{X^{(n)}_s}), X^{(n)}_s- X^{(n)}_{t\wedge \tau_{R}^{n}} \rangle_{\mathbb{V}} ds\\
		\!\!\!\!\!\!\!\!&&+\int_{t\wedge \tau_{R}^{n}}^{(t+\Lambda )\wedge \tau_{R}^{n}}\| \mathcal{B}(s, X^{(n)}_s, \mathscr{L}_{X^{(n)}_s})\|_{L_{2}(U; \mathbb{H})}^{2}ds\\
		\!\!\!\!\!\!\!\!&&+2\int_{t\wedge \tau_{R}^{n}}^{(t+\Lambda )\wedge \tau_{R}^{n}}\langle  \mathcal{B}(s, X^{(n)}_s, \mathscr{L}_{X^{(n)}_s})  d W^{(n)}_{s}, X^{(n)}_s- X^{(n)}_{t\wedge \tau_{R}^{n}}\rangle_{\mathbb{H}}.
	\end{eqnarray*}
	It follows that
	\begin{eqnarray}\label{b2}
		\!\!\!\!\!\!\!\!&&\mathbb{E}^{(n)} \bigg(\sup _{0 \leq \Lambda \leq \delta}\int_{0}^{T-\Lambda } \|X^{(n)}_{(t+\Lambda )\wedge \tau_{R}^{n} }-X^{(n)}_{t\wedge \tau_{R}^{n} }\|_{\mathbb{H}}^2 dt\bigg)\nonumber\\
		\lesssim\!\!\!\!\!\!\!\!&&\int_{0}^{T}\mathbb{E}^{(n)} \Big(\sup _{0 \leq \Lambda \leq \delta}\|X^{(n)}_{(t+\Lambda )\wedge \tau_{R}^{n} }-X^{(n)}_{t\wedge \tau_{R}^{n} }\|_{\mathbb{H}}^2 \Big)dt\nonumber\\
		\lesssim\!\!\!\!\!\!\!\!&&\int_{0}^{T}\mathbb{E}^{(n)} \bigg(\sup _{0 \leq \Lambda \leq \delta}\Big|\int_{t\wedge \tau_{R}^{n}}^{(t+\Lambda )\wedge \tau_{R}^{n}} \,_{\mathbb{V}^*}\langle  \mathcal{A}(s, X^{(n)}_s, \mathscr{L}_{X^{(n)}_s}),  X^{(n)}_{t\wedge \tau_{R}^{n}} \rangle_{\mathbb{V}} ds\Big|\bigg)dt\nonumber\\
		\!\!\!\!\!\!\!\!&&+\int_{0}^{T}\mathbb{E}^{(n)} \bigg(\sup _{0 \leq \Lambda \leq \delta}\int_{t\wedge \tau_{R}^{n}}^{(t+\Lambda )\wedge \tau_{R}^{n}}\Big(2\,_{\mathbb{V}^*}\langle  \mathcal{A}(s, X^{(n)}_s, \mathscr{L}_{X^{(n)}_s}), X^{(n)}_s \rangle_{\mathbb{V}}
		\nonumber\\
		\!\!\!\!\!\!\!\!&&+\| \mathcal{B}(s, X^{(n)}_s, \mathscr{L}_{X^{(n)}_s}) \|_{L_{2}(U; \mathbb{H})}^{2}\Big)ds\bigg)dt\nonumber\\
		\!\!\!\!\!\!\!\!&&+\int_{0}^{T}\mathbb{E}^{(n)} \bigg(\sup _{0 \leq \Lambda \leq \delta}\Big|\int_{t\wedge \tau_{R}^{n}}^{(t+\Lambda )\wedge \tau_{R}^{n}}\langle \mathcal{B}(s, X^{(n)}_s, \mathscr{L}_{X^{(n)}_s})  d W^{(n)}_{s}, X^{(n)}_s- X^{(n)}_{t\wedge \tau_{R}^{n}}\rangle_{\mathbb{H}}\Big|\bigg)dt\nonumber\\
		=:\!\!\!\!\!\!\!\!&&\sum_{i=1}^{3} \mathscr{I}_i.
	\end{eqnarray}
	For the first term, we use the assumption $(\mathbf{A}_4)$, H\"older's inequality and Lemmas \ref{Yesti1}, \ref{Yesti2} to obtain
	\begin{eqnarray}\label{b21}
		\mathscr{I}_1\!\!\!\!\!\!\!\!&&\lesssim\int_{0}^{T}\mathbb{E}^{(n)} \bigg(\sup _{0 \leq \Lambda \leq \delta} \int_{t\wedge \tau_{R}^{n}}^{(t+\Lambda )\wedge \tau_{R}^{n}} \|\mathcal{A}(s, X^{(n)}_s, \mathscr{L}_{X^{(n)}_s})\|_{\mathbb{V}^*} \|X^{(n)}_{t\wedge \tau_{R}^{n}}\|_{\mathbb{V}}ds \bigg)dt\nonumber \\
		\!\!\!\!\!\!\!\!&&=\mathbb{E}^{(n)} \int_{0}^{T\wedge\tau_{R}^{n}}\int_{0\vee (s-\delta)}^{s} \|\mathcal{A}(s, X^{(n)}_s, \mathscr{L}_{X^{(n)}_s})\|_{\mathbb{V}^*} \|X^{(n)}_{t\wedge \tau_{R}^{n}}\|_{\mathbb{V}}dtds\nonumber \\
		\!\!\!\!\!\!\!\!&&\lesssim \delta^{\frac{\alpha-1}{\alpha}}\mathbb{E}^{(n)} \int_{0}^{T\wedge\tau_{R}^{n}}  \|\mathcal{A}(s, X^{(n)}_s, \mathscr{L}_{X^{(n)}_s})\|_{\mathbb{V}^*} \bigg( \int_{0}^{T\wedge \tau_{R}^{n}}\|  X^{(n)}_{t}\|_{\mathbb{V}}^{\alpha}dt \bigg)^{\frac{1}{\alpha}}ds \nonumber \\
		\!\!\!\!\!\!\!\!&&\lesssim_{R,T} \delta^{\frac{\alpha-1}{\alpha}}\bigg\{ \mathbb{E}^{(n)}\int_{0}^{T\wedge\tau_{R}^{n}}\|\mathcal{A}(s, X^{(n)}_s, \mathscr{L}_{X^{(n)}_s})\|_{\mathbb{V}^*}^{\frac{\alpha}{\alpha-1}}ds \bigg\}^{\frac{\alpha-1}{\alpha}}\nonumber \\
		\!\!\!\!\!\!\!\!&&\lesssim_{R,T} \delta^{\frac{\alpha-1}{\alpha}}\bigg\{ \mathbb{E}^{(n)}\int_0^{T\wedge\tau_{R}^{n}} (1 + \|X^{(n)}_s\|_{\mathbb{V}}^{\alpha} + \mathscr{L}_{X^{(n)}_s}(\|\cdot\|_{\mathbb{V}}^{\alpha}))
		\nonumber \\
		\!\!\!\!\!\!\!\!&&~~~~~~~~~
		\cdot (1 + \|X^{(n)}_s\|_{\mathbb{H}}^{\beta} + \mathscr{L}_{X^{(n)}_s}(\|\cdot\|_{\mathbb{H}}^{\beta}))ds \bigg\}^{\frac{\alpha-1}{\alpha}}\nonumber\\
		\!\!\!\!\!\!\!\!&&\lesssim_{R,T} \delta^{\frac{\alpha-1}{\alpha}}.
	\end{eqnarray}
	By $(\mathbf{A}_3)$ and Lemma \ref{Yesti2}, the term $\mathscr{I}_2$ can be estimated as
	\begin{eqnarray}\label{b22}
		\mathscr{I}_2 \!\!\!\!\!\!\!\!&&\lesssim \int_{0}^{T}\mathbb{E}^{(n)} \bigg(\sup _{0 \leq \Lambda \leq \delta}\int_{t\wedge \tau_{R}^{n}}^{(t+\Lambda )\wedge \tau_{R}^{n}} \big(1 + \|X^{(n)}_s\|_{\mathbb{H}}^{2} + \mathscr{L}_{X^{(n)}_s}(\|\cdot\|_{\mathbb{H}}^{2})\big)ds \bigg)dt\nonumber \\
		\!\!\!\!\!\!\!\!&&\lesssim \mathbb{E}^{(n)} \int_{0}^{T\wedge\tau_{R}^{n}}\int_{0\vee (s-\delta)}^{s} \big(1 + \|X^{(n)}_s\|_{\mathbb{H}}^{2} + \mathscr{L}_{X^{(n)}_s}(\|\cdot\|_{\mathbb{H}}^{2})\big)dtds\nonumber\\
		\!\!\!\!\!\!\!\!&&\lesssim_R\delta.
	\end{eqnarray}
	By $(\mathbf{A}_4)$ and Lemma \ref{Yesti2}, using B-D-G's inequality and H\"older's inequality,  we have
	\begin{eqnarray}\label{b23}
		\mathscr{I}_3 \!\!\!\!\!\!\!\!&&\lesssim \int_{0}^{T}\mathbb{E}^{(n)} \bigg(\int_{t\wedge \tau_{R}^{n}}^{(t+\delta)\wedge \tau_{R}^{n}} \| \mathcal{B}(s, X^{(n)}_s, \mathscr{L}_{X^{(n)}_s}) \|_{L_2(U;\mathbb{H})}^2 \| X^{(n)}_s- X^{(n)}_{t\wedge \tau_{R}^{n}}\|_{\mathbb{H}}^2ds \bigg)^{\frac{1}{2}}dt\nonumber \\
		\!\!\!\!\!\!\!\!&&\lesssim_R \bigg(\mathbb{E}^{(n)} \int_{0}^{T\wedge\tau_{R}^{n}}\int_{0\vee (s-\delta)}^{s} \| \mathcal{B}(s, X^{(n)}_s, \mathscr{L}_{X^{(n)}_s}) \|_{L_2(U;\mathbb{H})}^2dtds\bigg)^{\frac{1}{2}}\nonumber \\
		\!\!\!\!\!\!\!\!&&\lesssim_R\delta^{\frac{1}{2}} \bigg(\mathbb{E}^{(n)} \int_{0}^{T\wedge\tau_{R}^{n}}(1 + \|X^{(n)}_s\|_{\mathbb{H}}^{\beta} + \mathscr{L}_{X^{(n)}_s}(\|\cdot\|_{\mathbb{H}}^{\beta}))ds\bigg)^{\frac{1}{2}}\nonumber\\
		\!\!\!\!\!\!\!\!&&\lesssim_{R,T}\delta^{\frac{1}{2}}.
	\end{eqnarray}
	Substituting (\ref{b21})-(\ref{b23}) into (\ref{b2}) yields that (\ref{b1}) holds.
	
	\vspace{1mm}
	\noindent\textbf{Case 2:} ($\alpha >2$).
	Using  It\^{o}'s formula, it follows that
	\begin{eqnarray*}
		\!\!\!\!\!\!\!\!&& \|X^{(n)}_{(t+\Lambda )\wedge \tau_{R}^{n} }-X^{(n)}_{t\wedge \tau_{R}^{n} }\|_{\mathbb{H}}^\alpha\nonumber\\
		=\!\!\!\!\!\!\!\!&&\alpha\int_{t\wedge \tau_{R}^{n}}^{(t+\Lambda )\wedge \tau_{R}^{n}} \|X^{(n)}_{s}-X^{(n)}_{t\wedge \tau_{R}^{n} }\|^{\alpha-2} \,_{\mathbb{V}^*}\langle P_{n} \mathcal{A}(s, X^{(n)}_s, \mathscr{L}_{X^{(n)}_s}), X^{(n)}_{s}-X^{(n)}_{t\wedge \tau_{R}^{n} } \rangle_{\mathbb{V}} ds\nonumber\\
		\!\!\!\!\!\!\!\!&&+\frac{\alpha }{2}\int_{t\wedge \tau_{R}^{n}}^{(t+\Lambda )\wedge \tau_{R}^{n}} \|X^{(n)}_{s}-X^{(n)}_{t\wedge \tau_{R}^{n} }\|^{\alpha-2}\|P_n \mathcal{B}(s, X^{(n)}_s, \mathscr{L}_{X^{(n)}_s}) \|_{L_{2}(U; \mathbb{H})}^{2}ds\nonumber\\
		\!\!\!\!\!\!\!\!&&+\frac{\alpha(\alpha-2)}{2}\int_{t\wedge \tau_{R}^{n}}^{(t+\Lambda )\wedge \tau_{R}^{n}}\Big\{\|X^{(n)}_{s }-X^{(n)}_{t\wedge \tau_{R}^{n} }\|_{\mathbb{H}}^{\alpha-4}
		\nonumber\\
		\!\!\!\!\!\!\!\!&&\cdot
		\|(P_n \mathcal{B}(s, X^{(n)}_s, \mathscr{L}_{X^{(n)}_s}) )^*(X^{(n)}_{s}-X^{(n)}_{t\wedge \tau_{R}^{n} })\|_{U}^{2}\Big\}ds\nonumber\\
		\!\!\!\!\!\!\!\!&&+
		\alpha\int_{t\wedge \tau_{R}^{n}}^{(t+\Lambda )\wedge \tau_{R}^{n}}\|X^{(n)}_{s}-X^{(n)}_{t\wedge \tau_{R}^{n} }\|^{\alpha-2}\langle P_n \mathcal{B}(s, X^{(n)}_s, \mathscr{L}_{X^{(n)}_s})  d W^{(n)}_{s}, X^{(n)}_s- X^{(n)}_{t\wedge \tau_{R}^{n}}\rangle_{\mathbb{H}}.
	\end{eqnarray*}
	Then it follows that
	\begin{eqnarray}\label{b3}
		\!\!\!\!\!\!\!\!&&\mathbb{E}^{(n)} \bigg(\sup _{0 \leq \Lambda \leq \delta}\int_{0}^{T-\Lambda } \|X^{(n)}_{(t+\Lambda )\wedge \tau_{R}^{n} }-X^{(n)}_{t\wedge \tau_{R}^{n} }\|_{\mathbb{H}}^\alpha dt\bigg)\nonumber\\
		\lesssim\!\!\!\!\!\!\!\!&&\int_{0}^{T}\mathbb{E}^{(n)} \Big(\sup _{0 \leq \Lambda \leq \delta}\|X^{(n)}_{(t+\Lambda )\wedge \tau_{R}^{n} }-X^{(n)}_{t\wedge \tau_{R}^{n} }\|_{\mathbb{H}}^\alpha \Big)dt\nonumber\\
		\lesssim\!\!\!\!\!\!\!\!&&\int_{0}^{T}\mathbb{E}^{(n)} \bigg(\sup _{0 \leq \Lambda \leq \delta}\Big|\int_{t\wedge \tau_{R}^{n}}^{(t+\Lambda )\wedge \tau_{R}^{n}} \|X^{(n)}_{s}-X^{(n)}_{t\wedge \tau_{R} }\|_{\mathbb{H}}^{\alpha-2}
		\nonumber\\
		\!\!\!\!\!\!\!\!&&\cdot\,_{\mathbb{V}^*}\langle \mathcal{A}(s, X^{(n)}_s, \mathscr{L}_{X^{(n)}_s}),  X^{(n)}_{t\wedge \tau_{R}^{n}} \rangle_{\mathbb{V}} ds\Big|\bigg)dt\nonumber\\
		\!\!\!\!\!\!\!\!&&+\int_{0}^{T}\mathbb{E}^{(n)} \bigg(\sup _{0 \leq \Lambda \leq \delta}\int_{t\wedge \tau_{R}^{n}}^{(t+\Lambda )\wedge \tau_{R}^{n}}\|X^{(n)}_{s}-X^{(n)}_{t\wedge \tau_{R} }\|_{\mathbb{H}}^{\alpha-2}\nonumber\\
		\!\!\!\!\!\!\!\!&&\cdot \big(2\,_{\mathbb{V}^*}\langle  \mathcal{A}(s, X^{(n)}_s, \mathscr{L}_{X^{(n)}_s}), X^{(n)}_s \rangle_{\mathbb{V}}+\| \mathcal{B}(s, X^{(n)}_s, \mathscr{L}_{X^{(n)}_s}) \|_{L_{2}(U; \mathbb{H})}^{2}\big)ds\bigg)dt\nonumber\\
		\!\!\!\!\!\!\!\!&&+\int_{0}^{T}\mathbb{E}^{(n)} \bigg(\sup _{0 \leq \Lambda \leq \delta}\int_{t\wedge \tau_{R}^{n}}^{(t+\Lambda )\wedge \tau_{R}^{n}}\|X^{(n)}_{s }-X^{(n)}_{t\wedge \tau_{R}^{n} }\|_{\mathbb{H}}^{\alpha-4}
		\nonumber\\
		\!\!\!\!\!\!\!\!&&\cdot
		\|(P_n \mathcal{B}(s, X^{(n)}_s, \mathscr{L}_{X^{(n)}_s}) )^*(X^{(n)}_{s}-X^{(n)}_{t\wedge \tau_{R}^{n} })\|_{U}^{2}ds\bigg)dt\nonumber\\
		\!\!\!\!\!\!\!\!&&+\int_{0}^{T}\mathbb{E}^{(n)} \bigg(\sup _{0 \leq \Lambda \leq \delta}\Big|\int_{t\wedge \tau_{R}^{n}}^{(t+\Lambda )\wedge \tau_{R}^{n}}\|X^{(n)}_{s}-X^{(n)}_{t\wedge \tau_{R}^{n} }\|_{\mathbb{H}}^{\alpha-2}
		\nonumber\\
		\!\!\!\!\!\!\!\!&&\cdot
		\langle P_n \mathcal{B}(s, X^{(n)}_s, \mathscr{L}_{X^{(n)}_s})  d W^{(n)}_{s}, X^{(n)}_s- X^{(n)}_{t\wedge \tau_{R}^{n}}\rangle_{\mathbb{H}}\Big|\bigg)dt\nonumber\\
		=:\!\!\!\!\!\!\!\!&&\sum_{i=1}^{4} \mathscr{J}_i.
	\end{eqnarray}
	Similar to the proof of (\ref{b21})-(\ref{b23}), due to the definition of stopping time $\tau_{R}^{n}$, it can be inferred that
	\begin{eqnarray*}
		\mathscr{J}_1\!\!\!\!\!\!\!\!&&\lesssim_{R,T} \delta^{\frac{\alpha-1}{\alpha}},\\
		\mathscr{J}_2\!\!\!\!\!\!\!\!&&\lesssim_{R,T}\delta,\\
		\mathscr{J}_4\!\!\!\!\!\!\!\!&&\lesssim_{R,T}\delta^{\frac{1}{2}}.
	\end{eqnarray*}
	By $(\mathbf{A}_4)$ and Lemma \ref{Yesti2}, the term $\mathscr{J}_3$ can be estimated as follows
	\begin{eqnarray*}
		\mathscr{J}_3\!\!\!\!\!\!\!\!&&\lesssim \mathbb{E}^{(n)} \int_{0}^{T}\int_{t\wedge \tau_{R}^{n}}^{(t+\delta)\wedge \tau_{R}^{n}} \|X^{(n)}_{s }-X^{(n)}_{t\wedge \tau_{R}^{n} }\|_{\mathbb{H}}^{\alpha-2}\big(1 + \|X^{(n)}_s\|_{\mathbb{H}}^{\beta} + \mathscr{L}_{X^{(n)}_s}(\|\cdot\|_{\mathbb{H}}^{\beta})\big)dsdt \nonumber \\
		\!\!\!\!\!\!\!\!&&\lesssim_R \delta  \mathbb{E}^{(n)}\int_{0 }^{T\wedge \tau_{R}^{n}} \big(1 + \|X^{(n)}_s \|_{\mathbb{H}}^\beta + \mathbb{E}^{(n)}\| X^{(n)}_s\|_{\mathbb{H}}^\beta\big)ds\nonumber\\
		\!\!\!\!\!\!\!\!&&\lesssim_{R,T} \delta.
	\end{eqnarray*}
	Hence, in view of (\ref{b3}), we can conclude that (\ref{ball}) is satisfied.
\end{proof}

Set $$\mathcal{Z}_T:=\mathbb{C}_T(\mathbb{V}^{*}) \cap L^\alpha ([0, T]; \mathbb{H}).$$
We derive the following corollary.
\begin{corollary}\label{coro1}
	$\{\mathscr{L}_{X^{(n)}}\}_{n=1}^{\infty}$ is tight in $\mathscr{P}(\mathcal{Z}_T)$.
\end{corollary}

\begin{proof}
	Due to Lemmas \ref{tightness} and \ref{tightness2}, for any sequence $\{\mathscr{L}_{X^{(n)}}\}_{n=1}^{\infty}$ we can find a subsequence  $\{\mathscr{L}_{X^{(n_k)}}\}_{k=1}^{\infty}$ which  converges weakly in $\mathscr{P}(\mathcal{Z}_T)$. Hence, the assertion follows.
\end{proof}

\subsection{Proof of Theorem \ref{th1}}\label{sec3.4}
Set
$$
\Upsilon := \mathcal{Z}_T \times \mathbb{C}_T(U_1),
$$
where $U_1$ is a Hilbert space such that the embedding $U \subset U_1$ is Hilbert--Schmidt. By Corollary \ref{coro1} and the Skorohod representation theorem, there exists a probability space
$(\tilde{\Omega}, \tilde{\mathscr{F}},  \tilde{\mathbb{P}}) $ and a sequence of $\Upsilon$-valued random vectors $ \{(\tilde{X}^{(n)}, \tilde{W}^{(n)})\}_{n=1}^{\infty}$ and $(\tilde{X}, \tilde{W})$ such that

\vspace{1mm}
(i) $\mathscr{L}_{(\tilde{X}^{(n)}, \tilde{W}^{(n)})}|_{\tilde{\mathbb{P}}} = \mathscr{L}_{(X^{(n)}, W^{(n)})}|_{\mathbb{P}};$

\vspace{1mm}
(ii) $\tilde{W}^{(n)}\xrightarrow{n\to \infty}\tilde{W}$ in  $\mathbb{C}_T(U_1)$, ~~$\tilde{\mathbb{P}}$-a.s.;

\vspace{1mm}
(iii) $ \tilde{X}^{(n)}\xrightarrow{n\to \infty}\tilde{X}$ in  $ \mathcal{Z}_T$,  ~~$\tilde{\mathbb{P}}$-a.s..

\vspace{1mm}
Let $(\tilde{\mathscr{F}}^{(n)}_t)_{t\geq 0}$ be the usual filtration generated by
$\big\{\tilde{X}^{(n)}_s,\tilde{W}^{(n)}_s:s\in[0,t]\big\}.$
Notice that by the claim (i),
\begin{eqnarray*}
	\!\!\!\!\!\!\!\!&&\mathbb{P}(W^{(n)}_t-W^{(n)}_s\in\cdot|\mathscr{F}^{(n)}_s)=\mathbb{P}(W^{(n)}_t-W^{(n)}_s\in\cdot)
	\\
	\Rightarrow\!\!\!\!\!\!\!\!&&\tilde{\mathbb{P}}(\tilde{W}^{(n)}_t-\tilde{W}^{(n)}_s\in\cdot|\tilde{\mathscr{F}}^{(n)}_s)=\tilde{\mathbb{P}}(\tilde{W}^{(n)}_t-\tilde{W}^{(n)}_s\in\cdot).
\end{eqnarray*}
Namely, $\tilde{W}^{(n)}$ is an $(\tilde{\mathscr{F}}^{(n)}_t)$-cylindrical Wiener process on $U$.

Let $(\tilde{\mathscr{F}}_t)_{t\geq 0}$ be the usual filtration generated by
$\big\{\tilde{X}_s,\tilde{W}_s:s\in[0,t]\big\}.$ In the following, we will show that $(\tilde{X}, \tilde{W})$ is a solution of (\ref{eqSPDE}). From the equation (\ref{approeq}), it follows that
\begin{equation*}
	d\tilde{X}^{(n)}_t = P_n \mathcal{A}(t, \tilde{X}^{(n)}_t, \mathscr{L}_{\tilde{X}^{(n)}_t}) dt + P_n \mathcal{B}(t, \tilde{X}^{(n)}_t, \mathscr{L}_{\tilde{X}^{(n)}_t}) d\tilde{W}^{(n)}_t,~\tilde{X}^{(n)}_0\sim\mu_0\circ P_n^{-1}.
\end{equation*}
Moreover, Lemmas \ref{Yesti1}-\ref{nIPSesti2} imply that for any $\gamma \in (0,1)$ and $l \in (1,\frac{p}{\beta}] \cap (1,\frac{p}{2})$
\begin{eqnarray}\label{estitilde}
	\!\!\!\!\!\!\!\!&&\sup_{t\in[0,T]}\tilde{\mathbb{E}}\|\tilde{X}^{(n)}_t\|_{\mathbb{H}}^{ p}+\tilde{\mathbb{E}}\Big[\sup_{t\in[0,T]}\|\tilde{X}^{(n)}_t\|_{\mathbb{H}}^{\gamma p}\Big]\nonumber\\
	\!\!\!\!\!\!\!\!&&+\tilde{\mathbb{E}}\Big(\int_{0}^{T}\|\tilde{X}^{(n)}_t\|_{\mathbb{V}}^{\alpha }dt\Big)^l+\tilde{\mathbb{E}}\int_0^T\|\tilde{X}^{(n)}_t\|_{\mathbb{H}}^{p-2}\|\tilde{X}^{(n)}_t\|_{\mathbb{V}}^{\alpha}dt
	<\infty.
\end{eqnarray}
Since $\|\cdot\|_\mathbb{V}$ is lower semicontinuous in $\mathbb{H}$, by Fatou's lemma we derive
\begin{eqnarray}\label{estiX}
	\!\!\!\!\!\!\!\!&&\sup_{t\in[0,T]}\tilde{\mathbb{E}}\|\tilde{X}_t\|_{\mathbb{H}}^{ p}+\tilde{\mathbb{E}}\Big[\sup_{t\in[0,T]}\|\tilde{X}_t\|_{\mathbb{H}}^{\gamma p}\Big]\nonumber\\
	\!\!\!\!\!\!\!\!&&+\tilde{\mathbb{E}}\Big(\int_{0}^{T}\|\tilde{X}_t\|_{\mathbb{V}}^{\alpha }dt\Big)^l+\tilde{\mathbb{E}}\int_0^T\|\tilde{X}_t\|_{\mathbb{H}}^{p-2}\|\tilde{X}_t\|_{\mathbb{V}}^{\alpha}dt
	<\infty.
\end{eqnarray}

As a consequence of (\ref{estitilde}), there exists a subsequence of $\{n\}$, denoted again by $\{n\}$, such that for $n \to \infty$,

\vspace{1mm}
(i)
$\tilde{X}^{(n)} \to \bar{X}$ weakly in $ L^\alpha([0, T] \times \tilde{\Omega}; \mathbb{V})$ and weakly star in $L^p(\tilde{\Omega}; L^\infty([0, T]; \mathbb{H})) $;

\vspace{1mm}
(ii) $G^{(n)} := \mathcal{A}(\cdot, \tilde{X}^{(n)}_{\cdot}, \mathscr{L}_{\tilde{X}^{(n)}_{\cdot}}) \rightharpoonup G$ in $L^{\frac{\alpha}{\alpha - 1}}([0, T] \times \tilde{\Omega}; \mathbb{V}^*)$.


\vspace{1mm}
The following lemma establishes the convergence of stochastic integral.
\begin{lemma}\label{B}
	Along a subsequence of $\{n\}$, we have
	\begin{equation*}\label{B-con}
		\lim_{n \to\infty }\tilde{\mathbb{E}}\Bigg\{\sup_{t \in [0,T]}\Big\|\int_0^t P_n \mathcal{B}(t,\tilde{X}^{(n)}_t,\mathscr{L}_{\tilde{X}^{(n)}_t})d\tilde{W}^{(n)}_s-\int_0^t \mathcal{B}(t,\tilde{X}_t,\mathscr{L}_{\tilde{X}_t})d\tilde{W}_s\Big\|_\mathbb{H}^2\Bigg\}=0.
	\end{equation*}
\end{lemma}
\begin{proof}
	Since $\|\tilde{X}^{(n)} - \tilde{X}\|_{L^\alpha([0,T];\mathbb{H})} \to 0$ $\mathbb{P}$-a.s., we can find a subsequence still denoted by $\{\tilde{X}^{(n)}\}$ such that
	\begin{equation}\label{XconverH}
		\lim_{n \to \infty} \|\tilde{X}^{(n)}_t- \tilde{X}_t\|_\mathbb{H} = 0 \quad dt \times \tilde{\mathbb{P}}\text{-a.e.}.
	\end{equation}
	By assumption $(\mathbf{A}_1)$, we have
	\begin{equation}\label{Bcon}
		\lim_{n\to\infty}\| \mathcal{B}(t,\tilde{X}^{(n)}_t,\mathscr{L}_{\tilde{X}^{(n)}_t})-\mathcal{B}(t,\tilde{X}_t,\mathscr{L}_{\tilde{X}_t})\|_{L_2(U;\mathbb{H})}=0 \quad dt \times \tilde{\mathbb{P}}\text{-a.e.}.
	\end{equation}
	Thus, the estimates (\ref{estitilde}) and (\ref{estiX}) and Vitali's convergence theorem yield that
	\begin{equation}\label{Bintcon}
		\lim_{n\to\infty}\tilde{\mathbb{E}}\int_0^T\| \mathcal{B}(t,\tilde{X}^{(n)}_t,\mathscr{L}_{\tilde{X}^{(n)}_t})-\mathcal{B}(t,\tilde{X}_t,\mathscr{L}_{\tilde{X}_t})\|^2_{L_2(U;\mathbb{H})}dt=0.
	\end{equation}
	Then it is easy to get
	\begin{eqnarray*}
		\!\!\!\!\!\!\!\!&&\lim_{n\to\infty}\tilde{\mathbb{E}}\int_0^T\| P_n\mathcal{B}(t,\tilde{X}^{(n)}_t,\mathscr{L}_{\tilde{X}^{(n)}_t})-\mathcal{B}(t,\tilde{X}_t,\mathscr{L}_{\tilde{X}_t})\|^2_{L_2(U;\mathbb{H})}dt\\
		\lesssim\!\!\!\!\!\!\!\!&&\lim_{n\to\infty}\tilde{\mathbb{E}}\int_0^T\| P_n\big(\mathcal{B}(t,\tilde{X}^{(n)}_t,\mathscr{L}_{\tilde{X}^{(n)}_t})-\mathcal{B}(t,\tilde{X}_t,\mathscr{L}_{\tilde{X}_t})\big)\|^2_{L_2(U;\mathbb{H})}dt\\
		\!\!\!\!\!\!\!\!&&+\lim_{n\to\infty}\tilde{\mathbb{E}}\int_0^T\| (P_n-I)\mathcal{B}(t,\tilde{X}_t,\mathscr{L}_{\tilde{X}_t})\|^2_{L_2(U;\mathbb{H})}dt\\
		\lesssim\!\!\!\!\!\!\!\!&&\lim_{n\to\infty}\tilde{\mathbb{E}}\int_0^T\| \mathcal{B}(t,\tilde{X}_t,\mathscr{L}_{\tilde{X}_t})^*(P_n-I)\|^2_{L_2(\mathbb{H};U)}dt.
	\end{eqnarray*}
	Recall $P_n|_\mathbb{H}$ is an orthogonal projection onto $\mathbb{H}_n$ on $\mathbb{H}$. Then by  the definition of  Hilbert--Schmidt operators and using Fubini's theorem, it follows that
	\begin{eqnarray*}
		\!\!\!\!\!\!\!\!&&\lim_{n\to\infty}\tilde{\mathbb{E}}\int_0^T\| \mathcal{B}(t,\tilde{X}_t,\mathscr{L}_{\tilde{X}_t})^*(P_n-I)\|^2_{L_2(\mathbb{H};U)}dt\\
		\!\!\!\!\!\!\!\!&&=\lim_{n \to \infty} \tilde{\mathbb{E}} \int_0^T \sum_{i=1}^\infty \| \mathcal{B}(t, \tilde{X}_t, \mathscr{L}_{\tilde{X}_t})^*(P_n-I)e_i \|_{U}^2 dt \nonumber\\
		\!\!\!\!\!\!\!\!&&= \lim_{n \to \infty} \tilde{\mathbb{E}} \int_0^T \sum_{i > n} \| \mathcal{B}(t, \tilde{X}_t, \mathscr{L}_{\tilde{X}_t})^* e_i \|_{U}^2 dt .
	\end{eqnarray*}
	Due to $(\mathbf{A}_4)$ and (\ref{estiX}), it follows that
	\begin{equation*}
		\tilde{\mathbb{E}} \int_0^T\sum_{i=1}^\infty\| \mathcal{B}(t, \tilde{X}_t, \mathscr{L}_{\tilde{X}_t})^*  e_i \|_{U}^2 dt=\tilde{\mathbb{E}} \int_0^T \| \mathcal{B}(t, \tilde{X}_t, \mathscr{L}_{\tilde{X}_t}) \|_{L_2(U;\mathbb{H})}^2 dt
		< \infty,
	\end{equation*}
	then
	\begin{equation*}
		\lim_{n\to\infty}\tilde{\mathbb{E}}\int_0^T\| \mathcal{B}(t,\tilde{X}_t,\mathscr{L}_{\tilde{X}_t})^*\big(P_n-I\big)\|^2_{L_2(\mathbb{H};U)}dt=0,
	\end{equation*}
	which implies that
	\begin{equation}\label{BPn}
		\lim_{n\to\infty}\tilde{\mathbb{E}}\int_0^T\| P_n\mathcal{B}(t,\tilde{X}^{(n)}_t,\mathscr{L}_{\tilde{X}^{(n)}_t})-\mathcal{B}(t,\tilde{X}_t,\mathscr{L}_{\tilde{X}_t})\|^2_{L_2(U;\mathbb{H})}dt.
	\end{equation}

	According to Lemma 4.3 in \cite{BMX}, by (\ref{BPn})  and the convergence of ${W}^{(n)}$ we can deduce that
	\begin{equation*}
		\lim_{n\to\infty}\sup_{t \in [0,T]}\Big\|\int_0^t P_n \mathcal{B}(t,\tilde{X}^{(n)}_t,\mathscr{L}_{\tilde{X}^{(n)}_t})d\tilde{W}^{(n)}_s-\int_0^t \mathcal{B}(t,\tilde{X}_t,\mathscr{L}_{\tilde{X}_t})d\tilde{W}_s \Big\|_\mathbb{H}^2=0\quad \text{in probability}.
	\end{equation*}
	From (\ref{conb}), (\ref{estitilde}) and (\ref{estiX}),  using  B-D-G's inequality yields that there exists a constant $p_0>1$ such that
	\begin{equation*}
		\tilde{\mathbb{E}}\Bigg\{\sup_{t \in [0,T]}\Big\|\int_0^t P_n \mathcal{B}(t,\tilde{X}^{(n)}_t,\mathscr{L}_{\tilde{X}^{(n)}_t})d\tilde{W}^{(n)}_s-\int_0^t \mathcal{B}(t,\tilde{X}_t,\mathscr{L}_{\tilde{X}_t})d\tilde{W}_s\Big\|_\mathbb{H}^2\Bigg\}^{p_0}<\infty.
	\end{equation*}
	Then the Vitali's convergence theorem yields
	\begin{equation*}\label{B-con}
		\lim_{n \to\infty }\tilde{\mathbb{E}}\Bigg\{\sup_{t \in [0,T]}\Big\|\int_0^t P_n \mathcal{B}(t,\tilde{X}^{(n)}_t,\mathscr{L}_{\tilde{X}^{(n)}_t})d\tilde{W}^{(n)}_s-\int_0^t \mathcal{B}(t,\tilde{X}_t,\mathscr{L}_{\tilde{X}_t})d\tilde{W}_s\Big\|_\mathbb{H}^2\Bigg\}=0.
	\end{equation*}
	We complete the proof.
\end{proof}

Let us define
\begin{equation}\label{eq5}
	X_t := \tilde{X}_0 + \int_0^t G_s ds + \int_0^t \mathcal{B}(t, \tilde{X}_{t}, \mathscr{L}_{\tilde{X}_{t}}) d\tilde{W}_s, \quad t \in [0, T].
\end{equation}
Then by a standard argument (cf.~\cite[Page 105]{RSZ}), it is straightforward that
$$X = \bar{X} = \tilde{X}\quad dt \times \tilde{\mathbb{P}}\text{-a.e.}.$$
From now on,  without loss of generality, we drop all the superscripts to simplify the notation, for example, we write $\tilde{X}^{(n)}$ as $X^{(n)}$.  By Theorem 4.2.5 in \cite{liu2015stochastic},  $X$ is an $\mathbb{H}$-valued continuous $ (\mathscr{F}_t)$-adapted process.

Now it is sufficient to prove that
$$
\mathcal{A}(t, X_{t}, \mathscr{L}_{X_{t}}) = G_{t} \quad dt \times \mathbb{P}\text{-a.e.}.
$$

\begin{lemma}\label{A}
	Under the assumptions in Theorem \ref{th1}, if
	\begin{equation*}
		X^{(n)} \rightharpoonup X \text{ in }  L^\alpha([0, T] \times \Omega; \mathbb{V}),
	\end{equation*}
	\begin{equation}\label{Acondition2}
		\mathcal{A}(\cdot, X^{(n)}_{\cdot}, \mathscr{L}_{X^{(n)}_{\cdot}}) \rightharpoonup G_{\cdot}  \text{ in } L^{\frac{\alpha}{\alpha - 1}}([0, T] \times \Omega; \mathbb{V}^*),
	\end{equation}
	\begin{equation}\label{Acondition3}
		\liminf_{n \to \infty} \mathbb{E} \int_0^T {_{\mathbb{V}^*}\langle \mathcal{A}(t, X^{(n)}_t, \mathscr{L}_{X^{(n)}_t}), X^{(n)}_t \rangle_{\mathbb{V}}}dt \geq \mathbb{E} \int_0^T{ _{\mathbb{V}^*}\langle G_t, X_t \rangle_{\mathbb{V}}}dt,
	\end{equation}
	then  we have
	$$\mathcal{A}(t, X_{t}, \mathscr{L}_{X_{t}}) = G_{t}\quad dt \times \mathbb{P}\text{-a.e.}.$$
\end{lemma}
\begin{proof}
	Note that
	\begin{eqnarray*}
		\!\!\!\!\!\!\!\!&&{_{\mathbb{V}^*}\langle \mathcal{A}(t, X^{(n)}_t, \mathscr{L}_{X^{(n)}_t}), X^{(n)}_t-X_t \rangle_{\mathbb{V}}}\\
		=\!\!\!\!\!\!\!\!&&{_{\mathbb{V}^*}\langle \mathcal{A}(t, X^{(n)}_t, \mathscr{L}_{X^{(n)}_t}), X^{(n)}_t \rangle_{\mathbb{V}}}-{_{\mathbb{V}^*}\langle \mathcal{A}(t, X^{(n)}_t, \mathscr{L}_{X^{(n)}_t}), X_t \rangle_{\mathbb{V}}}\\
		\leq\!\!\!\!\!\!\!\!&& {_{\mathbb{V}^*}\langle \mathcal{A}(t, X^{(n)}_t, \mathscr{L}_{X^{(n)}_t}), X^{(n)}_t \rangle_{\mathbb{V}}}+\|\mathcal{A}(t, X^{(n)}_t, \mathscr{L}_{X^{(n)}_t})\|_{\mathbb{V}^*}\|X_t\|_{\mathbb{V}}.
	\end{eqnarray*}
	We mention that in view of the estimates (\ref{estitilde}), it implies
	\begin{equation*}
		\mathscr{L}_{X_{t}^{(n)}} \in \mathscr{P}_{\alpha}(\mathbb{V}) \cap \mathscr{P}_{p}(\mathbb{H})\quad d t \text {-a.e.. }
	\end{equation*}
	Using $(\mathbf{A}_3)$-$(\mathbf{A}_4)$ and the estimates (\ref{estitilde}) yields
	\begin{eqnarray*}
		\!\!\!\!\!\!\!\!&&{_{\mathbb{V}^*}\langle \mathcal{A}(t, X^{(n)}_t, \mathscr{L}_{X^{(n)}_t}), X^{(n)}_t-X_t \rangle_{\mathbb{V}}}\nonumber\\
		\leq\!\!\!\!\!\!\!\!&& -\delta\|X^{(n)}_t\|_{\mathbb{V}}^\alpha+C\big(1+\|X^{(n)}_t\|_{\mathbb{H}}^2+\mathscr{L}_{X^{(n)}_t}(\|\cdot\|_{\mathbb{H}}^2)\big)\nonumber\\
		\!\!\!\!\!\!\!\!&&+C\Big[\big(1+\|X^{(n)}_t\|_{\mathbb{V}}^{\alpha}+\mathscr{L}_{X^{(n)}_t}(\|\cdot\|_{\mathbb{V}}^{\alpha} )\big)\big(1+\|X^{(n)}_t\|_{{\mathbb{H}}}^{\beta}+\mathscr{L}_{X^{(n)}_t}(\|\cdot\|_{\mathbb{H}}^{\beta})\big)\Big]^{\frac{\alpha -1}{\alpha } }\|X_t\|_{\mathbb{V}}\nonumber\\
		\leq\!\!\!\!\!\!\!\!&& -(\delta-\epsilon _0)\|X^{(n)}_t\|_{\mathbb{V}}^\alpha+\epsilon _0\mathbb{E}\|X^{(n)}_t\|_{\mathbb{V}}^\alpha+C_{p,T}\big(1+\|X^{(n)}_t\|_{\mathbb{H}}^2+(1+\|X^{(n)}_t\|_{{\mathbb{H}}}^{\beta(\alpha -1)})\|X_t\|_{\mathbb{V}}^{\alpha}\big),
	\end{eqnarray*}
	where the last step used Young's inequality and $\epsilon _0 \in (0,\delta/2).$
	For the sake of convenience, set
	\begin{equation*}
		L^{(n)}_t:={_{\mathbb{V}^*}\langle \mathcal{A}(t, X^{(n)}_t, \mathscr{L}_{X^{(n)}_t}), X^{(n)}_t-X_t \rangle_{\mathbb{V}}},
	\end{equation*}
	\begin{equation*}
		F^{(n)}_t:=C_{p,T}\big(1+\|X^{(n)}_t\|_{\mathbb{H}}^2+(1+\|X^{(n)}_t\|_{{\mathbb{H}}}^{\beta(\alpha -1)})\|X_t\|_{\mathbb{V}}^{\alpha}\big).
	\end{equation*}
	Then it follows that
	\begin{equation}\label{Aall}
		L^{(n)}_t\leq-(\delta-\epsilon _0)\|X^{(n)}_t\|_{\mathbb{V}}^\alpha+\epsilon _0\mathbb{E}\|X^{(n)}_t\|_{\mathbb{V}}^\alpha+F^{(n)}_t.
	\end{equation}
	
	We proceed to prove the lemma by the following three steps.
	
	\noindent\textbf{Claim 1.} We claim that
	\begin{equation}\label{A1all}
		\limsup_{n \to \infty} L^{(n)}_t \leq 0\quad dt \times  \mathbb{P}\text{-a.e.}.
	\end{equation}
	
	Firstly, it follows from (\ref{XconverH}) that there exists a null set $\aleph$ such that for any $(t,\omega)\in([0, T] \times \Omega ) \backslash \aleph$,
	\begin{equation}\label{A11}
		X^{(n)}_t(\omega) \xrightarrow{n\to \infty} X_t(\omega) \quad\text{ in } \mathbb{H}.
	\end{equation}
	Now we prove (\ref{A1all}) by contradiction. More precisely, for fixed $(t,\omega)\in([0, T] \times \Omega ) \backslash \aleph$ we assume
	\begin{equation*}
		\limsup_{n \to \infty} L^{(n)}_t(\omega) > 0.
	\end{equation*}
	Then there exists a subsequence $\{n_k\}_{k=1}^{\infty}$ such that
	\begin{equation}\label{A12}
		\lim_{k \to \infty} L^{(n_k)}_t(\omega) > 0.
	\end{equation}
	
	By (\ref{estitilde}) and (\ref{A11}) we can get
	\begin{equation}\label{Aall0}
		\mathscr{L}_{X^{(n_k)}_t} \xrightarrow{k\to \infty} \mathscr{L}_{X_t}\quad \text{ in } \mathscr{P}_{p_0}(\mathbb{H}),
	\end{equation}
	with $p_0<p.$ It follows from  (\ref{Aall}) and taking expectation that
	\begin{equation}\label{Aall3}
		\mathbb{E}\|X^{(n_k)}_t\|_{\mathbb{V}}^\alpha\lesssim 1+\mathbb{E}\|X^{(n_k)}_t\|_{\mathbb{H}}^2+\mathbb{E}\big((1+\|X^{(n_k)}_t\|_{\mathbb{H}}^{\beta(\alpha-1)})\|X_t\|_{\mathbb{V}}^{\alpha}\big).
	\end{equation}
	Once we can prove there exists a subsequence still denoted by $\{n_k\}_{k=1}^{\infty}$ such that
	\begin{equation}\label{cont}
		\mathbb{E}(\|X^{(n_k)}_t\|_{\mathbb{H}}^{\beta(\alpha-1)}\|X_t\|_{\mathbb{V}}^{\alpha})\xrightarrow{k\to \infty} \mathbb{E}(\|X_t\|_{\mathbb{H}}^{\beta(\alpha-1)}\|X_t\|_{\mathbb{V}}^{\alpha})\quad dt\text{-a.e.},
	\end{equation}
	this together with (\ref{Aall0}) and (\ref{Aall3}) implies
	\begin{equation}\label{Aall1}
		\big\{ \mathscr{L}_{X^{(n_k)}_t}(\|\cdot\|_{\mathbb{V}}^\alpha)\big\}_{k=1}^{\infty}~~\text{is bounded}.
	\end{equation}
	
	In light of  (\ref{Aall}) and (\ref{Aall1}), there exists a constant $K>0$ such that
	\begin{equation}\label{Aall4}
		L^{(n_k)}_t\leq-(\delta-\epsilon _0)\|X^{(n_k)}_t\|_{\mathbb{V}}^\alpha+\epsilon_0K+F^{(n_k)}_t.
	\end{equation}
	Combining (\ref{Aall4}), (\ref{A12}) with the convergence (\ref{A11}) and (\ref{Aall0}), it follows that
	\begin{equation*}
		\{ \|X^{(n_k)}_t(\omega)\|_{\mathbb{V}}^\alpha\}_{k=1}^{\infty}~~\text{is bounded}.
	\end{equation*}
	Consequently, we can find a subsequence still denoted by $\{n_k\}_{k=1}^{\infty}$ such that $\{ X^{(n_k)}(\omega) \}_{k=1}^{\infty}$ converges weakly in $\mathbb{V}$. This, combined with (\ref{A11}), implies
	\begin{equation}\label{Aall2}
		X^{(n_k)}_t(\omega) \rightharpoonup X_t(\omega) \quad \text{in}~~ \mathbb{V}.
	\end{equation}

	Combining (\ref{A12}), (\ref{Aall0}), (\ref{Aall1}) with (\ref{Aall2})  and using the pseudo-monotonicity of the operator $\mathcal{A}(t,\cdot,\cdot)$ in the sense of Definition \ref{deps},  we can conclude that
	\begin{equation*}
		\limsup_{k \to \infty} L^{(n_k)}_t(\omega) \leq 0,
	\end{equation*}
	which contradicts with  (\ref{A12}). Therefore, (\ref{A1all}) is satisfied.
	
	From now on, we shall prove the convergence (\ref{cont})
	by employing the cut-off argument.
	To be more precise,  let $\chi_R\in C^{\infty}_c(\mathbb{R})$ be a cut-off function with
	$$\chi_R(r)=\begin{cases} 1,~~~~|r|\leq R,&\quad\\
		0,~~~~|r|>2R.&\quad\end{cases}$$
	Set
	\begin{equation*}
		\Psi_R(t,w):=\|w_t\|_{\mathbb{H}}^{\beta(\alpha-1)}\chi_R(\|w_t\|_{\mathbb{H}}),
		~ \Psi(t,w):=\|w_t\|_{\mathbb{H}}^{\beta(\alpha-1)}.
	\end{equation*}
	On the one hand, due to the estimate (\ref{estiX}) and the convergence (\ref{A11}), by the dominated convergence theorem and the continuity of $\chi_R$ we can deduce
	\begin{equation}\label{Aall5}
		\lim_{k \to \infty}\int_0^T\mathbb{E}\big(|\Psi_R(t,X^{(n_k)})-\Psi_R(t,X)|\cdot\|X_t\|_{\mathbb{V}}^{\alpha}\big)dt=0.
	\end{equation}
	On the other hand, by the estimates (\ref{estitilde}) and (\ref{estiX})  there exists a constant $q_0\in(\beta(\alpha-1),p/2)$ such that
	\begin{eqnarray}\label{Aall6}
		\sup_{n\in\mathbb{N}}\int_0^T\mathbb{E}(\|X^{(n)}_t\|_{\mathbb{H}}^{q_0}\|X_t\|_{\mathbb{V}}^{\alpha}) dt
		\leq\!\!\!\!\!\!\!\!&&\mathbb{E}\bigg[\sup_{t\in[0,T]}\|X^{(n)}_t\|_{\mathbb{H}}^{q_0}\cdot
		\int_0^T\|X_t\|_{\mathbb{V}}^{\alpha}dt\bigg]
		\nonumber\\
		\lesssim
		\!\!\!\!\!\!\!\!&&\mathbb{E}\Big[\sup_{t\in[0,T]}\|X^{(n)}_t\|_{\mathbb{H}}^{2q_0}\Big]+\mathbb{E}
		\bigg(\int_0^T\|X_t\|_{\mathbb{V}}^{\alpha}dt\bigg)^2
		\nonumber\\
		<
		\!\!\!\!\!\!\!\!&&\infty.
	\end{eqnarray}
	Then by Chebyshev's inequality we can deduce
	\begin{eqnarray}\label{Aall7}
		\!\!\!\!\!\!\!\!&&\int_0^T\mathbb{E}\big(|\Psi_R(t,X^{(n_k)})-\Psi(t,X^{(n_k)})|\cdot\|X_t\|_{\mathbb{V}}^{\alpha}\big)dt
		\nonumber\\
		\leq\!\!\!\!\!\!\!\!&&\int_0^T\mathbb{E}\big(\|X^{(n_k)}_t\|_{\mathbb{H}}^{\beta(\alpha-1)}\|X_t\|_{\mathbb{V}}^{\alpha}
		\mathbf{1}_{\{\|X^{(n_k)}_t\|_{\mathbb{H}}>R\}}\big)dt
		\nonumber\\
		\leq\!\!\!\!\!\!\!\!&&\bigg\{\int_0^T\mathbb{E}\big(\|X^{(n_k)}_t\|_{\mathbb{H}}^{q_0}\|X_t\|_{\mathbb{V}}^{\alpha}\big)dt\bigg\}\Big/R^{q_0-\beta\alpha+\beta}.
	\end{eqnarray}
	Similarly, we also have
	\begin{equation}\label{Aall8}
		\int_0^T\mathbb{E}\big(|\Psi_R(t,X)-\Psi(t,X)|\cdot\|X_t\|_{\mathbb{V}}^{\alpha}\big)dt
		\leq\bigg\{\int_0^T\mathbb{E}\big(\|X_t\|_{\mathbb{H}}^{q_0}\|X_t\|_{\mathbb{V}}^{\alpha}\big)dt\bigg\}\Big/R^{q_0-\beta\alpha+\beta}.
	\end{equation}
	Combining (\ref{Aall5})-(\ref{Aall8}) and taking $k\to\infty$ first and then $R\to\infty$, we obtain
	\begin{equation*}
		\lim_{k \to \infty}\int_0^T\mathbb{E}\big(|\Psi(t,X^{(n_k)})-\Psi(t,X)|\cdot\|X_t\|_{\mathbb{V}}^{\alpha}\big)dt=0,
	\end{equation*}
	which implies (\ref{cont}) holds.

	\vspace{1mm}
	\noindent\textbf{Claim 2.} We claim that there exists a subsequence still denoted by $\{n\}$ such that
	\begin{equation}\label{A21}
		\lim_{n \to \infty} L^{(n)}_t = 0 \quad dt \times  \mathbb{P}\text{-a.e.}.
	\end{equation}

	Note that $|L^{(n)}_t| = 2L^{(n)}_+ - L^{(n)}_t,$ where $L^{(n)}_{+}:=\max\{L^{(n)}_t, 0\}.$
	It follows from (\ref{A1all}) that
	\begin{equation}\label{es7}
		\lim_{n \to \infty} L^{(n)}_+ = 0 \quad dt \times  \mathbb{P}\text{-a.e.}.
	\end{equation}
	Once we can prove
	\begin{equation}\label{A22}
		\lim_{n \to \infty} \mathbb{E} \int_0^T L^{(n)}_t dt = 0,
	\end{equation}
	then in view of  (\ref{estitilde}), (\ref{estiX}), (\ref{Aall}),  (\ref{Aall6}), and (\ref{es7})  and applying the dominated convergence theorem,  we derive
	\begin{equation*}
		\lim_{n \to \infty} \mathbb{E} \int_0^T |L^{(n)}_t| dt =2\mathbb{E} \int_0^T\lim_{n \to \infty}  L^{(n)}_+ dt-\lim_{n \to \infty} \mathbb{E} \int_0^T L^{(n)}_t dt= 0.
	\end{equation*}
	This indicates that there exists a subsequence for which (\ref{A21}) holds.
	
	Now we aim to prove (\ref{A22}).
	By the estimates (\ref{estitilde}), (\ref{estiX}), and (\ref{Aall6}), then Fatou's Lemma indicates
	\begin{equation*}
		\limsup_{n \to \infty} \mathbb{E} \int_0^T L^{(n)}_t dt \leq \mathbb{E} \int_0^T \limsup_{n \to \infty} L^{(n)}_t dt \leq 0,
	\end{equation*}
	where the last step used (\ref{A1all}).
	On the other hand, the conditions (\ref{Acondition2}) and (\ref{Acondition3}) imply that
	\begin{eqnarray*}
		\liminf_{n \to \infty} \mathbb{E} \int_0^T L^{(n)}_t dt=\!\!\!\!\!\!\!\!&&\liminf_{n \to \infty} \mathbb{E} \int_0^T{_{\mathbb{V}^*}\langle \mathcal{A}(t, X^{(n)}_t, \mathscr{L}_{X^{(n)}_t}), X^{(n)}_t \rangle_{\mathbb{V}}}dt\\
		\!\!\!\!\!\!\!\!&&-\liminf_{n \to \infty} \mathbb{E} \int_0^T{_{\mathbb{V}^*}\langle \mathcal{A}(t, X^{(n)}_t, \mathscr{L}_{X^{(n)}_t}), X_t \rangle_{\mathbb{V}}}dt \geq 0.
	\end{eqnarray*}
	In conclusion, the limit of the sequence
	$$\mathbb{E} \int_0^T L^{(n)}_t dt$$  exists and is equal to 0.

	\vspace{1mm}
	\noindent\textbf{Claim 3.}
	We claim that
	$$\mathcal{A}(\cdot, X_{\cdot}, \mathscr{L}_{X_{\cdot}}) = G_{\cdot}\quad dt \times \mathbb{P}\text{-a.e.}.$$
	
	Due to (\ref{Aall4}) and (\ref{A21}), it follows that
	\begin{equation*}
		\sup_{n \in \mathbb{N}} \|X^{(n)}_t\|_\mathbb{V}^{\alpha} < \infty \quad dt \times  \mathbb{P}\text{-a.e.}.
	\end{equation*}
	Hence, combining (\ref{estitilde}), (\ref{estiX}) with (\ref{A11}), it follows that
	$$X^{(n)}_t \rightharpoonup X_t \text{ in } \mathbb{V}\quad dt \times  \mathbb{P}\text{-a.e.}$$
	and
	$$\mathscr{L}_{X^{(n)}_t} \to \mathscr{L}_{X_t} \text{ in } \mathscr{P}_{p_0}(\mathbb{H})\quad dt\text{-a.e.},$$
	with $p_0<p$.
	
	On the other hand, \textbf{Claim 2} means that
	$$\liminf_{n \to \infty}L^{(n)}_t=0\quad dt \times \mathbb{P}\text{-a.e.}.$$
	Then by  the equivalent characterization of pseudo monotonicity (see Lemma \ref{pseudoequivalent}), we have
	\begin{equation*}
		\mathcal{A}(t, X^{(n)}_{t}, \mathscr{L}_{X^{(n)}_{t}}) \rightharpoonup \mathcal{A}(t, X_{t}, \mathscr{L}_{X_{t}})  \text{ in } \mathbb{V}^*.
	\end{equation*}
	According to the uniqueness of the limit, it is clear that \textbf{Claim 3} follows.
\end{proof}

In the sequel, we proceed to prove Theorem \ref{th1}.

\vspace{1mm}
\noindent\textbf{Proof of Theorem \ref{th1}.}
Recall the equality (\ref{eq5}) and Lemma \ref{A}. It suffices to verify (\ref{Acondition3}). Note that using It\^{o}'s formula and taking expectation to $\|X_t^{(n)}\|_\mathbb{H}^2$ and  $\|X_t\|_\mathbb{H}^2$, respectively, yields that
\begin{eqnarray}\label{A-Xn}
	\mathbb{E} \|X_t^{(n)}\|_\mathbb{H}^2=\!\!\!\!\!\!\!\!&& \mu_0\circ P_n^{-1}(\|\cdot\|_{\mathbb{H}}^2) + 2\mathbb{E} \int_{0}^{T}\,_{\mathbb{V}^*} \langle P_n\mathcal{A}(t, X^{(n)}_t, \mathscr{L}_{X^{(n)}_t}),X^{(n)}_t \rangle_{\mathbb{V}} dt \nonumber\\
	\!\!\!\!\!\!\!\!&&+ \mathbb{E} \int_{0}^{T} \|P_n \mathcal{B}(t, X^{(n)}_t, \mathscr{L}_{X^{(n)}_t}) \|_{L_2(U;\mathbb{H})}^2 dt
\end{eqnarray}
and
\begin{equation}\label{A-X}
	\mathbb{E} \|X_t\|_\mathbb{H}^2= \mu_0(\|\cdot\|_{\mathbb{H}}^2) + 2\mathbb{E} \int_{0}^{T}\,_{\mathbb{V}^*} \langle G_t,X_t \rangle_{\mathbb{V}} dt
	+ \mathbb{E} \int_{0}^{T} \|\mathcal{B}(t, X_t, \mathscr{L}_{X_t}) \|_{L_2(U;\mathbb{H})}^2 dt.
\end{equation}
Since the function $\|\cdot\|_\mathbb{H}$ is lower semi-continuous in $\mathbb{V}^*$, Fatou's lemma implies
\begin{equation*}
	\mathbb{E}\|X_t\|_\mathbb{H}^2 \leq \mathbb{E}\liminf_{n \to \infty} \|X^{(n)}_t\|_\mathbb{H}^2 \leq \liminf_{n \to \infty} \mathbb{E}\|X^{(n)}_t\|_\mathbb{H}^2.
\end{equation*}
In light of (\ref{BPn}), (\ref{A-Xn}), (\ref{A-X}), and the definition of the projection $P_n$, we conclude that (\ref{Acondition3}) holds. Furthermore, the estimate (\ref{esq370}) follows from the estimate (\ref{estiX}). Thus, the proof of Theorem \ref{th1} is finished. \hspace{\fill}$\Box$

\subsection{Proof of Theorems \ref{th2} and \ref{th3}}\label{sec4.5}
In this subsection, we aim to prove Theorem \ref{th2} and Theorem \ref{th3}.

\vspace{1mm}
\noindent\textbf{Proof of Theorem \ref{th2}.} \textbf{Step 1.} We claim that the conditions $(\mathbf{A}_2^*)$ and $(\mathbf{A}_5)$ imply the condition $(\mathbf{A}_2)$. Then
Theorem \ref{th1} ensures that Eq.~(\ref{eqSPDE}) has a weak solution $(X, W)$.

To be more precise,
for any sequence
$\{u_n\}_{n=1}^{\infty}\subset\mathbb{V}$ and $\{\mu_n\}_{n=1}^{\infty}\subset\mathfrak{M}_b$ with
$u_n\rightharpoonup u$ in $\mathbb{V}$ and $\mu_n\to\mu$ in $\mathscr{P}_\beta(\mathbb{H})\cap\mathscr{P}_{2}(\mathbb{H})$,
and
\begin{equation}\label{liminf}
	\liminf_{n\to\infty}\,_{\mathbb{V}^*}
	\langle \mathcal{A}(u_n,\mu_n), u_n - u \rangle_{\mathbb{V}}
	\geq 0,
\end{equation}
then for any $v \in \mathbb{V}$, we need to verify
\begin{equation*}
	\limsup_{n\to\infty}\,_{\mathbb{V}^*}
	\langle \mathcal{A}(u_n,\mu_n), u_n - v \rangle_{\mathbb{V}}
	\leq \,_{\mathbb{V}^*}\langle \mathcal{A}(u,\mu), u - v \rangle_{\mathbb{V}}.
\end{equation*}
We set
\begin{eqnarray*}
	K_0 :=
	\!\!\!\!\!\!\!\!&& \|v\|_{\mathbb{V}}+\|u\|_{\mathbb{V}}  + \sup_n \|u_n\|_{\mathbb{V}}+\mu(\|\cdot\|_{\mathbb{V}}^{\alpha})+\mu(\|\cdot\|_{\mathbb{H}}^{\beta})
	\nonumber\\
	\!\!\!\!\!\!\!\!&&
	+\sup_n\mu_n(\|\cdot\|_{\mathbb{V}}^{\alpha})
	+\sup_n\mu_n(\|\cdot\|_{\mathbb{H}}^{\beta})
\end{eqnarray*}
and
\begin{eqnarray*}
	K_1 := \!\!\!\!\!\!\!\!&&\sup \Big\{\eta(u,\mu) + \rho(u,\mu) : u \in \mathbb{V},\mu\in \mathfrak{M} ,
	\nonumber\\
	\!\!\!\!\!\!\!\!&&\quad\quad\|u\|_{\mathbb{V}}+\mu(\|\cdot\|_{\mathbb{H}}^{\beta})
	+\mu(\|\cdot\|_{\mathbb{V}}^{\alpha}) \leq K_0 \Big\}.
\end{eqnarray*}
Since the embedding $\mathbb{V} \subset \mathbb{H}$ is compact, it follows that $u_n \to u$ in $\mathbb{H}$ and
\begin{equation*}
	\lim_{n\to\infty}\,_{\mathbb{V}^*}
	\langle K_1 u_n,\, u_n - v\rangle_{\mathbb{V}}=\,_{\mathbb{V}^*}
	\langle K_1 u,\, u - v\rangle_{\mathbb{V}}.
\end{equation*}
Hence, it is sufficient to show that
\begin{equation}\label{aimlimsup}
	\limsup_{n\to\infty} \,_{\mathbb{V}^*}\langle \mathcal{A}_0(u_n,\mu_n),\, u_n - v \rangle_{\mathbb{V}}\leq\,_{\mathbb{V}^*}\langle \mathcal{A}_0(u,\mu), u - v \rangle_{\mathbb{V}},
\end{equation}
where $\mathcal{A}_0(u,\mu) := \mathcal{A}(u,\mu) - K_1u$.

By the condition $(\mathbf{A}_5)$, it follows that
\begin{equation*}
	\limsup_{n\to\infty}\,_{\mathbb{V}^*}\langle\mathcal{A}_0(u_n,\mu_n)
	-\mathcal{A}_0(u,\mu),u_n - u\rangle_{\mathbb{V}}\leq\limsup_{n\to\infty}K_1\mathcal{W}_{2,\mathbb{H}}(\mu_n,\mu)^2=0.
\end{equation*}
Due to the weak convergence of $u_n$ in $\mathbb{V}$, we deduce that
\begin{equation*}
	\limsup_{n\to\infty}\,_{\mathbb{V}^*}\langle\mathcal{A}_0(u_n,\mu_n)
	,u_n - u\rangle_{\mathbb{V}}\leq0.
\end{equation*}
This, combined with (\ref{liminf}), implies
\begin{equation}\label{limun-u}
	\lim_{n\to\infty}\,_{\mathbb{V}^*}\langle\mathcal{A}_0(u_n,\mu_n)
	,u_n - u\rangle_{\mathbb{V}}=0.
\end{equation}
Moreover, the condition $(\mathbf{A}_2^*)$ yields that
\begin{equation*}
	\lim_{n\to\infty}\,_{\mathbb{V}^*}\langle\mathcal{A}_0(u_n,\mu_n)
	,u - v\rangle_{\mathbb{V}}=\,_{\mathbb{V}^*}\langle\mathcal{A}_0(u,\mu)
	,u - v\rangle_{\mathbb{V}}.
\end{equation*}
This, combined with (\ref{limun-u}), implies the claim holds.

\noindent\textbf{Step 2.} According to the modified Yamada-Watanabe theorem (cf. \cite[Lemma 5.9]{HLL26}), if the following decoupled equation
\begin{equation}\label{barXeq}
	d\bar{X}_t = \mathcal{A}^\mu(t, \bar{X}_t)dt + \mathcal{B}^\mu(t, \bar{X}_t)dW_t,~\bar{X}_0\sim \mu_0,
\end{equation}
where $\mu_t := \mathscr{L}_{X_t}$, $\mathcal{A}^\mu(t, \cdot) := \mathcal{A}(t, \cdot, \mu_t)$, $\mathcal{B}^\mu(t, \cdot) := \mathcal{B}(t, \cdot, \mu_t)$, satisfies the pathwise uniqueness, then Eq.~(\ref{eqSPDE}) has a strong solution, i.e., Theorem \ref{th2} holds.

To that end, we intend to prove the pathwise uniqueness of Eq.~(\ref{barXeq}). Let $\bar{X}$ and $\bar{Y}$ be two solutions of Eq.~(\ref{barXeq}) with the same initial value $\xi \in L^p(\Omega, \mathscr{F}_0, \mathbb{P}; \mathbb{H})$. We mention that the following moment estimates hold
\begin{eqnarray}
	\!\!\!\!\!\!\!\!&&\mathbb{E}\Big[\sup_{t\in[0,T]}\|\bar{X}_t\|_{\mathbb{H}}^{\gamma p}\Big]+\mathbb{E}\int_0^T(1+\|\bar{X}_t\|_{\mathbb{H}}^{p-2})\|\bar{X}_t\|_{\mathbb{V}}^{\alpha}dt<\infty,\label{barXesti}
	\\
	\!\!\!\!\!\!\!\!&&
	\mathbb{E}\Big[\sup_{t\in[0,T]}\|\bar{Y}_t\|_{\mathbb{H}}^{\gamma p}\Big]+\mathbb{E}\int_0^T(1+\|\bar{Y}_t\|_{\mathbb{H}}^{p-2})\|\bar{Y}_t\|_{\mathbb{V}}^{\alpha}dt<\infty,\label{barYesti}
\end{eqnarray}
for any $\gamma \in (0,1)$. Define
$$
\phi(t) := \rho(\bar{X}_t, \mu_t) + \eta(\bar{Y}_t, \mu_t).
$$
Based on the condition $(\mathbf{A}_5),$ applying It\^{o}'s formula and the product rule yields
\begin{eqnarray*}
	\!\!\!\!\!\!\!\!&& e^{-\int_0^t \phi(s)  ds} \| \bar{X}_t - \bar{Y}_t \|_{\mathbb{H}}^2\\
	\leq\!\!\!\!\!\!\!\!&& \int_0^t e^{-\int_0^s \phi(r)  dr} \Big( 2_{\mathbb{V}^*} \langle \mathcal{A}^\mu (s, \bar{X}_s) - \mathcal{A}^\mu (s, \bar{Y}_s), \bar{X}_s - \bar{Y}_s \rangle_\mathbb{V}\\
	\!\!\!\!\!\!\!\!&&+ \| \mathcal{B}^\mu (s, \bar{X}_s) - \mathcal{B}^\mu (s, \bar{Y}_s) \|_{L_2(U; \mathbb{H})}^2 - \phi(s) \| X_s - Y_s \|_{\mathbb{H}}^2 \Big) ds\\
	\!\!\!\!\!\!\!\!&&+ 2 \int_0^t e^{-\int_0^s \phi(r)  dr} \langle \bar{X}_s - \bar{Y}_s, (\mathcal{B}^\mu (s, \bar{X}_s) - \mathcal{B}^\mu (s, \bar{Y}_s)) dW_s\rangle_{\mathbb{H}}\\
	\leq\!\!\!\!\!\!\!\!&& 2 \int_0^t e^{-\int_0^s \phi(r)  dr} \langle\bar{X}_s - \bar{Y}_s, (\mathcal{B}^\mu (s, \bar{X}_s) - \mathcal{B}^\mu (s, \bar{Y}_s)) dW_s\rangle_{\mathbb{H}}.
\end{eqnarray*}
Taking expectation on both sides of the above
inequality, it follows that
\begin{equation*}
	\mathbb{E} \Big[ e^{-\int_0^t \phi(s)  ds} \| \bar{X}_t - \bar{Y}_t \|_{\mathbb{H}}^2 \Big] \leq 0.
\end{equation*}
On the other hand, by (\ref{barXesti}), (\ref{barYesti}), and the definition of $\phi$ we deduce that
\begin{equation*}
	\int_0^T \phi(t)  dt < \infty \quad \mathbb{P} \text{-a.s.}.
\end{equation*}
Consequently, we have
\begin{equation*}
	\bar{X}_t = \bar{Y}_t \quad \mathbb{P}\text{-a.s.}, \quad t \in [0,T],
\end{equation*}
which implies the pathwise uniqueness of (\ref{barXeq}) by the path continuity on $\mathbb{H}$.

We finish the proof of Theorem \ref{th2}. \hspace{\fill}$\Box$

\vspace{1mm}
\noindent\textbf{Proof of Theorem \ref{th3}.} \textbf{Step 1.}
By the modified Yamada-Watanabe theorem, in order to establish the existence and uniqueness of weak and strong solutions to Eq.~(\ref{eqSPDE}), it remains to prove
the pathwise uniqueness of Eq.~(\ref{eqSPDE}) on $\mathbb{H}$. Let $X, Y$ be two solutions of Eq.~(\ref{eqSPDE}) with initial values $X_0 = Y_0 = \xi \in L^p(\Omega, \mathscr{F}_0, \mathbb{P}; \mathbb{H})$,  and the following moment estimates hold
\begin{eqnarray}
	\!\!\!\!\!\!\!\!&&\mathbb{E}\Big[\sup_{t\in[0,T]}\|X_t\|_{\mathbb{H}}^{\gamma p}\Big]+\mathbb{E}\int_0^T(1+\|X_t\|_{\mathbb{H}}^{p-2})\|X_t\|_{\mathbb{V}}^{\alpha}dt<\infty,\label{unique-Xesti}
	\\
	\!\!\!\!\!\!\!\!&&\mathbb{E}\Big[\sup_{t\in[0,T]}\|Y_t\|_{\mathbb{H}}^{\gamma p}\Big]+\mathbb{E}\int_0^T(1+\|Y_t\|_{\mathbb{H}}^{p-2})\|Y_t\|_{\mathbb{V}}^{\alpha}dt<\infty,\label{unique-Yesti}
\end{eqnarray}
for any $\gamma \in (0,1)$.

Applying It\^{o}'s formula to $\|X_t - Y_t\|^2_{\mathbb{H}}$ yields that for any $t \in [0, T],$
\begin{eqnarray*}
	\!\!\!\!\!\!\!\!&&\|X_t - Y_t\|^2_{\mathbb{H}}
	\\
	=\!\!\!\!\!\!\!\!&& \int_0^t \Big( 2_{\mathbb{V}^*} \langle \mathcal{A}(s, X_s, \mathscr{L}_{X_s}) - \mathcal{A}(s, Y_s, \mathscr{L}_{Y_s}), X_s - Y_s \rangle_\mathbb{V}\\
	\!\!\!\!\!\!\!\!&&+ \|\mathcal{B}(s, X_s, \mathscr{L}_{X_s}) - \mathcal{B}(s, Y_s, \mathscr{L}_{Y_s})\|^2_{L_2(U; \mathbb{H})} \Big) ds\\
	\!\!\!\!\!\!\!\!&&+2\int_0^t \langle X_s - Y_s, (\mathcal{B}(s, X_s, \mathscr{L}_{X_s}) - \mathcal{B}(s, Y_s, \mathscr{L}_{Y_s}) ) dW_s \rangle_{\mathbb{H}}\\
	\leq\!\!\!\!\!\!\!\!&&  \int_0^t ( \rho(0, \mathscr{L}_{X_s}) + \eta(0, \mathscr{L}_{Y_s}) ) \|X_s - Y_s\|^2_{\mathbb{H}} ds\\
	\!\!\!\!\!\!\!\!&&+  \int_0^t ( \rho(X_s, \mathscr{L}_{X_s}) + \eta(Y_s, \mathscr{L}_{Y_s})) \mathbb{E} \|X_s - Y_s\|^2_{\mathbb{H}} ds\\
	\!\!\!\!\!\!\!\!&&+2\int_0^t \langle X_s - Y_s, (\mathcal{B}(s, X_s, \mathscr{L}_{X_s}) - \mathcal{B}(s, Y_s, \mathscr{L}_{Y_s}) ) dW_s \rangle_{\mathbb{H}},
\end{eqnarray*}
where we used $(\mathbf{A}^*_5)$ in the last step. Taking expectation on both sides of the above
inequality, we have
\begin{equation}\label{unique-1}
	\mathbb{E}\|X_{t}-Y_{t}\|_{\mathbb{H}}^{2} \lesssim \int_{0}^{t}(1+\mathbb{E} \rho(X_{s}, \mathscr{L}_{X_{s}})+\mathbb{E} \eta(Y_{s}, \mathscr{L}_{Y_{s}})) \mathbb{E}\|X_{s}-Y_{s}\|_{\mathbb{H}}^{2} d s.
\end{equation}
Note that by (\ref{unique-Xesti}) and (\ref{unique-Yesti}), it follows that
\begin{equation}\label{rho-eta}
	\mathbb{E} \int_{0}^{T} ( \rho(X_t, \mathscr{L}_{X_t}) + \eta(Y_t, \mathscr{L}_{Y_t})) dt < \infty.
\end{equation}
Combining (\ref{unique-1}) with (\ref{rho-eta}), Gronwall's lemma gives that
\begin{equation*}
	\mathbb{E} \| X_t - Y_t \|_{\mathbb{H}}^{2} \leq 0,
\end{equation*}
which implies
\begin{equation*}
	X_t = Y_t \quad \mathbb{P} \text{-a.s., } \quad t \in [0, T].
\end{equation*}
Then the pathwise uniqueness of Eq.~(\ref{eqSPDE}) holds due to the path continuity on $\mathbb{H}$.

\vspace{1mm}
\noindent\textbf{Step 2.} In this step, we prove (\ref{initialconti}).
Firstly, it is clear that the following moment estimates hold for any $\gamma\in(0,1)$
\begin{equation}\label{xinesti}
	\sup_{n \in \mathbb{N}}\bigg\{\mathbb{E}\Big[\sup_{t\in[0,T]}\|X_{t}(\xi_n)\|_{\mathbb{H}}^{\gamma p}\Big]+\mathbb{E}\int_0^T\|X_{t}(\xi_n)\|_{\mathbb{H}}^{p-2}\|X_{t}(\xi_n)\|_{\mathbb{V}}^{\alpha}dt\bigg\}<\infty,
\end{equation}
\begin{equation}\label{xiesti}
	\mathbb{E}\Big[\sup_{t\in[0,T]}\|X_{t}(\xi)\|_{\mathbb{H}}^{\gamma p}\Big]+\mathbb{E}\int_0^T\|X_{t}(\xi)\|_{\mathbb{H}}^{p-2}\|X_{t}(\xi)\|_{\mathbb{V}}^{\alpha}dt<\infty.
\end{equation}
We define the stopping time
$$
\tau^n_M :=\inf\big\{t \in[0,T]: \|X_t(\xi_n)\|_{\mathbb{H}} >M\big\} \wedge
\inf\big\{t \in[0,T]: \|X_t(\xi)\|_{\mathbb{H}} >M\big\}\wedge T.
$$
Applying It\^{o}'s formula and using $(\mathbf{A}^*_5)$ yields that
\begin{eqnarray}\label{itoxi}
	\!\!\!\!\!\!\!\!&&\|X_{t}(\xi_n)-X_{t}(\xi)\|^2_\mathbb{H}-\|\xi_n-\xi\|^2_\mathbb{H}\nonumber\\
	=\!\!\!\!\!\!\!\!&& \int_0^t \Big( 2_{\mathbb{V}^*} \langle \mathcal{A}(s, X_{s}(\xi_n), \mathscr{L}_{X_{s}(\xi_n)}) - \mathcal{A}(s, X_{s}(\xi), \mathscr{L}_{X_{s}(\xi)}), X_{s}(\xi_n)-X_{s}(\xi) \rangle_\mathbb{V}\nonumber\\
	\!\!\!\!\!\!\!\!&&+ \|\mathcal{B}(s, X_{s}(\xi_n), \mathscr{L}_{X_{s}(\xi_n)}) - \mathcal{B}(s, X_{s}(\xi), \mathscr{L}_{X_{s}(\xi)})\|^2_{L_2(U; \mathbb{H})} \Big) ds+2\mathcal{M}_t\nonumber\\
	\leq\!\!\!\!\!\!\!\!&&  \int_0^t ( \rho(0, \mathscr{L}_{X_{s}(\xi_n)}) + \eta(0, \mathscr{L}_{X_{s}(\xi)}) ) \|X_{s}(\xi_n)-X_{s}(\xi)\|^2_{\mathbb{H}} ds+2\mathcal{M}_t\nonumber\\
	\!\!\!\!\!\!\!\!&&+  \int_0^t ( \rho(X_{s}(\xi_n), \mathscr{L}_{X_{s}(\xi_n)}) + \eta(X_{s}(\xi), \mathscr{L}_{X_{s}(\xi)})) \mathcal{W}_{2,\mathbb{H}}(\mathscr{L}_{X_{s}(\xi_n)},\mathscr{L}_{X_{s}(\xi)})^2 ds,
\end{eqnarray}
where
$$\mathcal{M}_t:=\int_0^t \langle X_{s}(\xi_n)-X_{s}(\xi), (\mathcal{B}(s, X_{s}(\xi_n), \mathscr{L}_{X_{s}(\xi_n)}) - \mathcal{B}(s, X_{s}(\xi), \mathscr{L}_{X_{s}(\xi)})) dW_s \rangle_{\mathbb{H}}.$$
Using Burkholder--Davis--Gundy inequality and Young's inequality, we derive
\begin{eqnarray*}
	\!\!\!\!\!\!\!\!&&\mathbb{E} \Big[ \sup_{t \in [0,T\wedge\tau^n_M]} |\mathcal{M}_t| \Big]\nonumber\\
	\leq\!\!\!\!\!\!\!\!&&  \mathbb{E} \bigg( \int_0^{T\wedge\tau^n_M} \|X_{t}(\xi_n)-X_{t}(\xi)\|_{\mathbb{H}}^2
	\nonumber\\
	\!\!\!\!\!\!\!\!&& \cdot\|\mathcal{B}(t, X_{t}(\xi_n), \mathscr{L}_{X_{t}(\xi_n)}) - \mathcal{B}(t, X_{t}(\xi), \mathscr{L}_{X_{t}(\xi)})\|_{L_2(U;\mathbb{H})}^2 dt \bigg)^{\frac{1}{2}}\nonumber\\
	\leq\!\!\!\!\!\!\!\!&& \mathbb{E} \bigg[ \Big( \sup_{t \in [0,T\wedge\tau^n_M]} \|X_{t}(\xi_n)-X_{t}(\xi)\|_{\mathbb{H}}^2 \Big)^{\frac{1}{2}}\nonumber\\
	\!\!\!\!\!\!\!\!&& \cdot \bigg( \int_0^{T\wedge\tau^n_M} \|\mathcal{B}(t, X_{t}(\xi_n), \mathscr{L}_{X_{t}(\xi_n)}) - \mathcal{B}(t, X_{t}(\xi), \mathscr{L}_{X_{t}(\xi)})\|_{L_2(U;\mathbb{H})}^2 dt \bigg)^{\frac{1}{2}} \bigg] \nonumber\\
	\leq\!\!\!\!\!\!\!\!&& \mathbb{E} \int_0^{T\wedge\tau^n_M} \|\mathcal{B}(t, X_{t}(\xi_n), \mathscr{L}_{X_{t}(\xi_n)}) - \mathcal{B}(t, X_{t}(\xi), \mathscr{L}_{X_{t}(\xi)})\|_{L_2(U;\mathbb{H})}^2 dt \nonumber\\
	\!\!\!\!\!\!\!\!&&+ \frac{1}{2} \mathbb{E} \Big[ \sup_{t \in [0,T\wedge\tau^n_M]} \|X_{t}(\xi_n)-X_{t}(\xi)\|_{\mathbb{H}}^2 \Big].
\end{eqnarray*}
Then it follows that
\begin{eqnarray*}
	\!\!\!\!\!\!\!\!&&\mathbb{E}\Big[\sup_{t\in[0,T\wedge\tau^n_M]}\|X_{t}(\xi_n)-X_{t}(\xi)\|^2_\mathbb{H}\Big]-\mathbb{E}\|\xi_n-\xi\|^2_\mathbb{H}\\
	\lesssim\!\!\!\!\!\!\!\!&&  \int_0^{T} ( \rho(0, \mathscr{L}_{X_{t}(\xi_n)}) + \eta(0, \mathscr{L}_{X_{t}(\xi)}) ) \mathbb{E}\|X_{t}(\xi_n)-X_{t}(\xi)\|^2_{\mathbb{H}} dt\\
	\!\!\!\!\!\!\!\!&&+  \int_0^{T} ( \mathbb{E}\rho(X_{t}(\xi_n), \mathscr{L}_{X_{t}(\xi_n)}) + \mathbb{E}\eta(X_{t}(\xi), \mathscr{L}_{X_{t}(\xi)})) \mathbb{E} \|X_{t}(\xi_n)-X_{t}(\xi)\|^2_{\mathbb{H}} dt\\
	\!\!\!\!\!\!\!\!&&+\mathbb{E} \int_0^{T\wedge\tau^n_M} \|\mathcal{B}(t, X_{t}(\xi_n), \mathscr{L}_{X_{t}(\xi_n)}) - \mathcal{B}(t, X_{t}(\xi), \mathscr{L}_{X_{t}(\xi)})\|_{L_2(U;\mathbb{H})}^2 dt.
\end{eqnarray*}
Taking $M\to\infty$ and applying the monotone convergence theorem and Gronwall's lemma, due to (\ref{esq22}), (\ref{xinesti}), and (\ref{xiesti}), we derive
\begin{eqnarray}\label{supxin-xi}
	\!\!\!\!\!\!\!\!&&\mathbb{E}\Big[\sup_{t\in[0,T]}\|X_{t}(\xi_n)-X_{t}(\xi)\|^2_\mathbb{H}\Big]-\mathbb{E}\|\xi_n-\xi\|^2_\mathbb{H}\nonumber\\
	\lesssim\!\!\!\!\!\!\!\!&&
	\mathbb{E} \int_0^{T} \|\mathcal{B}(t, X_{t}(\xi_n), \mathscr{L}_{X_{t}(\xi_n)}) - \mathcal{B}(t, X_{t}(\xi), \mathscr{L}_{X_{t}(\xi)})\|_{L_2(U;\mathbb{H})}^2 dt.
\end{eqnarray}

On the other hand,
by (\ref{itoxi}) and applying Gronwall's lemma, we also obtain
\begin{equation*}
	\mathbb{E}\|X_{t}(\xi_n)-X_{t}(\xi)\|^2_\mathbb{H}
	\lesssim\mathbb{E}\|\xi_n-\xi\|^2_\mathbb{H},
\end{equation*}
which implies that for any $t\in[0,T]$,
\begin{equation*}
	\mathcal{W}_{2,\mathbb{H}}(\mathscr{L}_{X_{t}(\xi_n)},\mathscr{L}_{X_{t}(\xi)})\xrightarrow{n\to\infty} 0,
\end{equation*}
and there exists a subsequence still denoted by $\{n\}$ such that
\begin{equation*}
	X_{t}(\xi_n)\xrightarrow{n\to\infty}X_{t}(\xi)\quad \mathbb{P}\text{-a.s.}~\text{in}~\mathbb{H}.
\end{equation*}
By $(\mathbf{A}_1)$, (\ref{xinesti}), and (\ref{xiesti}), the Vitali's convergence theorem implies
\begin{equation*}
	\lim_{n \to \infty}\mathbb{E} \int_0^{T}\|\mathcal{B}(t, X_{t}(\xi_n), \mathscr{L}_{X_{t}(\xi_n)}) - \mathcal{B}(t, X_{t}(\xi), \mathscr{L}_{X_{t}(\xi)})\|_{L_2(U;\mathbb{H})}^2 dt=0.
\end{equation*}
This, combined with (\ref{supxin-xi}), gives
\begin{equation*}
	\sup_{t\in[0,T]}\|X_{t}(\xi_n)-X_{t}(\xi)\|_\mathbb{H}\xrightarrow{n\to\infty} 0\quad\text{in probability}.
\end{equation*}
Then, due to (\ref{xinesti}) and (\ref{xiesti}), the Vitali's convergence theorem yields that
\begin{equation*}
	\lim_{n\to \infty}\mathbb{E}\Big[\sup_{t\in[0,T]}\|X_{t}(\xi_n)-X_{t}(\xi)\|^{p'}_\mathbb{H}\Big]
	=0,\quad p'<p.
\end{equation*}
Since the limit is unique, the convergence holds for the whole sequence. Thus, we obtain the desired result. \hspace{\fill}$\Box$
\section{Proof of quantitative propagation of chaos}\label{ProofPoC}
In this section, we devote to proving the propagation of chaos results.  More specifically, in Subsection \ref{secpoc} we present some uniform in $N$ moment bounds of the nIPS (\ref{nIPS}) and IPS (\ref{IPS}), respectively.
In Subsection \ref{Proofrate1}, we verify  Theorem \ref{rate1}. In Subsection \ref{Proofrate2}, we aim to prove Theorem \ref{rate2}.

\subsection{Uniform in $N$ bounds}\label{secpoc}
We first provide the uniform in $N$ moment bounds for IPS (\ref{IPS}).
\begin{lemma}\label{XNesti}
	For any $i \in \{1,2,\dots,N\}$ and $\gamma \in (0,1)$, we have
	\begin{equation}\label{XNesti1}
		\sup_{t\in [0,T]}\mathbb{E} \|X_t^{i,N}\|_{\mathbb{H}}^p+\mathbb{E}\int_{0}^{T}\|X_t^{i,N}\|_{\mathbb{V}}^{\alpha }\|X_t^{i,N}\|_{\mathbb{H}}^{p-2}dt
		\lesssim_{p,T} 1+\mathbb{E}
		\|\xi^i\|_{\mathbb{H}}^p
	\end{equation}
	and
	\begin{equation}\label{XNesti2}
		\mathbb{E}\Big[\sup_{t\in [0,T]} \|X_t^{i,N}\|_{\mathbb{H}}^{\gamma p}\Big]
		\lesssim_{p,T} (1+\mathbb{E}
		\|\xi^i\|_{\mathbb{H}}^{p})^\gamma.
	\end{equation}
\end{lemma}
\begin{proof}
	Applying It\^{o}'s formula to  $\|X_t^{i,N}\|_{\mathbb{H}}^p$ yields that
	\begin{eqnarray}\label{XNeito}
		\!\!\!\!\!\!\!\!&&\|X_t^{i,N}\|_{\mathbb{H}}^p-\|\xi^i \|_{\mathbb{H}}^p\nonumber\\
		=\!\!\!\!\!\!\!\!&&\frac{p}{2} \int_{0}^{t}\|X_s^{i,N}\|_{\mathbb{H}}^{p-2}\big(2\,_{\mathbb{V}^*}\langle \mathcal{A}(s, X_s^{i,N}, \mu^N_s), X_s^{i,N}\rangle_{\mathbb{V}}+\| \mathcal{B}(s, X_s^{i,N}, \mu^N_s) \|_{L_{2}(U;\mathbb{H})}^{2}\big) d s\nonumber\\
		\!\!\!\!\!\!\!\!&&+\frac{p(p-2)}{2}\int_{0}^{t}\|X_s^{i,N}\|_{\mathbb{H}}^{p-4}\| \mathcal{B}(s, X_s^{i,N}, \mu^N_s) ^*X_s^{i,N}\|_{U}^{2} d s\nonumber\\
		\!\!\!\!\!\!\!\!&&+p\int_{0}^{t}\|X_s^{i,N}\|_{\mathbb{H}}^{p-2}\langle X_s^{i,N},  \mathcal{B}(s, X_s^{i,N}, \mu^N_s)  d W^i_{s}\rangle_{\mathbb{H}}\nonumber\\
		\leq\!\!\!\!\!\!\!\!&&\frac{p}{2} \int_{0}^{t}\|X_s^{i,N}\|_{\mathbb{H}}^{p-2}\big[2\,_{\mathbb{V}^*}\langle  \mathcal{A}(s, X_s^{i,N}, \mu^N_s), X_s^{i,N}\rangle_{\mathbb{V}}+(p-1)\|\mathcal{B}(s, X_s^{i,N}, \mu^N_s)\|_{L_{2}(U;\mathbb{H})}^{2}\big]ds\nonumber\\
		\!\!\!\!\!\!\!\!&&+p\int_{0}^{t}\|X_s^{i,N}\|_{\mathbb{H}}^{p-2}\langle X_s^{i,N}, \mathcal{B}(s, X_s^{i,N}, \mu^N_s)  d W^i_{s}\rangle_{\mathbb{H}}.
	\end{eqnarray}
	Using the condition $(\mathbf{A}_3)$ and taking expectation on both sides of (\ref{XNeito}), it follows that
	\begin{eqnarray*}
		\!\!\!\!\!\!\!\!&&\mathbb{E}\Big[\frac{1}{N}\sum_{i=1}^{N} \|X_t^{i,N}\|_{\mathbb{H}}^p\Big]-\mathbb{E}\Big[\frac{1}{N}\sum_{i=1}^{N}\|\xi^i \|_{\mathbb{H}}^p\Big]\nonumber\\
		\leq\!\!\!\!\!\!\!\!&&-\frac{p\delta }{2} \mathbb{E}\int_{0}^{t}\frac{1}{N}\sum_{i=1}^{N} \big(\|X_s^{i,N}\|_{\mathbb{V}}^{\alpha }\|X_s^{i,N}\|_{\mathbb{H}}^{p-2}\big)ds\nonumber\\
		\!\!\!\!\!\!\!\!&&+C\mathbb{E}\int_{0}^{t}\frac{1}{N}\sum_{i=1}^{N} (1+\|X_s^{i,N}\|_{\mathbb{H}}^{2}+\frac{1}{N}\sum_{j=1}^{N}\|X_s^{j,N}\|_{\mathbb{H}}^{2})\|X_s^{i,N}\|_{\mathbb{H}}^{p-2}ds\nonumber\\
		\leq\!\!\!\!\!\!\!\!&&-\frac{p\delta }{2} \mathbb{E}\int_{0}^{t}\frac{1}{N}\sum_{i=1}^{N} \big(\|X_s^{i,N}\|_{\mathbb{V}}^{\alpha }\|X_s^{i,N}\|_{\mathbb{H}}^{p-2}\big)ds+C_{p,T}\nonumber\\
		\!\!\!\!\!\!\!\!&&+C_p\mathbb{E}\int_{0}^{t}\frac{1}{N}\sum_{i=1}^{N} \|X_s^{i,N}\|_{\mathbb{H}}^{p}ds+C_p\mathbb{E}\int_{0}^{t}\frac{1}{N^2}\sum_{i,j=1}^{N}(\big\|X_s^{j,N}\|_{\mathbb{H}}^{2}\|X_s^{i,N}\|_{\mathbb{H}}^{p-2}\big)ds\nonumber\\
		\leq\!\!\!\!\!\!\!\!&&-\frac{p\delta }{2} \mathbb{E}\int_{0}^{t}\frac{1}{N}\sum_{i=1}^{N} \big(\|X_s^{i,N}\|_{\mathbb{V}}^{\alpha }\|X_s^{i,N}\|_{\mathbb{H}}^{p-2}\big)ds+C_{p,T}+C_p\mathbb{E}\int_{0}^{t}\frac{1}{N}\sum_{i=1}^{N} \|X_s^{i,N}\|_{\mathbb{H}}^{p}ds.
	\end{eqnarray*}
	Then we can derive for any $t \in [0,T]$,
	\begin{eqnarray*}
		\!\!\!\!\!\!\!\!&&\quad \mathbb{E}\Big[\frac{1}{N}\sum_{i=1}^{N} \|X_t^{i,N}\|_{\mathbb{H}}^p\Big]+\mathbb{E}\int_{0}^{t}\frac{1}{N}\sum_{i=1}^{N} \big(\|X_s^{i,N}\|_{\mathbb{V}}^{\alpha }\|X_s^{i,N}\|_{\mathbb{H}}^{p-2}\big)ds\nonumber\\
		\!\!\!\!\!\!\!\!&&\lesssim_{p,T}\Big(1+\mathbb{E}\Big[\frac{1}{N}\sum_{i=1}^{N}\|\xi^i\|_{\mathbb{H}}^p\Big]\Big)+\mathbb{E}\int_{0}^{t}\frac{1}{N}\sum_{i=1}^{N} \|X_s^{i,N}\|_{\mathbb{H}}^{p}ds.
	\end{eqnarray*}
	By applying  Gronwall's lemma, we have
	\begin{equation}\label{XNe1}
		\sup_{t\in [0,T]}\mathbb{E}\Big[\frac{1}{N}\sum_{i=1}^{N} \|X_t^{i,N}\|_{\mathbb{H}}^p\Big]+\mathbb{E}\int_{0}^{T}\frac{1}{N}\sum_{i=1}^{N} \big(\|X_s^{i,N}\|_{\mathbb{V}}^{\alpha }\|X_s^{i,N}\|_{\mathbb{H}}^{p-2}\big)ds
		\lesssim_{p,T}  1+\mathbb{E}\|\xi^i\|_{\mathbb{H}}^p.
	\end{equation}
	
	Meanwhile, in view of  (\ref{XNeito}), one can also get
	\begin{eqnarray}\label{XNe2}
		\!\!\!\!\!\!\!\!&&\mathbb{E}\|X_t^{i,N}\|_{\mathbb{H}}^p-\mathbb{E}\|\xi^i \|_{\mathbb{H}}^p\nonumber\\
		\leq\!\!\!\!\!\!\!\!&&-\frac{p\delta }{2} \mathbb{E}\int_{0}^{t}\|X_s^{i,N}\|_{\mathbb{V}}^{\alpha }\|X_s^{i,N}\|_{\mathbb{H}}^{p-2}ds\nonumber\\
		\!\!\!\!\!\!\!\!&&+C\mathbb{E}\int_{0}^{t}\big(1+\|X_s^{i,N}\|_{\mathbb{H}}^{2}+\frac{1}{N}\sum_{j=1}^{N}\|X_s^{j,N}\|_{\mathbb{H}}^{2}\big)\|X_s^{i,N}\|_{\mathbb{H}}^{p-2}ds\nonumber\\
		\leq\!\!\!\!\!\!\!\!&&-\frac{p\delta }{2} \mathbb{E}\int_{0}^{t}\|X_s^{i,N}\|_{\mathbb{V}}^{\alpha }\|X_s^{i,N}\|_{\mathbb{H}}^{p-2}ds+C_{p,T}\nonumber\\
		\!\!\!\!\!\!\!\!&&+C_p\mathbb{E}\int_{0}^{t}\|X_s^{i,N}\|_{\mathbb{H}}^{p}ds+C_p\mathbb{E}\int_{0}^{t}\frac{1}{N}\sum_{j=1}^{N}(\big\|X_s^{j,N}\|_{\mathbb{H}}^{2}\|X_s^{i,N}\|_{\mathbb{H}}^{p-2}\big)ds\nonumber\\
		\leq\!\!\!\!\!\!\!\!&&-\frac{p\delta }{2} \mathbb{E}\int_{0}^{t}\|X_s^{i,N}\|_{\mathbb{V}}^{\alpha }\|X_s^{i,N}\|_{\mathbb{H}}^{p-2}ds+C_{p,T}\nonumber\\
		\!\!\!\!\!\!\!\!&&+C_p\mathbb{E}\int_{0}^{t}\frac{1}{N}\sum_{j=1}^{N} \|X_s^{j,N}\|_{\mathbb{H}}^{p}ds+C_p\mathbb{E}\int_{0}^{t}\|X_s^{i,N}\|_{\mathbb{H}}^{p}ds.
	\end{eqnarray}
	Combining (\ref{XNe1}) with (\ref{XNe2}), it follows that
	\begin{equation*}
		\mathbb{E}\|X_t^{i,N}\|_{\mathbb{H}}^p+\mathbb{E}\int_{0}^{t}\|X_s^{i,N}\|_{\mathbb{V}}^{\alpha }\|X_s^{i,N}\|_{\mathbb{H}}^{p-2}ds
		\lesssim_{p,T} \big(1+\mathbb{E}\|\xi^i\|_{\mathbb{H}}^p\big)+\mathbb{E}\int_{0}^{t}\|X_s^{i,N}\|_{\mathbb{H}}^{p}ds.
	\end{equation*}
	Hence,  Gronwall's lemma yields that (\ref{XNesti1}) holds.
	
	In order to prove (\ref{XNesti2}), let $\tau_0$ be a bounded stopping time with $\tau_0\leq T.$ By a similar argument as in (\ref{XNe2}), using (\ref{XNeito}), the condition $(\mathbf{A}_3)$ and Young's inequality yields that for any $t \in [0,T]$,
	\begin{eqnarray*}
		\!\!\!\!\!\!\!\!&&\quad \mathbb{E}\|X_{t\wedge \tau_0}^{i,N}\|_{\mathbb{H}}^p+\mathbb{E}\int_{0}^{t\wedge \tau_0}\|X_s^{i,N}\|_{\mathbb{V}}^{\alpha }\|X_s^{i,N}\|_{\mathbb{H}}^{p-2}ds\nonumber\\
		\!\!\!\!\!\!\!\!&&\lesssim_{p,T}\big(1+\mathbb{E}\|\xi^i\|_{\mathbb{H}}^p\big)+\mathbb{E}\int_{0}^{T}\frac{1}{N}\sum_{j=1}^{N} \|X_s^{j,N}\|_{\mathbb{H}}^{p}ds+\mathbb{E}\int_{0}^{t\wedge \tau_0}\|X_s^{i,N}\|_{\mathbb{H}}^{p}ds\nonumber\\
		\!\!\!\!\!\!\!\!&&\lesssim_{p,T}\big(1+\mathbb{E}\|\xi^i\|_{\mathbb{H}}^p\big)+\mathbb{E}\int_{0}^{t}\|X_{s\wedge \tau_0}^{i,N}\|_{\mathbb{H}}^{p}ds.
	\end{eqnarray*}
	By applying Gronwall's lemma, we derive for any $t \in [0,T]$,
	\begin{equation*}
		\mathbb{E}\|X_{t\wedge \tau_0}^{i,N}\|_{\mathbb{H}}^p
		\lesssim_{p,T} 1+\mathbb{E}\|\xi^i\|_{\mathbb{H}}^p.
	\end{equation*}
	Using Lenglart's inequality (cf.~\cite{GK} or \cite{NS}), 
	one gets  for any $\gamma \in (0,1)$,
	\begin{equation*}
		\mathbb{E}\Big[\sup_{t\in [0,T]}\|X_{t}^{i,N}\|_{\mathbb{H}}^{\gamma p}\Big]
		\lesssim_{p,T} (1+\mathbb{E}\|\xi^i\|_{\mathbb{H}}^{ p})^\gamma.
	\end{equation*}
	We complete the proof.
\end{proof}

\begin{lemma}\label{IPSesti2}
	For each $l \in (1,\frac{p}{\beta}] \cap (1,\frac{p}{2}) $ and $i \in \{1,2,\dots,N\}$, we have
	\begin{equation*}
		\mathbb{E}\bigg(\int_{0}^{T}\|X_t^{i,N}\|_{\mathbb{V}}^{\alpha}dt\bigg)^l\lesssim_T 1+\big(\mathbb{E}
		\|\xi^i\|_{\mathbb{H}}^{p}\big)^\frac{2l}{p}+\mathbb{E}\|\xi^i\|_{\mathbb{H}}^{\beta l}.
	\end{equation*}
\end{lemma}
\begin{proof}
	Applying It\^{o}'s formula to  $\|X_t^{i,N}\|_{\mathbb{H}}^2$ and using the condition $(\mathbf{A}_3)$ yields that
	\begin{eqnarray*}
		\!\!\!\!\!\!\!\!&&\|X_t^{i,N}\|_{\mathbb{H}}^{2}+\delta\int_{0}^{t}\|X_s^{i,N}\|_{\mathbb{V}}^{\alpha}ds\nonumber\\
		\leq \!\!\!\!\!\!\!\!&&\|\xi^i \|_{\mathbb{H}}^{2}
		+C\int_{0}^{t}\big(1+\|X_s^{i,N}\|_{\mathbb{H}}^{2}+\frac{1}{N}\sum_{i=1}^{N} \|X_t^{i,N}\|_{\mathbb{H}}^2\big)ds\nonumber\\
		\!\!\!\!\!\!\!\!&&+2\int_{0}^{t}\langle X^{i,N}_s, \mathcal{B}(s, X^{i,N} _s, \mu_s^N)  d W^i_{s}\rangle_{\mathbb{H}}.
	\end{eqnarray*}
	Then by the condition $(\mathbf{A}_4)$, B-D-G's inequality, and Young's inequality, we have
	\begin{eqnarray*}
		\!\!\!\!\!\!\!\!&&\quad\mathbb{E}\bigg(\int_{0}^{T}\|X_t^{i,N}\|_{\mathbb{V}}^{\alpha}dt\bigg)^l
		\\
		\!\!\!\!\!\!\!\!&&\lesssim_T\mathbb{E}\|\xi^i \|_{\mathbb{H}}^{2l}+\mathbb{E}\int_{0}^{T}\Big(1+\|X_t^{i,N}\|_{\mathbb{H}}^{2}+\frac{1}{N}\sum_{i=1}^{N} \|X_t^{i,N}\|_{\mathbb{H}}^2\Big)^l dt\\
		\!\!\!\!\!\!\!\!&&~~~~+\mathbb{E}\bigg\{\sup_{t\in [0,T]}\Big|\int_{0}^{t}\langle X^i_s, \mathcal{B}(s, X^{i,N}_s, \mu^N_{s})  d W^i_{s}\rangle_{\mathbb{H}}\Big|\bigg\}^l\\
		\!\!\!\!\!\!\!\!&&\lesssim_T 1+\mathbb{E}\|\xi^i \|_{\mathbb{H}}^{2l}+\int_{0}^{T}\mathbb{E}\|X_t^{i,N}\|_{\mathbb{H}}^{2l}dt+\mathbb{E}\int_{0}^{T}\frac{1}{N}\sum_{i=1}^{N} \|X_t^{i,N}\|_{\mathbb{H}}^{2l} dt\\
		\!\!\!\!\!\!\!\!&&~~~~+\mathbb{E}\bigg\{\sup_{t \in[0,T]}\|X_t^{i,N}\|_{\mathbb{H}}^{2}\cdot \int_{0}^{T}\Big(1+\|X_t^{i,N}\|_{\mathbb{H}}^{\beta}+\frac{1}{N}\sum_{i=1}^{N} \|X_t^{i,N}\|_{\mathbb{H}}^{\beta}\Big)dt\bigg\}^\frac{l}{2}\\
		\!\!\!\!\!\!\!\!&&\lesssim_T 1+\mathbb{E}\|\xi^i \|_{\mathbb{H}}^{2l}+\mathbb{E}\Big[\sup_{t \in[0,T]}\|X_t^{i,N}\|_{\mathbb{H}}^{2l}\Big]\\
		\!\!\!\!\!\!\!\!&&~~~~+\sup_{t \in[0,T]}\mathbb{E}\Big(\|X_t^{i,N}\|_{\mathbb{H}}^{\beta l}+\frac{1}{N}\sum_{i=1}^{N} \|X_t^{i,N}\|_{\mathbb{H}}^{2l}+\frac{1}{N}\sum_{i=1}^{N} \|X_t^{i,N}\|_{\mathbb{H}}^{\beta l}\Big).
	\end{eqnarray*}
	Then, by (\ref{XNesti1}) and (\ref{XNe1}) we obtain
	\begin{equation*}
		\mathbb{E}\Big(\int_{0}^{T}\|X_t^{i,N}\|_{\mathbb{V}}^{\alpha}dt\Big)^l\lesssim_T 1+\big(\mathbb{E}
		\|\xi^i\|_{\mathbb{H}}^{p}\big)^\frac{2l}{p}+\mathbb{E}\|\xi^i\|_{\mathbb{H}}^{\beta l}.
	\end{equation*}
	The proof is completed.
\end{proof}

We proceed to derive the following improved bounds for nIPS (\ref{nIPS}), which are crucial for obtaining the convergence rate of  PoC.
\begin{lemma}\label{nIPSestibeta2}
	For each $l \in (1,\frac{p}{2\beta}] \cap (1,\frac{p}{\beta+2}) $ and $i \in \{1,2,\dots,N\}$,
	\begin{equation*}
		\mathbb{E}\bigg(\int_{0}^{T}\|X_t^i\|_{\mathbb{H}}^{\beta}\|X_t^i\|_{\mathbb{V}}^{\alpha}dt\bigg)^l\lesssim_T 1+\big(\mathbb{E}
		\|\xi^i\|_{\mathbb{H}}^{p}\big)^\frac{(\beta+2)l}{p}+\mathbb{E}\|\xi^i\|_{\mathbb{H}}^{2\beta l}.
	\end{equation*}
\end{lemma}
\begin{proof}
	Applying It\^{o}'s formula to  $\|X_t^{i}\|_{\mathbb{H}}^{\beta +2}$, the condition $(\mathbf{A}_3)$ implies that
	\begin{eqnarray*}
		\!\!\!\!\!\!\!\!&&\|X_t^i\|_{\mathbb{H}}^{\beta +2}-\|\xi^i \|_{\mathbb{H}}^{\beta +2}\nonumber\\
		=\!\!\!\!\!\!\!\!&&\frac{{\beta +2}}{2} \int_{0}^{t}\|X_s^i\|_{\mathbb{H}}^{\beta}\big(2\,_{\mathbb{V}^*}\langle  \mathcal{A}(s, X^i_s, \mathscr{L}_{X^i_s}), X^i_s\rangle_{\mathbb{V}}+\| \mathcal{B}(s, X^i_s, \mathscr{L}_{X^i_s})\|_{L_{2}(U;\mathbb{H})}^{2}\big) d s\nonumber\\
		\!\!\!\!\!\!\!\!&&+\frac{\beta (\beta +2)}{2}\int_{0}^{t}\|X_s^i\|_{\mathbb{H}}^{\beta -2}\|\mathcal{B}(s, X^i_s, \mathscr{L}_{X^i_s}) ^*X^i_s\|_{U}^{2} d s\nonumber\\
		\!\!\!\!\!\!\!\!&&+(\beta +2)\int_{0}^{t}\|X_s^i\|_{\mathbb{H}}^{\beta}\langle X^i_s, \mathcal{B}(s, X^i_s, \mathscr{L}_{X^i_s})  d W^i_{s}\rangle_{\mathbb{H}}\nonumber\\
		\leq\!\!\!\!\!\!\!\!&&\frac{\beta +2}{2} \int_{0}^{t}\|X_s^i\|_{\mathbb{H}}^{\beta}\big[2\,_{\mathbb{V}^*}\langle  \mathcal{A}(s, X^i_s, \mathscr{L}_{X^i_s}), X^i_s\rangle_{\mathbb{V}}+(\beta +1)\| \mathcal{B}(s, X^i_s, \mathscr{L}_{X^i_s})\|_{L_{2}(U;\mathbb{H})}^{2}\big]ds\nonumber\\
		\!\!\!\!\!\!\!\!&&+(\beta +2)\int_{0}^{t}\|X_s^i\|_{\mathbb{H}}^{\beta}\langle X^i_s, \mathcal{B}(s, X^i_s, \mathscr{L}_{X^i_s})  d W^i_{s}\rangle_{\mathbb{H}}\nonumber\\
		\leq \!\!\!\!\!\!\!\!&&-\frac{(\beta +2)\delta}{2}\int_{0}^{t}\|X_s^i\|_{\mathbb{H}}^{\beta}\|X_s^i\|_{\mathbb{V}}^{\alpha}ds
		+C\int_{0}^{t}\big(1+\|X_s^i\|_{\mathbb{H}}^{2}+\mathbb{E}\|X_s^i\|_{\mathbb{H}}^{2}\big)\|X_s^i\|_{\mathbb{H}}^{\beta}ds\nonumber\\
		\!\!\!\!\!\!\!\!&&+(\beta +2)\int_{0}^{t}\|X_s^i\|_{\mathbb{H}}^{\beta}\langle X^i_s, \mathcal{B}(s, X^i_s, \mathscr{L}_{X^i_s})  d W^i_{s}\rangle_{\mathbb{H}}.
	\end{eqnarray*}
	Using the condition $(\mathbf{A}_4)$ and applying B-D-G's inequality and Young's inequality, it follows that
	\begin{eqnarray*}
		\!\!\!\!\!\!\!\!&&\quad\mathbb{E}\bigg(\int_{0}^{T}\|X_t^i\|_{\mathbb{H}}^{\beta}\|X_t^i\|_{\mathbb{V}}^{\alpha}dt\bigg)^l\\
		\!\!\!\!\!\!\!\!&&\lesssim_T\mathbb{E}\|\xi^i \|_{\mathbb{H}}^{(\beta +2)l}+\mathbb{E}\int_{0}^{T}\big(1+\|X_t^i\|_{\mathbb{H}}^{\beta+2}+\mathbb{E}\|X_t^i\|_{\mathbb{H}}^{\beta+2}\big)^l dt\\
		\!\!\!\!\!\!\!\!&&~~~~+\mathbb{E}\bigg\{\sup_{t\in [0,T]}\Big|\int_{0}^{t}\|X_s^i\|_{\mathbb{H}}^{\beta}\langle X^i_s, \mathcal{B}(s, X^i_s, \mathscr{L}_{X^i_s})  d W^i_{s}\rangle_{\mathbb{H}}\Big|\bigg\}^l\\
		\!\!\!\!\!\!\!\!&&\lesssim_T1+\mathbb{E}\|\xi^i \|_{\mathbb{H}}^{(\beta +2)l}+\int_{0}^{T}\mathbb{E}\|X_t^i\|_{\mathbb{H}}^{(\beta+2)l}dt\\
		\!\!\!\!\!\!\!\!&&~~~~+\mathbb{E}\bigg\{\int_{0}^{T}\|X_t^i\|_{\mathbb{H}}^{2\beta+2}\|\mathcal{B}(t, X^i_t, \mathscr{L}_{X^i_t})\|_{L_{2}(U;\mathbb{H})}^{2}dt\bigg\}^\frac{l}{2} \\
		\!\!\!\!\!\!\!\!&&\lesssim_T1+\mathbb{E}\|\xi^i \|_{\mathbb{H}}^{(\beta +2)l}+\int_{0}^{T}\mathbb{E}\|X_t^i\|_{\mathbb{H}}^{(\beta+2)l}dt\\
		\!\!\!\!\!\!\!\!&&~~~~+\mathbb{E}\bigg\{\int_{0}^{T}\|X_t^i\|_{\mathbb{H}}^{2\beta+2}\big(1+\|X_t^i\|_{\mathbb{H}}^{\beta}+\mathbb{E}\|X_t^i\|_{\mathbb{H}}^{\beta}\big)dt\bigg\}^\frac{l}{2} \\
		\!\!\!\!\!\!\!\!&&\lesssim_T1+\mathbb{E}\|\xi^i \|_{\mathbb{H}}^{(\beta +2)l}+\int_{0}^{T}\mathbb{E}\|X_t^i\|_{\mathbb{H}}^{(\beta+2)l}dt\\
		\!\!\!\!\!\!\!\!&&~~~~+\mathbb{E}\bigg\{\sup_{t \in[0,T]}\|X_t^i\|_{\mathbb{H}}^{\beta+2}\cdot \int_{0}^{T}\|X_t^i\|_{\mathbb{H}}^{\beta}\big(1+\|X_t^i\|_{\mathbb{H}}^{\beta}+\mathbb{E}\|X_t^i\|_{\mathbb{H}}^{\beta}\big)dt\bigg\}^\frac{l}{2}\\
		\!\!\!\!\!\!\!\!&&\lesssim_T1+\mathbb{E}\|\xi^i \|_{\mathbb{H}}^{(\beta +2)l}+\mathbb{E}\Big[\sup_{t \in[0,T]}\|X_t^i\|_{\mathbb{H}}^{(\beta+2)l}\Big]+\sup_{t \in[0,T]}\mathbb{E}\|X_t^i\|_{\mathbb{H}}^{2\beta l}.
	\end{eqnarray*}
	Thus, (\ref{esq370}) implies the desired result.
\end{proof}


\subsection{Proof of   Theorem \ref{rate1}}\label{Proofrate1}
This subsection is devoted to  proving the PoC under the local monotonicity condition $(\mathbf{A}'_5)$.

Before doing that, we present two key lemmas, which are crucial for studying the dimension-free convergence rate of PoC for the IPS (\ref{IPS}).
\begin{lemma}\label{iidrate}
	 Let $\{Z_j\}_{j=1}^{\infty}$ be i.i.d.~$\mathbb{R}$-valued random variables with $\mathbb{E}|Z_1|^r<\infty$ for some $r\ge2$. For any $R>\mathbb{E}|Z_1|$, we have
	\begin{equation*}
		\mathbb{P}\bigg(\frac{1}{N}\sum_{j=1}^N Z_j \ge R\bigg) \lesssim N^{-\frac{r}{2}}
	\end{equation*}
	holds for all $N\in\mathbb{N}$.
\end{lemma}
\begin{proof}
	This is a standard result from Rosenthal's inequality (cf.~Theorem 3 in \cite{R70}), we include the proof for reader's convenience.
	
	Let $Y_j:=Z_j-\mathbb{E}Z_1$. Then by Markov's inequality,
	\begin{equation}\label{esmi1}	
	\mathbb{P}\bigg(\frac{1}{N}\sum_{j=1}^N Z_j \ge R\bigg)
	=\mathbb{P}\bigg(\sum_{j=1}^N Y_j \geq N(R-\mathbb{E}Z_1)\bigg) \lesssim \frac{\mathbb{E}|\sum_{j=1}^N Y_j|^r}{N^r(R-\mathbb{E}Z_1)^r}	.
	\end{equation}	
Since $\{Y_j\}_{j=1}^{\infty}$ is an i.i.d. random variable sequence, it follows from	Rosenthal's inequality that
\begin{equation}\label{esmi2}	\mathbb{E}|\sum_{j=1}^N Y_j|^r\lesssim_r N^{\frac{r}{2}}. 	
\end{equation}	
Combining (\ref{esmi1}) and (\ref{esmi2}), we get the desired result.	
\end{proof}

\begin{lemma}\label{mdlemma}
	For any $1< q \leq 2$, let $\mathbb{X}$ be a separable $q$-uniformly smooth Banach space. Then for any $\mathbb{X}$-valued martingale difference sequence $\{Y_j\}_{j=1}^n$ with finite $q$-th moment, we have
	\begin{equation*}
		\mathbb{E}\| \sum_{j=1}^n Y_j \|_{\mathbb{X}}^q \lesssim \mathbb{E}\sum_{j=1}^n \|Y_j\|_{\mathbb{X}}^q.
	\end{equation*}
\end{lemma}
\begin{proof}
	Let $\{X_j\}_{j=1}^n$ be a martingale sequence in the sense of Definition \ref{mddef} and set $X_0=0$. We can denote
	$$
	Y_j=X_j-X_{j-1}.
	$$
	Using Proposition 2.4 in \cite{P1975}, it follows that
	\begin{equation*}
		\mathbb{E}\| \sum_{j=1}^n Y_j \|_{\mathbb{X}}^q
		=\mathbb{E}\| \sum_{j=1}^n(X_j-X_{j-1})\|_{\mathbb{X}}^q
		=\mathbb{E}\|X_n\|_{\mathbb{X}}^q
		\lesssim\mathbb{E}\sum_{j=1}^n \|Y_j\|_{\mathbb{X}}^q,
	\end{equation*}
	which  yields the desired result.
\end{proof}

Set the following stopping time
$$
\tau^K := \inf\bigg\{ t \geq 0 : \frac{1}{N} \sum_{i=1}^N \|X_t^i\|_{\mathbb{H}}^\beta+\frac{1}{N} \sum_{i=1}^N \int_0^t \|X_s^i\|_{\mathbb{V}}^{\alpha}(1+ \|X_s^i\|_{\mathbb{H}}^\beta) ds   \geq K \bigg\}
$$
where
$
\bar{\mu}_t^{N} := \frac{1}{N} \sum_{j=1}^N \delta_{X_t^{j}}
$
is the empirical law of solution $(X_t^{1},X_t^{2},\dots,X_t^{N})$ to nIPS (\ref{nIPS}).

By H\"older's inequality, it holds that
\begin{eqnarray}\label{rate1all}
	\!\!\!\!\!\!\!\!&&\frac{1}{N} \sum_{i=1}^N \mathbb{E} \Big[ \sup_{t \in [0,T]} \|X_t^i - X_t^{i,N}\|_{\mathbb{H}}^2 \Big]\nonumber\\
	=\!\!\!\!\!\!\!\!&&\frac{1}{N} \sum_{i=1}^N \mathbb{E} \Big[ \sup_{t \in [0,T]} \|X_t^i - X_t^{i,N}\|_{\mathbb{H}}^2 \mathbf{1}_{\{\tau^K \geq T\}} \Big]+ \frac{1}{N} \sum_{i=1}^N \mathbb{E} \Big[ \sup_{t \in [0,T]} \|X_t^i - X_t^{i,N}\|_{\mathbb{H}}^2 \mathbf{1}_{\{\tau^K < T\}} \Big] \nonumber\\
	\leq\!\!\!\!\!\!\!\!&&\frac{1}{N} \sum_{i=1}^N \mathbb{E} \Big[ \sup_{t \in [0,T]} \|X_{t \wedge \tau^K}^i - X_{t \wedge \tau^K}^{i,N}\|_{\mathbb{H}}^2 \Big] \nonumber\\
	\!\!\!\!\!\!\!\!&&+\frac{1}{N} \sum_{i=1}^N \Big\{ \mathbb{E} \Big[ \sup_{t \in [0,T]} \|X_t^i - X_t^{i,N}\|_{\mathbb{H}}^{\gamma p} \Big] \Big\}^{\frac{2}{\gamma p}} \Big\{ \mathbb{P}(\tau^K < T) \Big\}^{\frac{\gamma p-2}{\gamma p}},
\end{eqnarray}
where we choose $\gamma\in(\frac{2}{p},1)$.

\vspace{1mm}
\noindent\textbf{Proof of (\ref{pocth1}).} We will consider the bounds for the two terms on the right-hand side of (\ref{rate1all}), respectively, thereby completing the proof.

\vspace{1mm}
\noindent\textbf{Step 1.} In this step, we study the first term on the right-hand side of (\ref{rate1all}).

Applying It\^{o}'s formula to $\|X_t^i - X_t^{i,N}\|_{\mathbb{H}}^2$, we have
\begin{eqnarray}\label{R1all}
	\!\!\!\!\!\!\!\!&&\|X_t^i - X_t^{i,N}\|_{\mathbb{H}}^2\nonumber\\
	=\!\!\!\!\!\!\!\!&& \int_0^t \Big( 2\,_{\mathbb{V}^*}\langle \mathcal{A}(s, X_s^i, \mathscr{L}_{X^i_s}) - \mathcal{A}(s, X_s^{i,N}, \mu_s^N), X_s^i - X_s^{i,N} \rangle_{\mathbb{V}}\nonumber\\
	\!\!\!\!\!\!\!\!&&\quad + \|\mathcal{B}(s, X_s^i, \mathscr{L}_{X^i_s}) - \mathcal{B}(s, X_s^{i,N}, \mu_s^N)\|_{L_2(U;\mathbb{H})}^2 \Big) ds \nonumber\\
	\!\!\!\!\!\!\!\!&&+ 2 \int_0^t \langle  X_s^i - X_s^{i,N},\big(\mathcal{B}(s, X_s^i, \mathscr{L}_{X^i_s}) - \mathcal{B}(s, X_s^{i,N}, \mu_s^N)\big)dW_s^i\rangle_{\mathbb{H}} \nonumber\\
	\lesssim\!\!\!\!\!\!\!\!&&\int_0^t\,_{\mathbb{V}^*}\langle \mathcal{A}(s, X_s^i, \bar{\mu}_s^N)- \mathcal{A}(s, X_s^{i,N},\mu_s^N), X_s^i - X_s^{i,N} \rangle_{\mathbb{V}}ds\nonumber\\
	\!\!\!\!\!\!\!\!&&+\int_0^t\,_{\mathbb{V}^*}\langle \mathcal{A}(s, X_s^i,\mathscr{L}_{X^i_s}) -\mathcal{A}(s, X_s^i,\bar{\mu}_s^N), X_s^i - X_s^{i,N} \rangle_{\mathbb{V}}ds\nonumber\\
	\!\!\!\!\!\!\!\!&&+\int_0^t\|\mathcal{B}(s, X_s^i, \bar{\mu}_s^N) - \mathcal{B}(s, X_s^{i,N}, \mu_s^N)\|_{L_2(U;\mathbb{H})}^2ds \nonumber\\
	\!\!\!\!\!\!\!\!&&+\int_0^t\|\mathcal{B}(s, X_s^i, \mathscr{L}_{X^i_s})-\mathcal{B}(s, X_s^i, \bar{\mu}_s^N)\|_{L_2(U;\mathbb{H})}^2ds+\mathcal{M}_t,
\end{eqnarray}
where
\begin{equation*}
	\mathcal{M}_t:=\int_0^t \langle  X_s^i - X_s^{i,N},\big(\mathcal{B}(s, X_s^i, \mathscr{L}_{X^i_s})- \mathcal{B}(s, X_s^{i,N}, \mu_s^N)\big)dW_s^i\rangle_{\mathbb{H}}.
\end{equation*}
Using B-D-G's inequality and Young's inequality, we derive
\begin{eqnarray}\label{RMt}
	\!\!\!\!\!\!\!\!&&\frac{1}{N} \sum_{i=1}^N \mathbb{E} \Big[ \sup_{t \in [0,T]} |\mathcal{M}_{t \wedge \tau^K}| \Big]\nonumber\\
	\leq\!\!\!\!\!\!\!\!&& \frac{1}{N} \sum_{i=1}^N \mathbb{E} \bigg( \int_0^{T \wedge \tau^K} \|X_t^i - X_t^{i,N}\|_{\mathbb{H}}^2\cdot\|\mathcal{B}(t, X_t^i, \mathscr{L}_{X^i_t})- \mathcal{B}(t, X_t^{i,N}, \mu_t^N)\|_{L_2(U;\mathbb{H})}^2 dt \bigg)^{\frac{1}{2}}\nonumber\\
	\leq\!\!\!\!\!\!\!\!&& \frac{1}{N} \sum_{i=1}^N \mathbb{E} \bigg[ \Big( \sup_{t \in [0,T \wedge \tau^K]} \|X_t^i - X_t^{i,N}\|_{\mathbb{H}}^2 \Big)^{\frac{1}{2}}\nonumber\\
	\!\!\!\!\!\!\!\!&& \cdot \bigg( \int_0^{T \wedge \tau^K} \|\mathcal{B}(t, X_t^i, \mathscr{L}_{X^i_t}) - \mathcal{B}(t, X_t^{i,N}, \mu_t^N)\|_{L_2(U;\mathbb{H})}^2 dt \bigg)^{\frac{1}{2}} \bigg] \nonumber\\
	\leq\!\!\!\!\!\!\!\!&& \frac{C}{N} \sum_{i=1}^N \mathbb{E} \int_0^{T \wedge \tau^K} \|\mathcal{B}(t, X_t^i, \mathscr{L}_{X^i_t}) - \mathcal{B}(t, X_t^{i,N}, \mu_t^N)\|_{L_2(U;\mathbb{H})}^2 dt \nonumber\\
	\!\!\!\!\!\!\!\!&&+ \frac{1}{2N} \sum_{i=1}^N \mathbb{E} \Big[ \sup_{t \in [0,T \wedge \tau^K]} \|X_t^i - X_t^{i,N}\|_{\mathbb{H}}^2 \Big]\nonumber\\
	\leq\!\!\!\!\!\!\!\!&&\frac{C}{N} \sum_{i=1}^N \mathbb{E} \int_0^{T \wedge \tau^K}\|\mathcal{B}(t, X_t^i, \bar{\mu}_t^N)-\mathcal{B}(t, X_t^{i,N}, \mu_t^N)\|_{L_2(U;\mathbb{H})}^2dt\nonumber\\
	\!\!\!\!\!\!\!\!&&+\frac{C}{N} \sum_{i=1}^N \mathbb{E} \int_0^{T \wedge \tau^K}\|\mathcal{B}(t, X_t^i, \mathscr{L}_{X^i_t})-\mathcal{B}(t, X_t^i, \bar{\mu}_t^N) \|_{L_2(U;\mathbb{H})}^2dt\nonumber\\
	\!\!\!\!\!\!\!\!&&+ \frac{1}{2N} \sum_{i=1}^N \mathbb{E} \Big[ \sup_{t \in [0,T \wedge \tau^K]} \|X_t^i - X_t^{i,N}\|_{\mathbb{H}}^2 \Big].
\end{eqnarray}
Substituting (\ref{RMt}) into (\ref{R1all}), we have
\begin{eqnarray}\label{R1I1-4}
	\!\!\!\!\!\!\!\!&&\frac{1}{N} \sum_{i=1}^N \mathbb{E} \Big[ \sup_{t \in [0,T \wedge \tau^K]} \|X_t^i - X_t^{i,N}\|_{\mathbb{H}}^2 \Big]\nonumber\\
	\lesssim\!\!\!\!\!\!\!\!&& \frac{1}{N} \sum_{i=1}^N \mathbb{E} \int_0^{T \wedge \tau^K} \,_{\mathbb{V}^*}\langle \mathcal{A}(t, X_t^i, \bar{\mu}_t^N) - \mathcal{A}(t, X_t^{i,N}, \mu_t^N), X_t^i - X_t^{i,N} \rangle_{\mathbb{V}}dt\nonumber\\
	\!\!\!\!\!\!\!\!&&+ \frac{1}{N} \sum_{i=1}^N \mathbb{E} \int_0^{T \wedge \tau^K} \|\mathcal{B}(t, X_t^i, \bar{\mu}_t^N) - \mathcal{B}(t, X_t^{i,N}, \mu_t^N)\|_{L_2(U;\mathbb{H})}^2dt \nonumber\\
	\!\!\!\!\!\!\!\!&&+ \frac{1}{N} \sum_{i=1}^N \mathbb{E} \int_0^{T \wedge \tau^K} \,_{\mathbb{V}^*}\langle \mathcal{A}(t, X_t^i, \mathscr{L}_{X^i_t})-\mathcal{A}(t, X_t^i, \bar{\mu}_t^N), X_t^i - X_t^{i,N} \rangle_{\mathbb{V}}dt \nonumber\\
	\!\!\!\!\!\!\!\!&&+ \frac{1}{N} \sum_{i=1}^N \mathbb{E} \int_0^{T \wedge \tau^K} \|\mathcal{B}(t, X_t^i, \mathscr{L}_{X^i_t})-\mathcal{B}(t, X_t^i, \bar{\mu}_t^N)\|_{L_2(U;\mathbb{H})}^2dt \nonumber\\
	=:\!\!\!\!\!\!\!\!&& \sum_{m=1}^4 I_m.
\end{eqnarray}
Due to the condition $(\mathbf{A}'_5)$, we obtain
\begin{eqnarray}\label{I1-2}
	I_1 + I_2
	\lesssim\!\!\!\!\!\!\!\!&& \frac{1}{N} \sum_{i=1}^N \mathbb{E} \int_0^{T \wedge \tau^K} \rho(0, \bar{\mu}_s^N)  \|X_s^i - X_s^{i,N}\|_{\mathbb{H}}^2 ds\nonumber\\
	\!\!\!\!\!\!\!\!&&+ \frac{1}{N} \sum_{j=1}^N \mathbb{E} \int_0^{T \wedge \tau^K} \rho(X_s^j,\bar{\mu}_s^N)  \mathcal{W}_{2,\mathbb{H}}(\bar{\mu}_s^N, \mu_s^N)^2 ds\nonumber\\
	\lesssim\!\!\!\!\!\!\!\!&& \frac{1}{N} \sum_{i=1}^N \mathbb{E} \int_0^{T \wedge \tau^K}  \rho(0, \bar{\mu}_s^N)  \|X_s^i - X_s^{i,N}\|_{\mathbb{H}}^2 ds\nonumber\\
	\!\!\!\!\!\!\!\!&&+ \frac{1}{N} \sum_{i=1}^N \mathbb{E} \int_0^{T \wedge \tau^K} \Big[ \frac{1}{N} \sum_{j=1}^N  \rho(X_s^j, \bar{\mu}_s^N)  \Big] \|X_s^i - X_s^{i,N}\|_{\mathbb{H}}^2 ds.
\end{eqnarray}
Moreover, we claim that
\begin{equation}\label{I3}
	I_3\lesssim N^{-\frac{1}{\alpha}}
\end{equation}
and
\begin{equation}\label{I4}
	I_4\lesssim N^{-1},
\end{equation}
whose proof will be presented in \textbf{Step 2} below.

Substituting (\ref{I1-2})-(\ref{I4}) back into (\ref{R1I1-4}), we derive
\begin{eqnarray*}
	\!\!\!\!\!\!\!\!&&\frac{1}{N} \sum_{i=1}^N \mathbb{E} \Big[ \sup_{t \in [0,T \wedge \tau^K]} \|X_t^i - X_t^{i,N}\|_{\mathbb{H}}^2 \Big]\nonumber\\
	\lesssim\!\!\!\!\!\!\!\!&& \frac{1}{N} \sum_{i=1}^N \mathbb{E} \int_0^{T \wedge \tau^K}  \rho(0, \bar{\mu}_s^N)  \|X_s^i - X_s^{i,N}\|_{\mathbb{H}}^2 ds\nonumber\\
	\!\!\!\!\!\!\!\!&&+ \frac{1}{N} \sum_{i=1}^N \mathbb{E} \int_0^{T \wedge \tau^K} \Big[ \frac{1}{N} \sum_{j=1}^N \rho(X_s^j, \bar{\mu}_s^N)  \Big] \|X_s^i - X_s^{i,N}\|_{\mathbb{H}}^2 ds+N^{-\frac{1}{\alpha}}\nonumber\\
	\lesssim\!\!\!\!\!\!\!\!&&\frac{1}{N} \sum_{i=1}^N \mathbb{E} \int_0^{T \wedge \tau^K}
	\Big[\frac{1}{N} \sum_{j=1}^N \big( 1 + \|X_s^j\|_{\mathbb{V}}^\alpha  + \frac{1}{N} \sum_{l=1}^N \|X_s^l\|_{\mathbb{V}}^\alpha \big) \big( 1 + \|X_s^j\|_{\mathbb{H}}^\beta + \frac{1}{N} \sum_{l=1}^N \|X_s^l\|_{\mathbb{H}}^\beta\big) \Big]\nonumber\\
	\!\!\!\!\!\!\!\!&&\quad\quad\quad\quad\quad\quad\quad\cdot\|X_s^i - X_s^{i,N}\|_{\mathbb{H}}^2 ds+N^{-\frac{1}{\alpha}}.
\end{eqnarray*}
Then due to the definition of stopping time $\tau^K$, the stochastic Gronwall's lemma (cf.~\cite[Lemma 5.3]{GZ}) yields that
\begin{equation}\label{step1result}
	\frac{1}{N} \sum_{i=1}^N \mathbb{E} \Big[ \sup_{t \in [0,T \wedge \tau^K]} \|X_t^i - X_t^{i,N}\|_{\mathbb{H}}^2 \Big]\lesssim_K N^{-\frac{1}{\alpha}}.
\end{equation}

\noindent\textbf{Step 2.} In this step, we aim to prove (\ref{I3}) and (\ref{I4}).

We first focus on (\ref{I3}). By H\"older's inequality and $(\mathbf{A}_6)$, Lemma \ref{IPSesti2} and the estimate (\ref{esq370}) lead to
\begin{eqnarray*}
	I_3
	\!\!\!\!\!\!\!\!&&=\frac{1}{N} \sum_{i=1}^N \mathbb{E} \int_0^{T \wedge \tau^K} \,_{\mathbb{V}^*}\langle \mathcal{A}(t, X_t^i, \mathscr{L}_{X^i_t})-\mathcal{A}(t, X_t^i, \bar{\mu}_t^N), X_t^i - X_t^{i,N} \rangle_{\mathbb{V}}dt\\
	\!\!\!\!\!\!\!\!&&\lesssim \frac{1}{N} \sum_{i=1}^N \mathbb{E} \int_0^T \|X_t^i - X_t^{i,N}\|_{\mathbb{V}}\cdot \|\mathcal{A}(t, X_t^i, \mathscr{L}_{X^i_t})-\mathcal{A}(t, X_t^i, \bar{\mu}_t^N)\|_{\mathbb{V}^*} dt
	\\
	\!\!\!\!\!\!\!\!&&\lesssim \frac{1}{N} \sum_{i=1}^N \mathbb{E} \int_0^T \|X_t^i - X_t^{i,N}\|_{\mathbb{V}}\cdot \| \int \tilde{\mathcal{A}}(t, X_t^i, y) \mathscr{L}_{X^i_t}(dy)-\int \tilde{\mathcal{A}}(t, X_t^i, y) \bar{\mu}_t^N(dy) \|_{\mathbb{V}^*} dt
	\\
	\!\!\!\!\!\!\!\!&&= \frac{1}{N^2} \sum_{i=1}^N \mathbb{E} \int_0^T \|X_t^i - X_t^{i,N}\|_{\mathbb{V}}\cdot \| N \int \tilde{\mathcal{A}}(t, X_t^i, y) \mathscr{L}_{X^i_t}(dy) - \sum_{j=1}^N \tilde{\mathcal{A}}(t, X_t^i, X_t^j) \|_{\mathbb{V}^*} dt \\
	\!\!\!\!\!\!\!\!&&\lesssim \frac{1}{N^2} \sum_{i=1}^N \bigg\{ \mathbb{E} \int_0^T \|X_t^i - X_t^{i,N}\|_{\mathbb{V}}^\alpha  dt \bigg\}^{\frac{1}{\alpha}}\\
	\!\!\!\!\!\!\!\!&& \quad\quad\quad\quad\quad \cdot \bigg\{ \mathbb{E} \int_0^T \| N \int \tilde{\mathcal{A}}(t, X_t^i, y) \mathscr{L}_{X^i_t}(dy) - \sum_{j=1}^N \tilde{\mathcal{A}}(t, X_t^i, X_t^j)\|_{\mathbb{V}^*}^{\frac{\alpha}{\alpha-1}} dt\bigg\}^{\frac{\alpha-1}{\alpha}} \\
	\!\!\!\!\!\!\!\!&&\lesssim \frac{1}{N^2} \sum_{i=1}^N \bigg\{ \mathbb{E} \int_0^T \| \sum_{j=1}^N \Big( \int \tilde{\mathcal{A}}(t, X_t^i, y) \mathscr{L}_{X^i_t}(dy) -  \tilde{\mathcal{A}}(t, X_t^i, X_t^j)\Big)\|_{\mathbb{V}^*}^{\frac{\alpha}{\alpha-1}} dt \bigg\}^{\frac{\alpha-1}{\alpha}}.
\end{eqnarray*}
We claim that
\begin{equation}\label{keyI3}
	\mathbb{E} \int_0^T \| \sum_{j=1}^N \Big( \int \tilde{\mathcal{A}}(t, X_t^i, y) \mathscr{L}_{X^i_t}(dy) -  \tilde{\mathcal{A}}(t, X_t^i, X_t^j)\Big)\|_{\mathbb{V}^*}^{\frac{\alpha}{\alpha-1}} dt\lesssim N,
\end{equation}
then  it follows that
\begin{equation*}
	I_3 \lesssim \frac{1}{N^2} \sum_{i=1}^N N^{\frac{\alpha-1}{\alpha}} \lesssim N^{-\frac{1}{\alpha}},
\end{equation*}
which implies (\ref{I3}).

Now, it  remains to prove (\ref{keyI3}). Note that
\begin{eqnarray*}
	\!\!\!\!\!\!\!\!&&\mathbb{E} \int_0^T \| \sum_{j=1}^N \Big( \int \tilde{\mathcal{A}}(t, X_t^i, y) \mathscr{L}_{X^i_t}(dy) -  \tilde{\mathcal{A}}(t, X_t^i, X_t^j)\Big)\|_{\mathbb{V}^*}^{\frac{\alpha}{\alpha-1}} dt\\
	\lesssim\!\!\!\!\!\!\!\!&&\mathbb{E} \int_0^T \| \int \tilde{\mathcal{A}}(t, X_t^i, y) \mathscr{L}_{X^i_t}(dy) -  \tilde{\mathcal{A}}(t, X_t^i, X_t^i)\|_{\mathbb{V}^*}^{\frac{\alpha}{\alpha-1}} dt\\
	\!\!\!\!\!\!\!\!&&+\mathbb{E} \int_0^T \| \sum_{j\ne i}^N \Big( \int \tilde{\mathcal{A}}(t, X_t^i, y) \mathscr{L}_{X^i_t}(dy) -  \tilde{\mathcal{A}}(t, X_t^i, X_t^j)\Big)\|_{\mathbb{V}^*}^{\frac{\alpha}{\alpha-1}} dt\\
	\lesssim\!\!\!\!\!\!\!\!&&\mathbb{E} \int_0^T \Big( \int \|\tilde{\mathcal{A}}(t, X_t^i, y)\|_{\mathbb{V}^*}^{\frac{\alpha}{\alpha-1}}\mathscr{L}_{X^i_t}(dy) + \|\tilde{\mathcal{A}}(t, X_t^i, X_t^i)\|_{\mathbb{V}^*}^{\frac{\alpha}{\alpha-1}} \Big) dt \\
	\!\!\!\!\!\!\!\!&&+ \int_0^T \mathbb{E} \Big[ \mathbb{E} \| \sum_{j\ne i}^N \Big( \int \tilde{\mathcal{A}}(t, x, y) \mathscr{L}_{X^i_t}(dy) -  \tilde{\mathcal{A}}(t, x, X_t^j)\Big)\|_{\mathbb{V}^*}^{\frac{\alpha}{\alpha-1}} \Big|_{x = X_t^i} \Big] dt.
\end{eqnarray*}

Recall the fact that $\mathbb{V}$ is a separable $\alpha$-uniformly convex Banach space with $\alpha\geq2$. By  Proposition \ref{dualization}, we know that $\mathbb{V}^*$ is a separable $\frac{\alpha}{\alpha-1}$-uniformly smooth Banach space. Since $\big\{\int \tilde{\mathcal{A}}(t, x, y) \mathscr{L}_{X^i_t}(dy) -  \tilde{\mathcal{A}}(t, x, X_t^j)\big\}_{j\neq i}$ is a sequence of martingale differences, in view of Lemma \ref{mdlemma}, it follows that
\begin{eqnarray*}
	\!\!\!\!\!\!\!\!&&\mathbb{E} \int_0^T \| \sum_{j=1}^N \Big( \int \tilde{\mathcal{A}}(t, X_t^i, y) \mathscr{L}_{X^i_t}(dy) -  \tilde{\mathcal{A}}(t, X_t^i, X_t^j)\Big)\|_{\mathbb{V}^*}^{\frac{\alpha}{\alpha-1}} dt\\
	\lesssim\!\!\!\!\!\!\!\!&&\mathbb{E} \int_0^T \Big(\int \|\tilde{\mathcal{A}}(t, X_t^i, y)\|_{\mathbb{V}^*}^{\frac{\alpha}{\alpha-1}}\mathscr{L}_{X^i_t}(dy) + \|\tilde{\mathcal{A}}(t, X_t^i, X_t^i)\|_{\mathbb{V}^*}^{\frac{\alpha}{\alpha-1}} \Big) dt \\
	\!\!\!\!\!\!\!\!&&+ \int_0^T \mathbb{E} \Big[ \mathbb{E} \sum_{j\ne i}^N\| \int \tilde{\mathcal{A}}(t, x, y) \mathscr{L}_{X^i_t}(dy) -  \tilde{\mathcal{A}}(t, x, X_t^j)\|_{\mathbb{V}^*}^{\frac{\alpha}{\alpha-1}} \Big|_{x = X_t^i} \Big] dt\\
	\lesssim\!\!\!\!\!\!\!\!&&\mathbb{E} \int_0^T \Big( \int \|\tilde{\mathcal{A}}(t, X_t^i, y)\|_{\mathbb{V}^*}^{\frac{\alpha}{\alpha-1}}\mathscr{L}_{X^i_t}(dy) + \|\tilde{\mathcal{A}}(t, X_t^i, X_t^i)\|_{\mathbb{V}^*}^{\frac{\alpha}{\alpha-1}} \Big) dt \\
	\!\!\!\!\!\!\!\!&&+ \sum_{j\ne i}^N  \mathbb{E}\int_0^T\Big( \int \|\tilde{\mathcal{A}}(t, X_t^i, y)\|_{\mathbb{V}^*}^{\frac{\alpha}{\alpha-1}}\mathscr{L}_{X^i_t}(dy) + \|\tilde{\mathcal{A}}(t, X_t^i, X_t^j)\|_{\mathbb{V}^*}^{\frac{\alpha}{\alpha-1}} \Big) dt\\
	\lesssim\!\!\!\!\!\!\!\!&& N.
\end{eqnarray*}
Here, the last step is due to
\begin{eqnarray*}
	\!\!\!\!\!\!\!\!&&\mathbb{E}\int_0^T\int \|\tilde{\mathcal{A}}(t, X_t^i, y)\|_{\mathbb{V}^*}^{\frac{\alpha}{\alpha-1}}\mathscr{L}_{X^i_t}(dy)dt\\
	\lesssim\!\!\!\!\!\!\!\!&& \mathbb{E}\int_0^T\int\big((1+\|X_t^i\|_{\mathbb{V}}^{\alpha}+\|y\|_{\mathbb{V}}^{\alpha})(1+\|X_t^i\|_{{\mathbb{H}}}^{\beta}+\|y\|_{\mathbb{H}}^{\beta})\big)\mathscr{L}_{X^i_t}(dy)dt \\
	\lesssim\!\!\!\!\!\!\!\!&& \mathbb{E}\int_0^T\big((1+\|X_t^i\|_{\mathbb{V}}^{\alpha})(1+\|X_t^i\|_{{\mathbb{H}}}^{\beta})\big)dt+\int_0^T\mathbb{E} \|X_t^i\|_{{\mathbb{H}}}^{\beta}\cdot\mathbb{E}\|X_t^i\|_{\mathbb{V}}^{\alpha}dt \\
	\lesssim\!\!\!\!\!\!\!\!&& \mathbb{E}\int_0^T\big((1+\|X_t^i\|_{\mathbb{V}}^{\alpha})(1+\|X_t^i\|_{{\mathbb{H}}}^{\beta})\big)dt+\sup_{t\in[0,T]}\mathbb{E} \|X_t^i\|_{{\mathbb{H}}}^{\beta}\cdot\mathbb{E}\int_0^T\|X_t^i\|_{\mathbb{V}}^{\alpha}dt \\
	<\!\!\!\!\!\!\!\!&&\infty
\end{eqnarray*}
and
\begin{eqnarray}\label{es8}
	\!\!\!\!\!\!\!\!&&\mathbb{E}\int_0^T \|\tilde{\mathcal{A}}(t, X_t^i, X_t^j)\|_{\mathbb{V}^*}^{\frac{\alpha}{\alpha-1}}dt\nonumber\\
	\lesssim\!\!\!\!\!\!\!\!&& \mathbb{E}\int_0^T\big((1+\|X_t^i\|_{\mathbb{V}}^{\alpha}+\|X_t^j\|_{\mathbb{V}}^{\alpha})(1+\|X_t^i\|_{{\mathbb{H}}}^{\beta}+\|X_t^j\|_{\mathbb{H}}^{\beta})\big)dt \nonumber\\
	\lesssim\!\!\!\!\!\!\!\!&& \mathbb{E}\int_0^T\big((1+\|X_t^i\|_{\mathbb{V}}^{\alpha})(1+\|X_t^i\|_{{\mathbb{H}}}^{\beta})\big)dt+\int_0^T\mathbb{E} \|X_t^i\|_{{\mathbb{H}}}^{\beta}\cdot\mathbb{E}\|X_t^j\|_{\mathbb{V}}^{\alpha}dt \nonumber\\
	<\!\!\!\!\!\!\!\!&&\infty,
\end{eqnarray}
where we have used the independence of the particles $X_t^i$ and $X_t^j$, with $j\neq i$, in the last step of (\ref{es8}).

On the other hand, according to Remark \ref{uniformspace} and Proposition \ref{dualization},  it is straightforward that the Hilbert space $L_2(U;\mathbb{H})$ is $2$-uniformly smooth. Following similar arguments as in the proof of $I_3$, we have
\begin{eqnarray*}
	I_4
	\lesssim\!\!\!\!\!\!\!\!&&\frac{1}{N} \sum_{i=1}^N \mathbb{E} \int_0^{T \wedge \tau^K} \|\int \tilde{\mathcal{B}}(t, X_t^i, y) \mathscr{L}_{X^i_t}(dy)-\int \tilde{\mathcal{B}}(t, X_t^i, y) \bar{\mu}_t^N(dy)\|_{L_2(U;\mathbb{H})}^2dt\\
	= \!\!\!\!\!\!\!\!&&\frac{1}{N^3} \sum_{i=1}^N \mathbb{E} \int_0^T  \| N \int \tilde{\mathcal{B}}(t, X_t^i, y) \mathscr{L}_{X^i_t}(dy) - \sum_{j=1}^N \tilde{\mathcal{B}}(t, X_t^i, X_t^j) \|_{L_2(U;\mathbb{H})}^2dt \\
	\lesssim\!\!\!\!\!\!\!\!&&\frac{1}{N^3} \sum_{i=1}^N\mathbb{E} \int_0^T \| \int \tilde{\mathcal{B}}(t, X_t^i, y) \mathscr{L}_{X^i_t}(dy) -  \tilde{\mathcal{B}}(t, X_t^i, X_t^i)\|_{L_2(U;\mathbb{H})}^2 dt\\
	\!\!\!\!\!\!\!\!&&+\frac{1}{N^3} \sum_{i=1}^N\mathbb{E} \int_0^T \| \sum_{j\ne i}^N \Big( \int \tilde{\mathcal{B}}(t, X_t^i, y) \mathscr{L}_{X^i_t}(dy) -  \tilde{\mathcal{B}}(t, X_t^i, X_t^j)\Big)\|_{L_2(U;\mathbb{H})}^2 dt\\
	\lesssim\!\!\!\!\!\!\!\!&&\frac{1}{N^3} \sum_{i=1}^N\mathbb{E} \int_0^T \Big( \int\|\tilde{\mathcal{B}}(t, X_t^i, y)\|_{L_2(U;\mathbb{H})}^2\mathscr{L}_{X^i_t}(dy) + \|\tilde{\mathcal{B}}(t, X_t^i, X_t^i)\|_{L_2(U;\mathbb{H})}^2 \Big) dt \\
	\!\!\!\!\!\!\!\!&&+ \frac{1}{N^3} \sum_{i=1}^N\int_0^T \mathbb{E} \Big[ \mathbb{E} \| \sum_{j\ne i}^N \Big( \int \tilde{\mathcal{B}}(t, x, y) \mathscr{L}_{X^i_t}(dy) -  \tilde{\mathcal{B}}(t, x, X_t^j)\Big)\|_{L_2(U;\mathbb{H})}^2 \Big|_{x = X_t^i} \Big] dt\\
	\lesssim\!\!\!\!\!\!\!\!&&\frac{1}{N^2} + \frac{1}{N^3} \sum_{i=1}^N\int_0^T \mathbb{E} \Big[ \mathbb{E} \sum_{j\ne i}^N\| \int \tilde{\mathcal{B}}(t, x, y) \mathscr{L}_{X^i_t}(dy) -  \tilde{\mathcal{B}}(t, x, X_t^j)\|_{L_2(U;\mathbb{H})}^2 \Big|_{x = X_t^i} \Big] dt\\
	\lesssim\!\!\!\!\!\!\!\!&&\frac{1}{N^2} + \frac{1}{N^3} \sum_{i=1}^N \sum_{j\ne i}^N  \mathbb{E}\int_0^T\Big( \int\|\tilde{\mathcal{B}}(t, X_t^i, y)\|_{L_2(U;\mathbb{H})}^2\mathscr{L}_{X^i_t}(dy)
	\\
	\!\!\!\!\!\!\!\!&&+ \|\tilde{\mathcal{B}}(t, X_t^i, X_t^j)\|_{L_2(U;\mathbb{H})}^2 \Big) dt\\
	\lesssim\!\!\!\!\!\!\!\!&& N^{-1},
\end{eqnarray*}
which completes the proof of (\ref{I4}).

\vspace{1mm}
\noindent\textbf{Step 3.} In this step, we shall establish a bound for the second term on the right-hand side of (\ref{rate1all}).
By (\ref{esq370}) and (\ref{XNesti2}), it follows that
\begin{eqnarray}\label{taugeT}
	\!\!\!\!\!\!\!\!&&\frac{1}{N} \sum_{i=1}^N \Big\{ \mathbb{E} \Big[ \sup_{t \in [0,T]} \|X_t^i - X_t^{i,N}\|_{\mathbb{H}}^{\gamma p} \Big] \Big\}^{\frac{2}{\gamma p}} \Big\{ \mathbb{P}(\tau^K < T) \Big\}^{\frac{\gamma p-2}{\gamma p}}\nonumber\\
	\leq\!\!\!\!\!\!\!\!&& \frac{1}{N} \sum_{i=1}^N \Big\{ \mathbb{E} \Big[ \sup_{t \in [0,T]} \|X_t^i\|_{\mathbb{H}}^{\gamma p} + \|X_t^{i,N}\|_{\mathbb{H}}^{\gamma p} \Big] \Big\}^{\frac{2}{\gamma p}} \Big\{ \mathbb{P}(\tau^K < T) \Big\}^{\frac{\gamma p-2}{\gamma p}} \nonumber\\
	\lesssim\!\!\!\!\!\!\!\!&& \Big\{\mathbb{P}(\tau^K < T) \Big\}^{\frac{\gamma p-2}{\gamma p}}.
\end{eqnarray}
By the definition of stopping time $\tau^K$ we have
\begin{equation*}
	\mathbb{P}(\tau^K\leq T)\leq\mathbb{P}\bigg( \frac{1}{N} \sum_{i=1}^N \Big( \sup_{t \in [0,T]} \|X_t^i\|_{\mathbb{H}}^\beta + \int_0^T \|X_t^i\|_{\mathbb{V}}^{\alpha} (1+\|X_t^i\|_{\mathbb{H}}^\beta) dt\Big) \geq K \bigg).
\end{equation*}
Let
\begin{equation*}
	Z_i := \sup_{t \in [0,T]} \|X_t^i\|_{\mathbb{H}}^\beta + \int_0^T \|X_t^i\|_{\mathbb{V}}^{\alpha} (1+\|X_t^i\|_{\mathbb{H}}^\beta) dt.
\end{equation*}
Since $\{Z_i\}_{i=1}^N$ are i.i.d.~$\mathbb{R}$-valued random variables,
by virtue of Lemmas \ref{nIPSestibeta2} and \ref{nIPSesti2} we can derive
\begin{equation*}
	\mathbb{E}|Z_1|^2
	\lesssim \mathbb{E}\Big(\sup_{t \in [0,T]} \|X_t^1\|_{\mathbb{H}}^{2\beta}\Big)+\mathbb{E}\Big(\int_0^T \|X_t^1\|_{\mathbb{V}}^{\alpha} (1+\|X_t^1\|_{\mathbb{H}}^\beta) dt\Big)^2< \infty.
\end{equation*}
Thus by Lemma \ref{iidrate}, taking $K = 2\mathbb{E}Z_1$, we have
\begin{equation}\label{tauP2}
	\mathbb{P}(\tau^K\leq T) \lesssim N^{-1}.
\end{equation}
Substituting  (\ref{tauP2}) into (\ref{taugeT}), it follows that
\begin{equation*}
	\frac{1}{N} \sum_{i=1}^N \Big\{ \mathbb{E} \Big[ \sup_{t \in [0,T]} \|X_t^i - X_t^{i,N}\|_{\mathbb{H}}^{\gamma p} \Big] \Big\}^{\frac{2}{\gamma p}} \Big\{ \mathbb{P}(\tau^K < T) \Big\}^{\frac{\gamma p-2}{\gamma p}}
	\lesssim_{K}N^{- {\frac{\gamma p-2}{\gamma p}}}.
\end{equation*}

Recalling (\ref{rate1all}) and (\ref{step1result}), by the arbitrariness of $\gamma$ and the definition of $p$ we conclude that
\begin{equation*}
	\frac{1}{N} \sum_{i=1}^N \mathbb{E} \Big[ \sup_{t \in [0,T]} \|X_t^i - X_t^{i,N}\|_{\mathbb{H}}^2 \Big]\lesssim N^{-\frac{1}{\alpha}}+N^{- {\frac{\gamma p-2}{\gamma p}}}\lesssim N^{-\frac{1}{\alpha}}.
\end{equation*}
The proof is thus completed. \hspace{\fill}$\Box$

\subsection{Proof of Theorem \ref{rate2}}\label{Proofrate2}
In this subsection, we  establish the convergence rate under the local monotonicity condition $(\mathbf{A}''_5)$ by constructing alternative stopping times.

More specifically, recall the stopping time $\tau^K$ defined in Subsection \ref{Proofrate1} and set
$$
\tau^\varepsilon := \inf\bigg\{ t \geq 0 : \mathcal{W}_{\beta,\mathbb{H}} ( \bar{\mu}_t^{N}, \mu_t^{N})^\beta + \int_0^t \mathcal{W}_{\alpha,\mathbb{V}} ( \bar{\mu}_s^{N}, \mu_s^{N})^\alpha ds \geq \varepsilon \bigg\}.
$$
Let $\tilde {\tau} := \tau^\varepsilon \wedge \tau^K$.
By the triangle inequality and the definition of $\tilde {\tau}$, it holds for any $t \in [0,T]$ that
\begin{eqnarray}\label{XNeta}
	\!\!\!\!\!\!\!\!&&\frac{1}{N} \sum_{i=1}^N \int_0^{t \wedge \tilde {\tau}} \|X_s^{i,N}\|_{\mathbb{V}}^{\alpha} ds+\frac{1}{N} \sum_{i=1}^N \|X_{t \wedge \tilde {\tau}}^{i,N}\|_{\mathbb{H}}^\beta\nonumber\\
	\lesssim\!\!\!\!\!\!\!\!&& \int_0^{t \wedge \tilde {\tau}} \mathcal{W}_{\alpha,\mathbb{V}} ( \bar{\mu}_s^{N}, \mu_s^{N})^\alpha ds+\frac{1}{N} \sum_{i=1}^N \int_0^{t \wedge \tilde {\tau}} \|X_s^i\|_{\mathbb{V}}^{\alpha} ds
	\nonumber\\
	\!\!\!\!\!\!\!\!&&+\mathcal{W}_{\beta,\mathbb{H}}( \bar{\mu}_{t \wedge \tilde {\tau}}^{N}, \mu_{t \wedge \tilde {\tau}}^{N})^\beta+\frac{1}{N} \sum_{i=1}^N \|X_{t \wedge \tilde {\tau}}^{i}\|_{\mathbb{H}}^\beta\nonumber\\
	\leq\!\!\!\!\!\!\!\!&& \varepsilon + K.
\end{eqnarray}
Note that
\begin{eqnarray}\label{rate2all}
	\!\!\!\!\!\!\!\!&&\frac{1}{N} \sum_{i=1}^N \mathbb{E} \Big[ \sup_{t \in [0,T]} \|X_t^i - X_t^{i,N}\|_{\mathbb{H}}^2 \Big]+\frac{1}{N} \sum_{i=1}^N \mathbb{E} \int_0^{T}\|X_s^i - X_s^{i,N}\|_{\mathbb{V}}^{\alpha}ds\nonumber\\
	\leq\!\!\!\!\!\!\!\!&&\frac{1}{N} \sum_{i=1}^N \mathbb{E} \Big[\sup_{t \in [0,T]} \|X_{t \wedge \tilde {\tau}}^i - X_{t \wedge \tilde {\tau}}^{i,N}\|_{\mathbb{H}}^2 \Big]+\frac{1}{N} \sum_{i=1}^N \mathbb{E} \int_0^{T \wedge \tilde {\tau}}\|X_s^i - X_s^{i,N}\|_{\mathbb{V}}^{\alpha}ds \nonumber\\
	\!\!\!\!\!\!\!\!&&+\frac{1}{N} \sum_{i=1}^N \Big\{ \mathbb{E} \Big[\sup_{t \in [0,T]} \|X_t^i - X_t^{i,N}\|_{\mathbb{H}}^{\gamma p} \Big] \Big\}^{\frac{2}{\gamma p}} \Big\{ \mathbb{P}(\tilde {\tau} < T) \Big\}^{\frac{\gamma p-2}{\gamma p}}\nonumber\\
	\!\!\!\!\!\!\!\!&&+\frac{1}{N} \sum_{i=1}^N \Big\{ \mathbb{E} \Big(\int_0^{T}\|X_s^i - X_s^{i,N}\|_{\mathbb{V}}^{\alpha}ds \Big)^{\frac{\vartheta p}{2}} \Big\}^{\frac{2}{\vartheta p}} \Big\{ \mathbb{P}(\tilde {\tau} < T) \Big\}^{\frac{\vartheta p-2}{\vartheta p}},
\end{eqnarray}
where $\frac{2}{p}<\vartheta<\frac{2}{\beta}$ and $\frac{2}{p}<\gamma<1$. Similarly, we also have
\begin{eqnarray}\label{rate2all2}
	\!\!\!\!\!\!\!\!&&\frac{1}{N} \sum_{i=1}^N \sup_{t \in [0,T]} \mathbb{E}  \|X_t^i - X_t^{i,N}\|_{\mathbb{H}}^2 \nonumber\\
	\leq\!\!\!\!\!\!\!\!&&\frac{1}{N} \sum_{i=1}^N \sup_{t \in [0,T]} \mathbb{E}  \|X_{t \wedge \tilde {\tau}}^i - X_{t \wedge \tilde {\tau}}^{i,N}\|_{\mathbb{H}}^2  \nonumber\\
	\!\!\!\!\!\!\!\!&&+\frac{1}{N} \sum_{i=1}^N \Big\{ \sup_{t \in [0,T]} \mathbb{E}  \|X_t^i - X_t^{i,N}\|_{\mathbb{H}}^{p}  \Big\}^{\frac{2}{ p}} \Big\{ \mathbb{P}(\tilde {\tau} < T) \Big\}^{\frac{ p-2}{ p}}.
\end{eqnarray}

\vspace{2mm}
\noindent\textbf{Proof of (\ref{pocth3}).} We adopt schemes analogous to those used in Subsection \ref{Proofrate1} to estimate the four terms on the right-hand side of (\ref{rate2all}).

\vspace{1mm}
\noindent\textbf{Step 1.} In this step, our goal is to establish a bound for the first two terms on the right-hand side of (\ref{rate2all}).
Adopting the same arguments as in (\ref{R1all})-(\ref{R1I1-4}), we derive
\begin{eqnarray}\label{R2I1-4}
	\!\!\!\!\!\!\!\!&&\frac{1}{N} \sum_{i=1}^N \mathbb{E} \Big[ \sup_{t \in [0,T \wedge \tilde {\tau}]} \|X_t^i - X_t^{i,N}\|_{\mathbb{H}}^2 \Big]\nonumber\\
	\lesssim\!\!\!\!\!\!\!\!&& \frac{1}{N} \sum_{i=1}^N \mathbb{E} \int_0^{T \wedge \tilde {\tau}} \,_{\mathbb{V}^*}\langle \mathcal{A}(t, X_t^i, \bar{\mu}_t^N)- \mathcal{A}(t, X_t^{i,N}, \mu_t^N), X_t^i - X_t^{i,N} \rangle_{\mathbb{V}}dt\nonumber\\
	\!\!\!\!\!\!\!\!&&+ \frac{1}{N} \sum_{i=1}^N \mathbb{E} \int_0^{T \wedge \tilde {\tau}} \|\mathcal{B}(t, X_t^i, \bar{\mu}_t^N) - \mathcal{B}(t, X_t^{i,N}, \mu_t^N)\|_{L_2(U;\mathbb{H})}^2dt \nonumber\\
	\!\!\!\!\!\!\!\!&&+ \frac{1}{N} \sum_{i=1}^N \mathbb{E} \int_0^{T \wedge \tilde {\tau}} \,_{\mathbb{V}^*}\langle \mathcal{A}(t, X_t^i, \mathscr{L}_{X^i_t})-\mathcal{A}(t, X_t^i, \bar{\mu}_t^N), X_t^i - X_t^{i,N} \rangle_{\mathbb{V}}dt \nonumber\\
	\!\!\!\!\!\!\!\!&&+ \frac{1}{N} \sum_{i=1}^N \mathbb{E} \int_0^{T \wedge \tilde {\tau}} \|\mathcal{B}(t, X_t^i, \mathscr{L}_{X^i_t})-\mathcal{B}(t, X_t^i, \bar{\mu}_t^N)\|_{L_2(U;\mathbb{H})}^2dt \nonumber\\
	=:\!\!\!\!\!\!\!\!&& \sum_{m=1}^4\tilde {I}_m.
\end{eqnarray}
By virtue of condition $(\mathbf{A}''_5)$, we get
\begin{eqnarray}\label{tildeI1-2}
	\!\!\!\!\!\!\!\!&&\tilde {I}_1 + \tilde {I}_2\nonumber\\
	\leq\!\!\!\!\!\!\!\!&& -\frac{\delta_0}{N} \sum_{i=1}^N \mathbb{E} \int_0^{T \wedge \tilde {\tau}}\|X_s^i - X_s^{i,N}\|_{\mathbb{V}}^{\alpha}ds
	\nonumber\\
	\!\!\!\!\!\!\!\!&& +\frac{1}{N} \sum_{i=1}^N \mathbb{E} \int_0^{T \wedge \tilde {\tau}} ( \rho(0, \bar{\mu}_s^N)+\eta(0, \mu_s^N) ) \|X_s^i - X_s^{i,N}\|_{\mathbb{H}}^2 ds\nonumber\\
	\!\!\!\!\!\!\!\!&&+ \frac{1}{N} \sum_{j=1}^N \mathbb{E} \int_0^{T \wedge \tilde {\tau}} ( \rho(X_s^j,\bar{\mu}_s^N)+\eta(X_s^{j,N}, \mu_s^N) ) \mathcal{W}_{2,\mathbb{H}}(\bar{\mu}_s^N, \mu_s^N)^2 ds\nonumber\\
	\leq\!\!\!\!\!\!\!\!&& -\frac{\delta_0}{N} \sum_{i=1}^N \mathbb{E} \int_0^{T \wedge \tilde {\tau}}\|X_s^i - X_s^{i,N}\|_{\mathbb{V}}^{\alpha}ds
	\nonumber\\
	\!\!\!\!\!\!\!\!&&+\frac{1}{N} \sum_{i=1}^N \mathbb{E} \int_0^{T \wedge \tilde {\tau}} (\rho(0, \bar{\mu}_s^N)+\eta(0, \mu_s^N) ) \|X_s^i - X_s^{i,N}\|_{\mathbb{H}}^2 ds\nonumber\\
	\!\!\!\!\!\!\!\!&&+ \frac{1}{N} \sum_{i=1}^N \mathbb{E} \int_0^{T \wedge \tilde {\tau}} \Big[ \frac{1}{N} \sum_{j=1}^N (\rho(X_s^j, \bar{\mu}_s^N)+\eta(X_s^{j,N}, \mu_s^N) ) \Big] \|X_s^i - X_s^{i,N}\|_{\mathbb{H}}^2 ds.
\end{eqnarray}

On the other hand, following the same  procedure used to estimate the term $I_4$ in \textbf{Step 2} in the proof of (\ref{pocth1}),  the term $\tilde {I}_4$ here can be estimated as follows
\begin{equation}\label{tildeI4}
	\tilde {I}_4\lesssim N^{-1}.
\end{equation}
We claim that there exists a constant $C>0$ such that
\begin{equation}\label{tildeI3}
	\tilde {I}_3\leq \frac{\delta_0}{2N} \sum_{i=1}^N \mathbb{E} \int_0^{T \wedge \tilde {\tau}}\|X_s^i - X_s^{i,N}\|_{\mathbb{V}}^{\alpha}ds+C N^{-\frac{1}{\alpha-1}},
\end{equation}
whose proof will be given in \textbf{Step 2} below.
Then substituting (\ref{tildeI1-2})-(\ref{tildeI3}) back into (\ref{R2I1-4}), it follows that
\begin{eqnarray*}
	\!\!\!\!\!\!\!\!&&\frac{1}{N} \sum_{i=1}^N \mathbb{E} \Big[ \sup_{t \in [0,T \wedge \tilde {\tau}]} \|X_t^i - X_t^{i,N}\|_{\mathbb{H}}^2 \Big]+\frac{1}{N} \sum_{i=1}^N \mathbb{E} \int_0^{T \wedge \tilde {\tau}}\|X_s^i - X_s^{i,N}\|_{\mathbb{V}}^{\alpha}ds\nonumber\\
	\lesssim\!\!\!\!\!\!\!\!&& \frac{1}{N} \sum_{i=1}^N \mathbb{E} \int_0^{T \wedge \tilde {\tau}} ( 1 + \rho(0, \bar{\mu}_s^N)+\eta(0, \mu_s^N) ) \|X_s^i - X_s^{i,N}\|_{\mathbb{H}}^2 ds\nonumber\\
	\!\!\!\!\!\!\!\!&&+ \frac{1}{N} \sum_{i=1}^N \mathbb{E} \int_0^{T \wedge \tilde {\tau}} \Big[ \frac{1}{N} \sum_{j=1}^N ( 1 + \rho(X_s^j, \bar{\mu}_s^N)+\eta(X_s^{j,N}, \mu_s^N) ) \Big] \|X_s^i - X_s^{i,N}\|_{\mathbb{H}}^2 ds+N^{-\frac{1}{\alpha-1}}.
\end{eqnarray*}
By the definition of $\tilde {\tau}$ and using (\ref{XNeta}), it follows from the stochastic Gronwall's lemma (cf.~\cite[Lemma 5.3]{GZ}) that
\begin{equation}\label{es10}
	\frac{1}{N} \sum_{i=1}^N \mathbb{E} \Big[ \sup_{t \in [0,T \wedge \tilde {\tau}]} \|X_t^i - X_t^{i,N}\|_{\mathbb{H}}^2 \Big]+\frac{1}{N} \sum_{i=1}^N \mathbb{E} \int_0^{T \wedge \tilde {\tau}}\|X_s^i - X_s^{i,N}\|_{\mathbb{V}}^{\alpha}ds \lesssim_{\varepsilon,K} N^{-\frac{1}{\alpha-1}}.
\end{equation}

\noindent\textbf{Step 2.} In this step, we  prove (\ref{tildeI3}).
By H\"older's inequality, Young's inequality, and $(\mathbf{A}_6)$ we deduce  that
\begin{eqnarray*}
	\tilde{I}_3\!\!\!\!\!\!\!\!&&\leq \frac{1}{N} \sum_{i=1}^N \mathbb{E} \int_0^T \|X_t^i - X_t^{i,N}\|_{\mathbb{V}} \cdot\|  \int \tilde{\mathcal{A}}(t, X_t^i, y) \mathscr{L}_{X^i_t}(dy) - \frac{1}{N}\sum_{j=1}^N \tilde{\mathcal{A}}(t, X_t^i, X_t^j) \|_{\mathbb{V}^*} dt
	\\
	\!\!\!\!\!\!\!\!&&\leq \frac{\delta_0}{2N} \sum_{i=1}^N \mathbb{E} \int_0^T \|X_t^i - X_t^{i,N}\|_{\mathbb{V}}^\alpha  dt \\
	\!\!\!\!\!\!\!\!&&\quad+ \frac{C}{N} \sum_{i=1}^N \mathbb{E} \int_0^T \|  \int \tilde{\mathcal{A}}(t, X_t^i, y) \mathscr{L}_{X^i_t}(dy) - \frac{1}{N}\sum_{j=1}^N \tilde{\mathcal{A}}(t, X_t^i, X_t^j) \|_{\mathbb{V}^*}^{\frac{\alpha}{\alpha-1}} dt\\
	\!\!\!\!\!\!\!\!&&\leq \frac{\delta_0}{2N} \sum_{i=1}^N \mathbb{E} \int_0^T \|X_t^i - X_t^{i,N}\|_{\mathbb{V}}^\alpha  dt \\
	\!\!\!\!\!\!\!\!&&\quad+ C N^{-\frac{2\alpha-1}{\alpha-1}} \sum_{i=1}^N \mathbb{E} \int_0^T \| N \int \tilde{\mathcal{A}}(t, X_t^i, y) \mathscr{L}_{X^i_t}(dy) - \sum_{j=1}^N \tilde{\mathcal{A}}(t, X_t^i, X_t^j) \|_{\mathbb{V}^*}^{\frac{\alpha}{\alpha-1}} dt.
\end{eqnarray*}
Then the assertion (\ref{keyI3}) implies
(\ref{tildeI3}).

\vspace{1mm}
\noindent\textbf{Step 3.} To deal with the latter two terms on the right-hand side of (\ref{rate2all}), we note that (\ref{esq370}), (\ref{XNesti2}), and Lemmas \ref{nIPSesti2} and \ref{IPSesti2} imply that
\begin{eqnarray}\label{tildetaugeT}
	\!\!\!\!\!\!\!\!&&\frac{1}{N} \sum_{i=1}^N \Big\{ \mathbb{E} \Big[ \sup_{t \in [0,T]} \|X_t^i - X_t^{i,N}\|_{\mathbb{H}}^{\gamma p} \Big] \Big\}^{\frac{2}{\gamma p}} \Big\{ \mathbb{P}(\tilde {\tau} < T) \Big\}^{\frac{\gamma p-2}{\gamma p}}\nonumber\\
	\!\!\!\!\!\!\!\!&&\quad+\frac{1}{N} \sum_{i=1}^N \Big\{ \mathbb{E} \Big(\int_0^{T}\|X_s^i - X_s^{i,N}\|_{\mathbb{V}}^{\alpha}ds \Big)^{\frac{\vartheta p}{2} } \Big\}^{\frac{1}{\vartheta p}} \Big\{ \mathbb{P}(\tilde {\tau} < T) \Big\}^{\frac{\vartheta p-2}{\vartheta p}}\nonumber\\
	\lesssim\!\!\!\!\!\!\!\!&&\Big\{ \mathbb{P}(\tilde {\tau} < T) \Big\}^{\frac{\vartheta p-2}{\vartheta p}}\nonumber\\
	\lesssim\!\!\!\!\!\!\!\!&& \Big\{ \mathbb{P}( \tau^\varepsilon < T < \tau^K) + \mathbb{P}(\tau^K \leq T) \Big\}^{\frac{\vartheta p-2}{\vartheta p}}.
\end{eqnarray}
By Chebyshev's inequality, we obtain
\begin{eqnarray}\label{Ptauvarepsilon}
	\!\!\!\!\!\!\!\!&&\quad\mathbb{P}( \tau^\varepsilon < T < \tau^K)
	\nonumber\\
	\!\!\!\!\!\!\!\!&&\leq\mathbb{P}\bigg( \sup_{t \in [0,T]}\mathcal{W}_{\beta,\mathbb{H}} ( \bar{\mu}_{t \wedge \tilde {\tau}}^{N}, \mu_{t \wedge \tilde {\tau}}^{N})^\beta+\int_0^T \mathcal{W}_{\alpha,\mathbb{V}} ( \bar{\mu}_{t \wedge \tilde {\tau}}^{N}, \mu_{t \wedge \tilde {\tau}}^{N})^\alpha dt\geq \varepsilon  \bigg) \nonumber\\
	\!\!\!\!\!\!\!\!&&
	\leq \mathbb{P}\bigg(\sup_{t \in [0,T]} \frac{1}{N} \sum_{i=1}^N \|X_{t \wedge \tilde {\tau}}^{i,N} - X_{t \wedge \tilde {\tau}}^i\|_\mathbb{H}^\beta+\int_0^T\frac{1}{N} \sum_{i=1}^N \|X_{t \wedge \tilde {\tau}}^{i,N} - X_{t \wedge \tilde {\tau}}^i\|_\mathbb{V}^\alpha dt \geq \varepsilon \bigg) \nonumber\\
	\!\!\!\!\!\!\!\!&&\leq \frac{1}{\varepsilon} \cdot \frac{1}{N} \sum_{i=1}^N \bigg\{\mathbb{E} \Big[ \sup_{t \in [0,T]} \|X_{t \wedge \tilde {\tau}}^{i,N} - X_{t \wedge \tilde {\tau}}^i\|_\mathbb{H}^{\beta} \Big]+\mathbb{E}\int_0^T\|X_{t \wedge\tilde {\tau}}^{i,N} - X_{t \wedge \tilde {\tau}}^i\|_\mathbb{V}^\alpha dt\bigg\}.
\end{eqnarray}
If $\beta>2$, using \cite[Proposition 6.10]{FGB} we obtain for any $\beta<q<p$,
\begin{eqnarray}\label{es1000}
\!\!\!\!\!\!\!\!&&\mathbb{E} \Big[ \sup_{t \in [0,T]} \|X_{t \wedge \tilde {\tau}}^{i,N} - X_{t \wedge \tilde {\tau}}^i\|_\mathbb{H}^\beta \Big]\nonumber\\
	\leq\!\!\!\!\!\!\!\!&&
\Big\{
\mathbb{E} \Big[ \sup_{t \in [0,T]} \|X_{t \wedge \tilde {\tau}}^{i,N} - X_{t \wedge \tilde {\tau}}^i\|_\mathbb{H}^2 \Big]\Big\}^{\frac{\theta\beta}{2}}\cdot
\Big\{
\mathbb{E} \Big[ \sup_{t \in [0,T]} \|X_{t \wedge \tilde {\tau}}^{i,N} - X_{t \wedge \tilde {\tau}}^i\|_\mathbb{H}^q \Big]\Big\}^{\frac{(1-\theta)\beta}{q}},
\end{eqnarray}
where $\theta\in(0,1)$ is chosen such that
\begin{equation*}
\frac1\beta=\frac{\theta}{2}+\frac{1-\theta}{q}.
\end{equation*}
Therefore,
\begin{equation*}
\frac{\theta\beta}{2}
=
\frac{q-\beta}{q-2}.
\end{equation*}
Recalling (\ref {Ptauvarepsilon}) and using (\ref{es10}) as well as (\ref{es1000}), it follows that
\begin{equation*}
\mathbb{P}( \tau^\varepsilon < T < \tau^K)\lesssim_{\varepsilon,K}N^{-\frac{1}{\alpha-1}\cdot\frac{q-\beta}{q-2}}
\end{equation*}
If $\beta=2$, the estimate (\ref {Ptauvarepsilon}) can be simplified to
\begin{eqnarray*}
	\!\!\!\!\!\!\!\!&&\quad\mathbb{P}( \tau^\varepsilon < T < \tau^K)
	\nonumber\\
	\!\!\!\!\!\!\!\!&&\leq \frac{1}{\varepsilon} \cdot \frac{1}{N} \sum_{i=1}^N \Big\{\mathbb{E} \Big[ \sup_{t \in [0,T]} \|X_{t \wedge \tilde {\tau}}^{i,N} - X_{t \wedge \tilde {\tau}}^i\|_\mathbb{H}^2 \Big]+\mathbb{E}\int_0^T\|X_{t \wedge\tilde {\tau}}^{i,N} - X_{t \wedge \tilde {\tau}}^i\|_\mathbb{V}^\alpha dt\Big\}  \nonumber\\
	\!\!\!\!\!\!\!\!&&\lesssim_{\varepsilon,K}N^{-\frac{1}{\alpha-1}}.
\end{eqnarray*}
In summary, for any $\beta\geq2$, we can get
\begin{equation*}
\mathbb{P}( \tau^\varepsilon < T < \tau^K)
\lesssim_{\varepsilon,K}N^{-\frac{1}{\alpha-1}\cdot\frac{q-\beta}{q-2}},
\end{equation*}
which combined with (\ref{tauP2}) and  (\ref{tildetaugeT}) implies
\begin{eqnarray}\label{es9}
	\!\!\!\!\!\!\!\!&&\quad\frac{1}{N} \sum_{i=1}^N \Big\{ \mathbb{E} \Big[ \sup_{t \in [0,T]} \|X_t^i - X_t^{i,N}\|_{\mathbb{H}}^{\gamma p} \Big] \Big\}^{\frac{2}{\gamma p}} \Big\{ \mathbb{P}(\tilde {\tau} < T) \Big\}^{\frac{\gamma p-2}{\gamma p}}\nonumber\\
	\!\!\!\!\!\!\!\!&&\quad+\frac{1}{N} \sum_{i=1}^N \Big\{ \mathbb{E} \Big(\int_0^{T}\|X_s^i - X_s^{i,N}\|_{\mathbb{V}}^{\alpha}ds \Big)^{\frac{\vartheta p}{2} } \Big\}^{\frac{1}{\vartheta p}} \Big\{ \mathbb{P}(\tilde {\tau} < T) \Big\}^{\frac{\vartheta p-2}{\vartheta p}}\nonumber\\
	\!\!\!\!\!\!\!\!&&\lesssim_{\varepsilon,K}N^{-\frac{1}{\alpha-1}\cdot\frac{q-\beta}{q-2}\cdot {\frac{\vartheta p-2}{\vartheta p}}},
\end{eqnarray}
where we recall $\frac{2}{p}<\vartheta<\frac{2}{\beta}$ and $\beta<q<p$.

Consequently, it follows from (\ref{es10}) and (\ref{es9}) that
\begin{equation*}
	\frac{1}{N} \sum_{i=1}^N \mathbb{E} \Big[ \sup_{t \in [0,T]} \|X_t^i - X_t^{i,N}\|_{\mathbb{H}}^2 \Big]+\frac{1}{N} \sum_{i=1}^N \mathbb{E} \int_0^{T}\|X_s^i - X_s^{i,N}\|_{\mathbb{V}}^{\alpha}ds\lesssim N^{-\frac{1}{\alpha-1}\cdot\frac{q-\beta}{q-2}\cdot {\frac{\vartheta p-2}{\vartheta p}}}.
\end{equation*}

We complete the proof.
\hspace{\fill}$\Box$

\vspace{2mm}
\noindent\textbf{Proof of (\ref{pocth4}).}
We need to estimate the two terms on the right-hand side of (\ref{rate2all2}).

Firstly, we can get
\begin{eqnarray*}
	\!\!\!\!\!\!\!\!&&\frac{1}{N} \sum_{i=1}^N \mathbb{E} \|X_{t \wedge \tilde {\tau}}^i - X_{t \wedge \tilde {\tau}}^{i,N}\|_{\mathbb{H}}^2 \nonumber\\
	\lesssim\!\!\!\!\!\!\!\!&& \frac{1}{N} \sum_{i=1}^N \mathbb{E} \int_0^{t \wedge \tilde {\tau}} \,_{\mathbb{V}^*}\langle \mathcal{A}(s, X_s^i, \bar{\mu}_s^N) - \mathcal{A}(s, X_s^{i,N}, \mu_s^N), X_s^i - X_s^{i,N} \rangle_{\mathbb{V}}ds\nonumber\\
	\!\!\!\!\!\!\!\!&&+ \frac{1}{N} \sum_{i=1}^N \mathbb{E} \int_0^{t \wedge \tilde {\tau}} \|\mathcal{B}(s, X_s^i, \bar{\mu}_s^N) - \mathcal{B}(s, X_s^{i,N}, \mu_s^N)\|_{L_2(U;\mathbb{H})}^2ds \nonumber\\
	\!\!\!\!\!\!\!\!&&+ \frac{1}{N} \sum_{i=1}^N \mathbb{E} \int_0^{t \wedge \tilde {\tau}} \,_{\mathbb{V}^*}\langle \mathcal{A}(s, X_s^i, \mathscr{L}_{X^i_s})-\mathcal{A}(s, X_s^i, \bar{\mu}_s^N), X_s^i - X_s^{i,N} \rangle_{\mathbb{V}}ds \nonumber\\
	\!\!\!\!\!\!\!\!&&+ \frac{1}{N} \sum_{i=1}^N \mathbb{E} \int_0^{t \wedge \tilde {\tau}} \|\mathcal{B}(s, X_s^i, \mathscr{L}_{X^i_s})-\mathcal{B}(s, X_s^i, \bar{\mu}_s^N)\|_{L_2(U;\mathbb{H})}^2ds.
\end{eqnarray*}
Recalling (\ref{tildeI1-2})-(\ref{tildeI3}), it follows that
\begin{eqnarray*}
	\!\!\!\!\!\!\!\!&&\frac{1}{N} \sum_{i=1}^N \mathbb{E} \|X_{t \wedge \tilde {\tau}}^i - X_{t \wedge \tilde {\tau}}^{i,N}\|_{\mathbb{H}}^2 +\frac{1}{N} \sum_{i=1}^N \mathbb{E} \int_0^{t \wedge \tilde {\tau}}\|X_s^i - X_s^{i,N}\|_{\mathbb{V}}^{\alpha}ds\nonumber\\
	\lesssim\!\!\!\!\!\!\!\!&& \frac{1}{N} \sum_{i=1}^N \mathbb{E} \int_0^{t \wedge \tilde {\tau}} ( 1 + \rho(0, \bar{\mu}_s^N)+\eta(0, \mu_s^N) ) \|X_s^i - X_s^{i,N}\|_{\mathbb{H}}^2 ds\nonumber\\
	\!\!\!\!\!\!\!\!&&+ \frac{1}{N} \sum_{i=1}^N \mathbb{E} \int_0^{t \wedge \tilde {\tau}} \Big[ \frac{1}{N} \sum_{j=1}^N ( 1 + \rho(X_s^j, \bar{\mu}_s^N)+\eta(X_s^{j,N}, \mu_s^N) ) \Big] \|X_s^i - X_s^{i,N}\|_{\mathbb{H}}^2 ds+N^{-\frac{1}{\alpha-1}}.
\end{eqnarray*}
Then due to the definition of the stopping time $\tilde {\tau}$ and using the stochastic Gronwall's lemma (cf.~\cite[Lemma 5.3]{GZ}), we have
\begin{equation}\label{step1result22}
	\frac{1}{N} \sum_{i=1}^N \sup_{t \in [0,T]} \mathbb{E} \|X_{t \wedge \tilde {\tau}}^i - X_{t \wedge \tilde {\tau}}^{i,N}\|_{\mathbb{H}}^2+\frac{1}{N} \sum_{i=1}^N \mathbb{E} \int_0^{T \wedge \tilde {\tau}}\|X_s^i - X_s^{i,N}\|_{\mathbb{V}}^{\alpha}ds\lesssim_{\varepsilon,K} N^{-\frac{1}{\alpha-1}}.
\end{equation}

For the second term on the right-hand side of (\ref{rate2all2}), we note that
\begin{eqnarray*}
	\!\!\!\!\!\!\!\!&&\quad\frac{1}{N} \sum_{i=1}^N \Big\{ \sup_{t \in [0,T]} \mathbb{E}  \|X_t^i - X_t^{i,N}\|_{\mathbb{H}}^{p}  \Big\}^{\frac{2}{ p}} \Big\{ \mathbb{P}(\tilde {\tau} < T) \Big\}^{\frac{ p-2}{ p}}\nonumber\\
	\!\!\!\!\!\!\!\!&&\leq \frac{1}{N} \sum_{i=1}^N \Big\{ \sup_{t \in [0,T]} \mathbb{E}  \|X_t^i\|_{\mathbb{H}}^{p}+\sup_{t \in [0,T]} \mathbb{E} \|X_t^{i,N}\|_{\mathbb{H}}^{p}  \Big\}^{\frac{2}{ p}} \Big\{ \mathbb{P}(\tilde {\tau} < T) \Big\}^{\frac{ p-2}{ p}} \nonumber\\
	\!\!\!\!\!\!\!\!&&\lesssim \Big\{ \mathbb{P}(\tilde {\tau} < T) \Big\}^{\frac{ p-2}{ p}}\nonumber\\
\!\!\!\!\!\!\!\!&&\lesssim \Big\{ \mathbb{P}( \tau^\varepsilon_1 < T < \tau^K_2) + \mathbb{P}(\tau^K_2 \leq T) \Big\}^{\frac{ p-2}{ p}}.
\end{eqnarray*}
By Chebyshev's inequality, we can derive
\begin{eqnarray}\label{tauP12}
	\!\!\!\!\!\!\!\!&&\quad\mathbb{P}( \tau^\varepsilon_1 < T < \tau^K_2)\nonumber\\ \!\!\!\!\!\!\!\!&&\leq\mathbb{P}\bigg(\tilde{\tau}<T, \mathcal{W}_{\beta,\mathbb{H}} ( \bar{\mu}_{\tilde{\tau}}^{N}, \mu_{\tilde{\tau}}^{N})^\beta+ \int_0^{\tilde{\tau}} \mathcal{W}_{\alpha,\mathbb{V}} ( \bar{\mu}_s^{N}, \mu_s^{N})^\alpha ds\geq \varepsilon  \bigg) \nonumber\\
	\!\!\!\!\!\!\!\!&&\leq\mathbb{P}\bigg( \mathcal{W}_{\beta,\mathbb{H}} ( \bar{\mu}_{T \wedge \tilde{\tau}}^{N}, \mu_{T \wedge \tilde{\tau}}^{N})^\beta+ \int_0^{T \wedge\tilde{\tau}} \mathcal{W}_{\alpha,\mathbb{V}} ( \bar{\mu}_s^{N}, \mu_s^{N})^\alpha ds\geq \varepsilon  \bigg) \nonumber\\
	\!\!\!\!\!\!\!\!&&\leq \mathbb{P}\bigg(  \frac{1}{N} \sum_{i=1}^N \|X_{T \wedge \tilde{\tau}}^{i,N} - X_{T \wedge \tilde{\tau}}^i\|_\mathbb{H}^\beta
	+\frac{1}{N} \sum_{i=1}^N  \int_0^{T \wedge \tilde {\tau}}\|X_s^i - X_s^{i,N}\|_{\mathbb{V}}^{\alpha}ds
	\geq \varepsilon \bigg) \nonumber\\
	\!\!\!\!\!\!\!\!&&\leq \frac{1}{\varepsilon} \cdot \bigg(\frac{1}{N} \sum_{i=1}^N \mathbb{E} \|X_{T \wedge \tilde{\tau}}^{i,N} - X_{T \wedge \tilde{\tau}}^i\|_\mathbb{H}^{\beta}+\frac{1}{N} \sum_{i=1}^N \mathbb{E} \int_0^{T \wedge \tilde {\tau}}\|X_s^i - X_s^{i,N}\|_{\mathbb{V}}^{\alpha}ds \bigg).
\end{eqnarray}
Note that
when $\beta>2$,
\begin{equation*}
\mathbb{E} \|X_{T \wedge \tilde{\tau}}^{i,N} - X_{T \wedge \tilde{\tau}}^i\|_\mathbb{H}^\beta
	\leq
\Big\{
\mathbb{E}  \|X_{T \wedge \tilde {\tau}}^{i,N} - X_{T \wedge \tilde {\tau}}^i\|_\mathbb{H}^2 \Big\}^{\frac{\theta\beta}{2}}\cdot
\Big\{
\mathbb{E} \|X_{T \wedge \tilde {\tau}}^{i,N} - X_{T \wedge \tilde {\tau}}^i\|_\mathbb{H}^p \Big\}^{\frac{(1-\theta)\beta}{p}},
\end{equation*}
where $\theta\in(0,1)$ is chosen such that
\begin{equation*}
\frac{\theta\beta}{2}
=
\frac{p-\beta}{p-2}.
\end{equation*}
Hence, in view of (\ref{step1result22}) and (\ref{tauP12}) we have
\begin{equation*}
\mathbb{P}( \tau^\varepsilon_1 < T < \tau^K_2)
\lesssim_{\varepsilon,K}N^{-\frac{1}{\alpha-1}\cdot\frac{p-\beta}{p-2}},
\end{equation*}
which together with (\ref{tauP2}) indicates that
\begin{equation}\label{es11}
	\frac{1}{N} \sum_{i=1}^N \Big\{ \sup_{t \in [0,T]} \mathbb{E}  \|X_t^i - X_t^{i,N}\|_{\mathbb{H}}^{p}  \Big\}^{\frac{2}{ p}} \Big\{ \mathbb{P}(\tilde {\tau} < T) \Big\}^{\frac{ p-2}{ p}}\lesssim_{\varepsilon,K}N^{-\frac{1}{\alpha-1}\cdot{\frac{ p-\beta}{ p}}}.
\end{equation}
When $\beta=2$, the above inequality holds obviously.
Combining  (\ref{step1result22}) with (\ref{es11}),
the proof of (\ref{pocth4}) is completed. \hspace{\fill}$\Box$

\section{Appendix: Proof of Proposition \ref{propos1}}
	Recall
	$$
	Z_j=\varepsilon_j e_{K_j},
	$$
	where $\{(\varepsilon_j,K_j)\}_{j=1}^{\infty}$ are i.i.d. copies of
	$(\varepsilon,K)$. Thus
	$$
	\mathbb P(K_j=k)=p_k,
	\qquad
	\mathbb P(\varepsilon_j=1)
	=
	\mathbb P(\varepsilon_j=-1)
	=
	\frac12.$$

	For each $k\geq2$, the $k$-th coordinate of $Z_j$ is
	$$
	(Z_j)_k
	=
	\varepsilon_j\mathbf 1_{\{K_j=k\}}.
	$$
	Then,
	$$
	(\bar{Z}_N)_k
	=
	\frac1N\sum_{j=1}^N
	\varepsilon_j\mathbf 1_{\{K_j=k\}},
	$$
	and hence
	$$
	\bar{Z}_N
	=
	\sum_{k=2}^{\infty}
	\left(
	\frac1N\sum_{j=1}^N
	\varepsilon_j\mathbf 1_{\{K_j=k\}}
	\right)e_k.
	$$
	We know that
	$$
	\|\bar{Z}_N\|_{\ell^q}^q
	=
	\frac1{N^q}
	\sum_{k=2}^{\infty}
	\Big|
	\sum_{j=1}^N
	\varepsilon_j\mathbf 1_{\{K_j=k\}}
	\Big|^q.
	$$
	Since all summands are nonnegative, this yields
	\begin{equation}\label{r001}
		\mathbb E\|\bar{Z}_N\|_{\ell^q}^q
		=
		\frac1{N^q}
		\sum_{k=2}^{\infty}
		\mathbb E
		\Big|
		\sum_{j=1}^N
		\varepsilon_j\mathbf 1_{\{K_j=k\}}
		\Big|^q.
	\end{equation}

	\noindent\textbf{Proof of the upper bound.}
	Let
	$
	\mathcal K_N:=\sigma(K_1,\ldots,K_N).
	$
	We first condition on \(\mathcal K_N\). For every \(k\geq2\), define
	$$
	I_{N,k}
	:=
	\big\{j\in\{1,\ldots,N\}:K_j=k\big\}.
	$$
	Since $\{\varepsilon_j\}_{j=1}^N$ is independent of
	$\{K_j\}_{j=1}^N$, the variables \(\varepsilon_1,\ldots,\varepsilon_N\)
	remain independent Rademacher random variables under the conditional law given
	\(\mathcal K_N\).
	
	For every $k\geq2$,
	$$
	\sum_{j=1}^N
	\varepsilon_j\mathbf 1_{\{K_j=k\}}
	=
	\sum_{j\in I_{N,k}}\varepsilon_j.
	$$
	Note that \(1<q\leq2\), it follows from H\"{o}lder's inequality under the conditional
	probability measure that
	\begin{equation}\label{r00}
		\mathbb E\Big[
		\big|
		\sum_{j\in I_{N,k}}\varepsilon_j
		\big|^q
		\big|
		\mathcal K_N
		\Big]
		\leq
		\bigg(
		\mathbb E\Big[
		\big|
		\sum_{j\in I_{N,k}}\varepsilon_j
		\big|^2
		\big|
		\mathcal K_N
		\Big]
		\bigg)^{q/2}.
	\end{equation}	
	Since the summation set
	$I_{N,k}$ is $\mathcal K_N$-measurable,
	\begin{equation}\label{r01}
		\mathbb E\Big[
		\big|
		\sum_{j\in I_{N,k}}\varepsilon_j
		\big|^2
		\big|
		\mathcal K_N
		\Big]
		=
		\mathbb E\Big[
		\sum_{i,j\in I_{N,k}}\varepsilon_i\varepsilon_j
		\big|
		\mathcal K_N
		\Big] =
		\sum_{i,j\in I_{N,k}}
		\mathbb E\Big[
		\varepsilon_i\varepsilon_j
		\big|
		\mathcal K_N
		\Big]=
		\sum_{j\in I_{N,k}}1
		=
		|I_{N,k}|.
	\end{equation}
	Substituting (\ref{r01}) into $(\ref{r00})$, we obtain
	\begin{equation}\label{r02}
		\mathbb E\Big[
		\big|
		\sum_{j\in I_{N,k}}\varepsilon_j
		\big|^q
		\big|
		\mathcal K_N
		\Big]
		\leq
		|I_{N,k}|^{q/2}\leq |I_{N,k}|.
	\end{equation}

	By the definition of  the \(\ell^q\)-norm and  applying (\ref{r02}),
	we get
	$$
	\mathbb E\Big[
	\big\|
	\sum_{j=1}^NZ_j
	\big\|_{\ell^q}^q
	\big|
	\mathcal K_N
	\Big]
	=
	\sum_{k=2}^{\infty}
	\mathbb E\Big[
	\big|
	\sum_{j\in I_{N,k}}\varepsilon_j
	\big|^q
	\big|
	\mathcal K_N
	\Big]
	\leq
	\sum_{k=2}^{\infty}|I_{N,k}|.
	$$
	Since each index \(j\in\{1,\ldots,N\}\) belongs to exactly one set \(I_{N,k}\), we deduce that
	\[
	\sum_{k=2}^{\infty}|I_{N,k}|=N.
	\]
	Thus,
	$$
	\mathbb E\Big[
	\big\|
	\sum_{j=1}^NZ_j
	\big\|_{\ell^q}^q
	\big|
	\mathcal K_N
	\Big]
	\leq N.
	$$
	Finally, we conclude that
	$$
	\mathbb E\|\bar{Z}_N\|_{\ell^q}^q
	=
	N^{-q}
	\mathbb E	\big\|
	\sum_{j=1}^NZ_j
	\big\|_{\ell^q}^q
	\leq
	N^{1-q}=N^{-\frac{1}{\alpha-1}}.
	$$
	
	\noindent\textbf{Proof of the lower bound.} Fix \(k\geq2\), and define
	$$
	M_{N,k}
	:=
	\sum_{j=1}^N\mathbf 1_{\{K_j=k\}},
	\quad
	S_{N,k}
	:=
	\sum_{j=1}^N
	\varepsilon_j\mathbf 1_{\{K_j=k\}}.
	$$
	The variable \(M_{N,k}\) counts how many samples have their nonzero coordinate
	at the \(k\)-th location. Since the \(K_j\)'s are i.i.d. and
	\(\mathbb P(K_j=k)=p_k\), we can infer that
	$$
	M_{N,k}\sim\operatorname{B}(N,p_k).
	$$
	
	On the event \(\{M_{N,k}=1\}\), exactly one index \(j_0\) satisfies
	\(K_{j_0}=k\). Thus,
	$$
	S_{N,k}=\varepsilon_{j_0},
	\qquad
	|S_{N,k}|^q=1.
	$$
	Then we have
	\begin{equation}\label{r03}
		\mathbb E|S_{N,k}|^q
		\geq
		\mathbb E\big[
		|S_{N,k}|^q
		\mathbf 1_{\{M_{N,k}=1\}}
		\big]
		=
		\mathbb P(M_{N,k}=1)=
		Np_k(1-p_k)^{N-1}.
	\end{equation}
	Since we choose
	$
	p_k=\frac{C_0}{k(\log k)^2},
	$
	there exists $N_0\in\mathbb N$ such that, for every $N\geq N_0$ and
	every $k\geq N$,
	\begin{equation}\label{r04}	
		p_k\leq \frac1{2N}.
	\end{equation}

	Recall that	if \(0\leq p\leq1/(2N)\), Bernoulli's inequality gives
	\begin{equation}\label{r05}	
		(1-p)^{N-1}
		\geq
		1-(N-1)p
		\geq
		\frac12.
	\end{equation}	
	Thus, by (\ref{r04}) and (\ref{r05}), for \(N\geq N_0\) and \(k\geq N\), we have
	\begin{equation}\label{r06}	
		\mathbb P(M_{N,k}=1)
		\geq
		\frac12Np_k.
	\end{equation}	
	
	Combining (\ref{r001}), (\ref{r03}), and (\ref{r06}), we now derive the lower
	bound. Recall that
	$$
	S_{N,k}
	=
	\sum_{j=1}^N
	\varepsilon_j\mathbf 1_{\{K_j=k\}}.
	$$
	Then, (\ref{r001}) can be rewritten as
	$$
	\mathbb E\|\bar{Z}_N\|_{\ell^q}^q
	=
	\frac1{N^q}
	\sum_{k=2}^{\infty}
	\mathbb E|S_{N,k}|^q\geq
	\frac1{N^q}
	\sum_{k=N}^{\infty}\mathbb E|S_{N,k}|^q.
	$$

	Combining (\ref{r03}) and (\ref{r06}), we obtain
	\begin{equation}\label{r07}	
		\mathbb E\|\bar{Z}_N\|_{\ell^q}^q
		\geq			\frac1{N^q}			\sum_{k=N}^{\infty}\mathbb P(M_{N,k}=1)\geq
		\frac12N^{1-q}
		\sum_{k=N}^{\infty}p_k.
	\end{equation}	
	
	Finally, the function
	$$
	f(x):=\frac1{x(\log x)^2}
	$$
	is decreasing on \([2,\infty)\). Hence, by integral comparison, we deduce
	$$
	\sum_{k=N}^{\infty}\frac1{k(\log k)^2}
	\geq
	\int_N^{\infty}\frac{dx}{x(\log x)^2}
	=
	\frac1{\log N}.
	$$
	Thus,
	\begin{equation}\label{r08}	
		\sum_{k=N}^{\infty}p_k
		=
		C_0
		\sum_{k=N}^{\infty}
		\frac1{k(\log k)^2}
		\geq
		\frac{C_0}{\log N}.
	\end{equation}	
	Substituting (\ref{r08}) into (\ref{r07}), we get
	$$
	\mathbb E\|\bar{Z}_N\|_{\ell^q}^q
	\geq
	\frac{C_0}{2}
	N^{1-q}(\log N)^{-1}=\frac{C_0}{2}N^{-\frac{1}{\alpha-1}}(\log N)^{-1}.
	$$
	We complete the proof.

\vspace{3mm}


\noindent\textbf{Data availability} Data sharing is not applicable to this article as no datasets were generated or analysed during
the current study.

\vspace{3mm}

\noindent\textbf{Conflict of interest} On behalf of all authors, the corresponding author states that there is no conflict of interest.

\end{document}